\documentclass[a4paper,11pt,reqno]{amsart}
\usepackage{amsmath}
\usepackage[T1]{fontenc}
\usepackage{amssymb}
\usepackage{amsthm}
\usepackage{color}
\usepackage[lmargin=2.5 cm,rmargin=2.5 cm,tmargin=3.5cm,bmargin=2.5cm,
paper=a4paper]{geometry}

\newcommand{\Be}{\color{blue}}

\newcommand{\Ab}{\mathbf A}
\newcommand{\nb}{\mathbf n}
\newcommand{\App}{\Ab_{{\rm app},\nu}}
\newcommand{\Ap}{\Ab_{\rm app}}
\newcommand{\Fb}{\mathbf F}

\newcommand{\R}{\mathbb R}

\newcommand{\E}{\mathrm{E}_{\rm gs}(\kappa,H)}
\newcommand{\Er}{\mathrm{E}_{\rm gs,r}}
\newcommand{\er}{\mathfrak{e}_{\rm gs}}
\newcommand{\be}{\mathfrak{b}}
\newcommand{\n}{\mathfrak{n}}

\DeclareMathOperator{\curl}{curl}

\newtheorem{thm}{Theorem}[section]
\newtheorem{prop}[thm]{Proposition}

\newtheorem{corol}[thm]{Corollary}
\newtheorem{conjecture}[thm]{Conjecture}

\theoremstyle{remark}
\newtheorem{rem}[thm]{Remark}

\numberwithin{equation}{section}

\title[Ginzburg-Landau with vanishing magnetic field]{The Ginzburg-Landau functional with vanishing magnetic field}

\author{Bernard Helffer}
\author{Ayman Kachmar}
\address[B. Helffer]{Laboratoire de Math\'ematiques, Universit\'e de Paris-Sud 11, B\^at 425, 91405 Orsay, France and  Laboratoire Jean Leray (Universit\'e de Nantes)}
\email{bernard.helffer@math.u-psud.fr}
\address[A. Kachmar]{Department of Mathematics, Lebanese University, Hadat, Lebanon; School of Arts and Sciences, Lebanese International University, Beirut, Lebanon}
\email{ayman.kashmar@liu.edu.lb}
\date{\today}

\begin{document}

\begin{abstract}
We study the infimum of the  Ginzburg-Landau functional in the case
of a vanishing external magnetic field in a two dimensional simply
connected domain. We obtain an energy asymptotics which is valid when the
Ginzburg-Landau parameter is large and  the strength of the external
field is comparable with the  third critical field. Compared with
the known results  when the external magnetic field does not vanish, we show in this regime a
 concentration of the energy  near the zero set of the external magnetic field.
Our results complete former results obtained by K. Attar and X-B.
Pan--K-H.~Kwek.
\end{abstract}

\maketitle 

\section{Introduction}\label{hc2-sec:int}

In a two dimensional  bounded and simply connected domain
$\Omega$ with smooth boundary, the Ginzburg-Landau functional is
defined over configurations $(\psi,\Ab)\in H^1(\Omega;\mathbb
C)\times H^1(\Omega;\mathbb R^2)$ by,
\begin{equation}\label{eq-3D-GLf}
\mathcal E(\psi,\Ab)=\int_\Omega e_{\kappa,H}(\psi,\Ab)\,dx
\end{equation}
where
$$
e_{\kappa,H}(\psi,\Ab):= |(\nabla-i\kappa
H\Ab)\psi|^2-\kappa^2|\psi|^2+\frac{\kappa^2}2|\psi|^4+(\kappa
H)^2|\curl\Ab-B_0|^2\,.
$$

The modulus of the wave function function $\psi$ measures the
density of the superconducting electrons; the curl of the vector
field $\Ab$ measures the induced magnetic field; the parameter $H$
measures the intensity of the external magnetic field and the
parameter $\kappa$ ($\kappa >0$)  is a characteristic of the superconducting
material;  $dx$ is the Lebesgue measure $dx_1\,dx_2$. The function
$B_0$ represents the profile of the external magnetic field  in
$\Omega$ and is allowed to vanish non-degenerately on a smooth
curve. We suppose that $B_0$  is defined and $C^\infty$  in a
neighborhood of $\overline{\Omega}$ and  satisfies,
\begin{equation}\label{eq:B0}
|B_0|+|\nabla B_0|\geq c>0 \quad{\rm
in~}\overline{\Omega}\,,\end{equation} and that the set
\begin{equation}\label{eq:Gamma}
\Gamma=\{x\in\overline{\Omega}~:~B_0(x)=0\}\end{equation} consists
of a finite number of simple smooth curves.\\
We also assume that:
\begin{equation}\label{eq:ass:Gam}
\Gamma\cap\partial\Omega~{\rm~
is~a~finite~set}\,.
\end{equation}
The ground state energy of the functional is,
\begin{equation}\label{eq-gse}
\E=\inf\{\mathcal E(\psi,\Ab)~:~(\psi,\Ab)\in H^1(\Omega;\mathbb
C)\times H^1(\Omega;\mathbb R^2)\}\,.
\end{equation}
We focus on the regime where $H$ satisfies
\begin{equation}\label{eq:ass:H}
H= \sigma \kappa^2\,,\quad \sigma \in\,(0,\infty) \,.\end{equation}

Our results allow for $\sigma$ to be a function of $\kappa$
satisfying $\sigma \gg \kappa^{-1}$.  Earlier results corresponding
to vanishing magnetic fields have been  obtained recently in
\cite{Att, Att2}. The assumption on the  strength of the magnetic
field was  $H \leq C \kappa$, where $C$ is a constant. In the regime
of large $\kappa$, K. Attar has obtained in \cite{Att, Att2}
parallel results to those known for the constant magnetic field in
\cite{SS02}. However, it is proved in \cite{Att} that if
\begin{equation}\label{eq-as:attar}
H= b \kappa\,,\end{equation} and $b$ is a constant, then when $b $
is large enough, the energy and the superconducting density are
concentrated near the set $\Gamma$ with a length scale $\frac1 b$.
Essentially, that is a consequence of the following asymptotics of
the energy ($\kappa\to \infty$),
\begin{equation}\label{eq-2D-thm}
 \E= \kappa^{2}\int_{\Omega}  g \left(\frac{H}{\kappa}|B_{0}(x)|\right)\,dx+\textit{o}\left(\kappa H\left|\ln\frac{\kappa}{H}\right|+1\right)\,,
\end{equation}
which is valid under the relaxed assumption that
\begin{equation}\label{hypatt}
\Lambda_1\kappa^{1/3}
\leq H\leq \Lambda_2\kappa\,,
\end{equation}
$\Lambda_1$ and $\Lambda_2$ being positive constants. In particular,
the assumption in \eqref{hypatt}  covers  the situation in
\eqref{eq-as:attar}. The function $g(b)$ appears in the analysis of
the two and three dimensional Ginzburg-Landau functional with
constant magnetic field, \cite{SS02, FK-cpde}. It is associated with
some effective model energy. The function $g(b)$ will play a central
role in this paper and its definition will be recalled later in this
text (see \eqref{eq:g}).

One purpose of this paper is  to give a precise description of the
aforementioned concentration of the order parameter and the energy
when $\sigma\gg 1$, thereby leading to the assumption in
\eqref{eq:ass:H}.

The leading order term of the ground state energy in \eqref{eq-gse}
is expressed via the quantity $E(\cdot)$ introduced in
Theorem~\ref{thm-FK} below. The function $(0,\infty)\ni L\mapsto
E(L)$ is a continuous function satisfying the following properties:
\begin{itemize}
\item $E(L)$ is defined via a reduced Ginzburg-Landau energy in the
strip  (this energy is introduced in \eqref{eq-gs-er''}).
\item $E(L)=0$ iff $L\geq \lambda_0^{-3/2}\,$, where $\lambda_0$ is a
universal constant defined as the bottom of the spectrum of a
Montgomery operator, see \eqref{eq-lambda0}.
\item As $L\to 0_+$, the expected asymptotic behavior of $E(L)$ is
like $L^{-4/3}\,$.
\end{itemize}

 Through this text, we use the following notation. If $A$ and
$B$ are two positive quantities, then
\begin{itemize}
\item $A\ll B$ means  $A=\delta(\kappa)B$ and
$\delta(\kappa)\to0$ as $\kappa\to\infty$\,;
\item $A\lesssim B$ means  $A\leq CB$  and $C>0$ is a
constant independent of $\kappa$\,;
\item  $A\gg B$ means $B\ll A$, and $A\gtrsim B$ means $B\lesssim A$\,;
\item $A\approx B$ means $c_1B\leq A\leq c_2 B$, $c_1>0$ and $c_2>0$ are constants independent of $\kappa$\,.
\end{itemize}

The main result in this paper is:
\begin{thm}\label{thm:HK}
Suppose that the function $B_0$ satisfies Assumptions
\eqref{eq:B0} and \eqref{eq:Gamma}. Let $b:\R_+\to\R_+$ be a
function satisfying
\begin{equation}\label{hypb}
\lim_{\kappa\to \infty}b(\kappa)= \infty\quad{\rm and}\quad
 \limsup_{\kappa\to \infty} \kappa^{-1}b(\kappa)<\infty \,.
 \end{equation}
%
Suppose that  $$H =  b(\kappa)\kappa\,.$$
 As $\kappa\to\infty$,
the ground state energy in \eqref{eq-gse} satisfies:
\begin{enumerate}
\item If $b(\kappa) \gg\kappa^{1/2}$, then
\begin{equation}\label{HK1}
\E=\kappa\left(\int_{\Gamma}\left(|\nabla
B_0(x)|\frac{H}{\kappa^2}\right)^{1/3}\,E\left(|\nabla
B_0(x)|\frac{H}{\kappa^2}\right)\,ds(x)\right)+\frac{\kappa^3}{H}o(1)\,,\end{equation}
where  $ds$ is the arc-length measure in $\Gamma$.
\item If $b(\kappa) \lesssim \kappa^{1/2}$, then,
\begin{equation}\label{HK1*}
\E= \kappa^2\int_{\Omega}
g\left(\frac{H}{\kappa}|B_0(x)|\right)\,dx+\frac{\kappa^3}{H}o(1)
\,.
\end{equation}
\end{enumerate}
\end{thm}
\begin{rem}
There is a small gap between the two regimes considered above.  Hence it would be interesting to show that the two asymptotics match
 in this intermediate zone.
\end{rem}
\begin{rem}
As we shall see in Section~2, Pan and Kwek \cite{PK} prove that if
$H$ is larger than a critical value $H_{c_3}(\kappa)\,$, then the
{minimizers} of the  functional in \eqref{eq-3D-GLf} are trivial and
the ground state energy is $\E=0\,$. Furthermore, the value of
$H_{c_3}(\kappa)$ as given in \cite{PK} admits, as $\kappa
\rightarrow \infty\,$,  the following asymptotics
$$H_{c_3}(\kappa)\sim c_0\kappa^2\,,$$
where $c_0$ is a universal explicit constant. As such, the
assumption on the magnetic field in Theorem~\ref{thm:HK} is
significant when $b(\kappa)\kappa\leq H\leq M\kappa^2$ and
$M\in(0,c_0]$ is a constant. Note also that our theorem gives a
bridge between the situations studied by Attar in \cite{Att, Att2}
and Pan-Kwek in \cite{PK}.
\end{rem}
\begin{rem}
As long as the intensity of the external magnetic field satisfies
$\kappa\ll H\leq M\kappa^2$ and $M\in(0,c_0)$, the remainder term
appearing in Theorem~\ref{thm:HK} is of lower order compared with
the principal term. The function $g(b)$ is bounded and vanishes when
$b\geq 1$. Accordingly,
$$\int_\Omega g\left(\frac{H}{\kappa}|B_0(x)|\right)\,dx=\int_{\{|B_0(x)|<\frac\kappa H\}}
g\left(\frac{H}{\kappa}|B_0(x)|\right)\,dx\approx \frac\kappa H\,.$$
We shall see in Theorem~\ref{thm:Lto0} that,
$$\left(|\nabla
B_0(x)|\frac{H}{\kappa^2}\right)^{1/3}\,E\left(|\nabla
B_0(x)|\frac{H}{\kappa^2}\right)\approx \frac{\kappa^2}{H}\,.$$
\end{rem}

Along the proof of Theorem~\ref{thm:HK}, we obtain:

\begin{thm}\label{thm:HK-loc}
Suppose that the function $B_0$ satisfies Assumptions
\eqref{eq:B0} and \eqref{eq:Gamma}. Let $b:\R_+\to\R_+$ be a
function satisfying \eqref{hypb}. Suppose that
  $$H = b(\kappa)\kappa $$ and that $(\psi,\Ab)$ is a minimizer of
the functional in \eqref{eq-3D-GLf}. \\
As $\kappa\to \infty\,$, there holds:\\
\begin{enumerate}
\item{\rm {\bf Estimate of the magnetic energy}.}
$$\kappa^2H^2\int_{\Omega}|\curl\Ab-B_0|^2\,dx =
\frac{\kappa^3}{H}\,o(1)\,.$$
\item{\rm {\bf Estimate of the local energy}.}\\
Let $D\subset \Omega$ be an open set with a smooth boundary such
that $\partial D\cap\Gamma$ is a finite set.
\begin{enumerate}
\item If $b(\kappa) \gg \kappa^{1/2}$, then
\begin{align}\mathcal
E_0(\psi,\Ab;D)&:=\int_D\left(|(\nabla-i\kappa
H\Ab)\psi|^2-\kappa^2|\psi|^2+\frac{\kappa^2}{2}|\psi|^4\right)\,dx\label{eq:E0}\\
&=\kappa\left(\int_{D\cap\Gamma}\left(|\nabla
B_0(x)|\frac{H}{\kappa^2}\right)^{1/3}\,E\left(|\nabla
B_0(x)|\frac{H}{\kappa^2}\right)\,ds(x)\right)+\frac{\kappa^3}{H}\,o(1)\,.
\nonumber\end{align}
\item If $1\ll b(\kappa) \lesssim\kappa^{1/2}$, then
$$
\mathcal E_0(\psi,\Ab;D)=\kappa^2\int_Dg\left(\frac{H}{\kappa}|B_0(x)|\right)\,dx
+\frac{\kappa^3}{H}\,o(1)\,.
$$
\end{enumerate}
\item{\rm {\bf Concentration of the order parameter}. }\\
Let $D\subset \Omega$ be an open set with a smooth boundary such
that $\partial D\cap\Gamma$ is a finite set.
\begin{enumerate}
\item If $ b(\kappa) \gg\kappa^{1/2}$, then
$$\int_D|\psi(x)|^4\,dx=-\frac2{\kappa}\left(\int_{D\cap\Gamma}\left(|\nabla
B_0(x)|\frac{H}{\kappa^2}\right)^{1/3}\,E\left(|\nabla
B_0(x)|\frac{H}{\kappa^2}\right)\,ds(x)\right)+\frac{\kappa}{H}\,o(1)\,.$$
\item If $1\ll b(\kappa)  \lesssim \kappa^{1/2}$, then
$$
\int_D|\psi(x)|^4\,dx=-\int_Dg\left(\frac{H}{\kappa}|B_0(x)|\right)\,dx
+\frac{\kappa}{H}\,o(1)\,.
$$
\end{enumerate}
\end{enumerate}
\end{thm}
\begin{rem} It could be interesting to improve the second term when
 $D\cap \Gamma
=\emptyset$.  In Theorem~\ref{thm:exdec}, we will prove that the
$L^2$-norm of the order parameter $\psi$ is concentrated near the
set $\Gamma$, and that $\psi$   exponentially decays as $\kappa
\rightarrow \infty$, away from $\Gamma$.
\end{rem}
\begin{rem}
In Theorem~\ref{thm:HK-loc}, the function $o(1)$  is  bounded independently of
the choice of the minimizer $(\psi,\Ab)$. In the first assertion, one can find a bound of
$o(1)$ which depends only  on the domain $\Omega$ and the function $B_0$, while
in the second and third assertions, the bound depends additionally on
the domain $D$.
\end{rem}

\section{Critical fields}\label{sec:CF}
The identification of the critical {\it magnetic} fields is an
important question regarding the functional in  \eqref{eq-3D-GLf}.
This question has an early appearance in physics (see e.g.
\cite{dG}) and was the subject of a vast mathematical literature in
the  past two decades. The two monographs \cite{FH-b, SaSe}
contain an extensive review of many important results. In this
section, we give a brief {\it informal} description of  critical
fields and highlight the importance of the case of a vanishing
applied magnetic field.

\subsection{Reminder: The constant field case}
When the magnetic field $B_0$ is a (non-zero) constant,  three critical values are assigned  to the magnetic field $H$, namely
$H_{c_1}$, $H_{c_2}$ and $H_{c_3}$. The behavior of minimizers (and
critical points) of the functional in \eqref{eq-3D-GLf} changes as
the parameter $H$ (i.e. magnetic field) crosses the values
$H_{c_1}$, $H_{c_2}$ and $H_{c_3}$. The identification of these
critical values is not easy, especially the value $H_{c_2}$ which is
still {\it loosely} defined.


Let us recall that a critical point $(\psi,\Ab)$ of the functional
in \eqref{eq-3D-GLf} is said to be {\it normal} if $\psi=0$
everywhere. The critical field $H_{c_3}(\kappa)$ is then defined as
the value at which the transition from {\it normal} to {\it
non-normal} critical points takes place.

The identification of the critical value $H_{c_3}(\kappa)$ is
strongly related to the spectral analysis of the magnetic
Schr\"odinger operator with a constant magnetic field and Neumann
boundary condition. Suppose that $\Omega\subset\R^2$ is  connected,
open, has a  smooth boundary and the boundary consists of a finite
number of connected components, $A_0$ a vector field satisfying
${\rm curl} A_0 = B_0$, the function $B_0$ is constant and positive,
and $\lambda (H \kappa A_0)$ the lowest eigenvalue of the magnetic
Schr\"odinger operator
\begin{equation}\label{eq:N}
-\Delta_{\kappa H A_0}=-(\nabla-i\kappa H
A_0)^2\quad{\rm in~}L^2(\Omega)\,,\end{equation} with Neumann
boundary conditions. It is proved that the function $t\mapsto
\lambda (tA_0)$ is monotonic for large values of $t$, see
\cite{FH-b} and the references therein.
  Grosso modo\footnote{Initially
 (see \cite{LuPa99}), one should start by defining four critical values according to locally or globally minimizing solutions.
 Following the terminology of \cite{FH-b}, these are upper or lower,  global or local fields.
 The four fields are proved  to be equal in
\cite{FH-b}.} the critical field $H_{c_3}$ is the unique solution
of the
 equation,
 \begin{equation}\label{a1}
 \lambda (H_{c_3}(\kappa)\kappa A_0) = \kappa^2\,.
 \end{equation}
 In this case, it was shown by Lu-Pan \cite{LuPa99}
 that,
 \begin{equation}\label{a2}
 \lambda (H \kappa A_0) \sim  (H \kappa) B_0 \Theta_0 \,,\, \mbox{ when } H \kappa \gg 1\,.
 \end{equation}
 Further improvements of \eqref{a2} are available, see \cite{FH-b} for the state of the art in 2009 and references therein.

 As a consequence of \eqref{a1} and \eqref{a2}, we get for $\kappa$ sufficiently large,
 \begin{equation}\label{a3}
 H_{c_3}(\kappa)\sim \kappa /(\Theta_0 B_0)\,.
 \end{equation}
 The second critical field $H_{c_2}(\kappa)$ is usually defined as
 follows
 \begin{equation}\label{a4}
 H_{c_2}(\kappa)= \kappa/B_0\,.
 \end{equation}
 Notice that this definition of $H_{c_2}$ is asymptotically matching
 with the following definition\footnote{Assuming the monotonicity of $t\mapsto \lambda^D (t A_0)$ for $t$ large.},
\begin{equation}\label{a4'}
\lambda^D (H_{c_2}(\kappa)\kappa A_0) = \kappa^2\,,\end{equation} where
$\lambda^D$ is the first eigenvalue of the operator in \eqref{eq:N},
but with Dirichlet boundary condition.

Near $H_{c_2}(\kappa)$, a transition takes place between  surface
and bulk  superconductivity. At the level of the energy, this
transition is described in \cite{FK-am}.

 We recall that $\Theta_0 <1\,$.  Hence, as expected, $H_{c_2}(\kappa)< H_{c_3}(\kappa)$ for $\kappa$ sufficiently large.
 For the identification of the critical field $H_{c_1}(\kappa)$, we refer to Sandier-Serfaty \cite{SaSe}. A natural question is to extend
 this discussion in the variable magnetic field case (i.e. $B_0$ is a non-constant function).

\subsection{The case of a non vanishing exterior magnetic field}
 Here we discuss the situation where the magnetic field $B_0$ is a non-constant function such that $B_0(x)\neq 0$ everywhere in $\overline \Omega\,$.
 In this case, it is proved by Lu-Pan \cite[Theorem~1]{LuPajmp}
 that,
  \begin{equation}\label{a5}
 \lambda (H \kappa A_0) \sim  (H \kappa)\min \left( \inf_{x\in \overline\Omega}  |B_0(x)| , \Theta_0 \inf_{x\in \partial  \Omega} |B_0(x)| \right) \,,
 \end{equation}
 as $H\kappa\to \infty$.
 Basically, this leads to consider two cases as follows.\\
 \paragraph{\bf Surface superconductivity}  First, we assume that
 \begin{equation}
  \inf_{x\in \overline\Omega}  |B_0(x)|  > \Theta_0 \inf_{x\in \partial  \Omega} |B_0(x)| \,.
 \end{equation}
 In this case, the phenomenon of surface superconductivity observed in the constant magnetic field case is preserved.
 More precisely, the superconductivity starts to appear at the points where $(B_0)_{/\partial \Omega}$ is minimal. The critical value  $H_{c_3}(\kappa)$ is still
 defined by \eqref{a1}. If the minima of $(B_0)_{/\partial\Omega}$ are non-degenerate, then  the monotonicity of the
 eigenvalue $\lambda(t\,A_0)$ for large values of $t$ is established
 in \cite[Section~6]{Ra}. Consequently, we get when $\kappa$ is sufficiently large,
  \begin{equation}\label{a7}
 H_{c_3}(\kappa)\sim \frac{ \kappa}{ \Theta_0 \inf_{x\in \partial  \Omega} |B_0(x)| }
 \,.
 \end{equation}
 Tentatively, one could think to define  $H_{c_2}(\kappa)$ either by
 \begin{equation}\label{a8}
 H_{c_2}(\kappa) =  \frac{ \kappa}{  \inf_{x\in  \overline\Omega} |B_0(x)| }
 \,,
 \end{equation}
or by
\begin{equation}\label{a9}
\lambda^D(H_{c_2}(\kappa)\kappa A_0) = \kappa^2\,,
\end{equation}
where $\lambda^D$ is the first eigenvalue of the operator in
\eqref{eq:N} with Dirichlet boundary condition. Notice that both
formulas agree with their analogues in the constant magnetic field
case (see \eqref{a4} and \eqref{a4'}). Also, the values of
$H_{c_2}(\kappa)$ given in \eqref{a8} or \eqref{a9}  asymptotically
match as $\kappa\to \infty\,$.

In order that the definition of $H_{c_2}(\kappa)$ in \eqref{a9} be
consistent, one should prove monotonicity of $t\mapsto \lambda^D (t A_0)$ for large of values of $t$. This will ensure that \eqref{a9} assigns a unique value of $H_{c_2}(\kappa)$.
However, such a monotonicity is not proved yet.
The definition in \eqref{a8}  was proposed in \cite{FH-b}.

\paragraph{\bf Interior onset of superconductivity}
Here we assume that
\begin{equation}\label{eq:ass-int}
  \inf_{x\in \overline\Omega}  |B_0(x)|  <  \Theta_0 \inf_{x\in \partial  \Omega} |B_0(x)| \,.
 \end{equation}
 In this case, the onset of superconductivity near the surface of the domain disappears. If one decreases gradually the intensity of the magnetic field $H$ from $\infty$,
 then superconductivity will start to appear near the minima of  the function $|B_0|$, i.e. inside a compact subset of  $\Omega$.

 In this situation, we have not to distinguish between the critical fields $H_{c_2}(\kappa)$ and $H_{c_3}(\kappa)$, since surface superconductivity is absent here.
 Consequently,
 we expect that,
 \begin{equation}\label{a10}
 H_{c_2}(\kappa)=H_{c_3}(\kappa)\sim \frac{ \kappa}{  \inf_{x\in  \overline\Omega} |B_0(x)| }\,.
\end{equation}
A partial justification of this fact can be done using the
linearized Ginzburg-Landau equation near a normal solution.
Actually, we may also define $H_{c_3}(\kappa)$ and $H_{c_2}(\kappa)$
as the values verifying \eqref{a1}  and \eqref{a4'}. It should be
noticed here that the vector field $A_0$ satisfies $\curl{A_0}=B_0$
and $B_0$ cannot  be constant. Under the assumption \eqref{eq:ass-int},
the known spectral asymptotics (which are actually the same in this
case) of the Dirichlet and Neumann eigenvalues will lead us to the
asymptotics given in the righthand side of  \eqref{a10}.  Under the
additional assumption that $ \inf_{x\in  \overline\Omega} |B_0(x)|$
is attained at a unique minimum in $\Omega$ and that this minimum is
non degenerate, a complete asymptotics of $\lambda^N(tA_0)$ can be
given (see Helffer-Mohamed \cite{HelMo1}, Helffer-Kordyukov
\cite{HelKo,HelKo2}, Raymond-Vu Ngoc \cite{RaSV} ) and the
monotonicity/strong diamagnetism property  holds for large values of $t$
(see Chapter 3 in \cite{FH-b}). Hence the definition of
$H_{c_3}(\kappa)$ is clear in this case.

Besides the aforementioned linearized calculations, the results of
\cite{Att} are useful to justify the equality of the critical fields
$H_{c_2}(\kappa)$ and $H_{c_3}(\kappa)$ as well as their definition
in \eqref{a10}.

{\bf First,} we observe that if  $C$ is a positive constant such
that  $C < \displaystyle\frac{1}{\inf _{x\in
\overline\Omega}|B_0(x)|}$, and if $H\leq C \kappa\,$, then the open
set $D=\{x\in\Omega~:~ |B_0(x)| < \frac 1C\}\not=\emptyset$ is
non-empty. Now,  Theorem~1.4 of \cite{Att} asserts that,
$$
\exists~\kappa_0>0\,,~\exists~\epsilon_D>0,\quad \int_D |\psi(x) |^4
dx \geq  \epsilon_D >0\,,$$ for any $\kappa\geq\kappa_0$ and any
{\it minimizer} $(\psi,\Ab)$ of the functional in \eqref{eq-3D-GLf}.\\
 Consequently, a minimizer can
not be a {\it normal solution}.

{\bf Now,} suppose that the constant $C$ satisfies  $$C >
\displaystyle\frac{1}{\inf _{x\in \overline\Omega}|B_0(x)|}\,.$$
 If
$H\geq C \kappa$, then  Theorem~1.4 of \cite{Att}  asserts that any
critical point $(\psi,\Ab)$ of \eqref{eq-3D-GLf} satisfies,
$$
\lim_{\kappa\to \infty}\int_\Omega |\psi(x)|^4 dx=0\,,$$ hence,
loosely speaking, critical points are {\it nearly} normal solutions.
However, repeating the proof given in \cite[Section~10.4]{FH-b} and
using the asymptotics of the first eigenvalue in \eqref{a5}, one can
get that such critical points are indeed normal solutions.

This discussion shows that the value  appearing in the right hand side of  \eqref{a10}
is indeed critical.

\subsection{The case of a vanishing exterior magnetic field}

We now discuss the case when $B_0$ vanishes along a curve, first
 considered in \cite{PK} and then in \cite{Att}. We assume
that
\begin{equation}\label{b1}
|B_0| + |\nabla B_0| \neq 0 \mbox{ in } \overline{\Omega}\,,
\end{equation}
which ensures that $B_0$ vanishes {\it non-degenerately}.

 At each
point of $B_0^{-1}(0) \cap \Omega$, Pan-Kwek \cite{PK} introduce a
toy model (a Montgomery operator parameterized by the intensity of
the magnetic field at this point) whose ground state energy, denoted
by $\lambda_0$,  captures the `local' ground state energy of the
Schr\"odinger operator in \eqref{eq:N}.

Similarly, at every point $x$ of $B_0^{-1}(0) \cap
\partial \Omega$, a toy operator is defined on $\mathbb R^2_+$
parameterized (up to unitary equivalence) by the intensity of
$B_0(x)$ and the angle $\theta(x)\in\,[0,\pi/2)$ between the unit
normal of the boundary and $\nabla B_0(x)$. The ground state energy
of this toy operator is denoted by  $\lambda_0(\R_+,\theta(x))$.

The leading order behavior of the ground state energy of the
operator in \eqref{eq:N} is now described as follows \cite{PK},
\begin{equation}\label{a11}
\lambda(H\kappa A_0)\sim (H\kappa)^{2/3}\,\alpha_1^{2/3}\,,
\end{equation}
as $H\kappa\to \infty\,$.\\
 Here
\begin{equation}\label{a12}
\alpha_1=\min\left(\lambda_0^{3/2}\,\min_{x\in\Gamma_{\rm blk}}|\nabla B_0(x)|\,,\,\min_{x\in\Gamma_{\rm bnd}}\lambda_0(\R_+,\theta(x))|\nabla B_0(x)|\right)\,,
\end{equation}
\begin{equation}
\Gamma_{\rm blk}=\{x\in\Omega~:~B_0(x)=0\}
\end{equation}
 and
 \begin{equation} \Gamma_{\rm
bnd}=\{x\in\partial\Omega~:~B_0(x)=0\}\,.
\end{equation}

The critical value $H_{c_3}(\kappa)$ could tentatively  be defined
as the solution of the equation in \eqref{a1}. However, when
$B_0=\curl A_0$ vanishes, monotonicity of $t\mapsto\lambda(tA_0)$ is
not  a direct application of Chapter 3 in \cite{FH-b} (see the
discussion below). Nevertheless, for the various
definitions of $H_{c_3} (\kappa)$ proposed in \cite{PK}, one always
get  that, for large values of $\kappa$,
\begin{equation}\label{b2}
H_{c_3}(\kappa)\sim\frac{\kappa^2}{\alpha_1}\,.
\end{equation}
Surface superconductivity is absent if
$$
\lambda_0^{3/2}\,\min_{x\in\Gamma_{\rm blk}}|\nabla B_0(x)|<\min_{x\in\Gamma_{\rm bnd}}\lambda_0(\R_+,\theta(x))|\nabla B_0(x)|\,,
$$
and in this case, we do not distinguish between $H_{c_2}$ and
$H_{c_3}$. However, if
\begin{equation}\label{eq:2ndcase}
\lambda_0^{3/2}\,\min_{x\in\Gamma_{\rm blk}}|\nabla B_0(x)|>\min_{x\in\Gamma_{\rm bnd}}\lambda_0(\R_+,\theta(x))|\nabla B_0(x)|\,,
\end{equation}
the phenomenon of surface superconductivity is observed in
decreasing magnetic fields. Superconductivity will nucleate near the
minima of the function
$$\Gamma_{\rm bnd}\ni x\mapsto
\lambda_0(\R_+,\theta(x))|\nabla B_0(x)|\,.$$
 In this case, a natural
definition of $H_{c_2}(\kappa)$ can be,
\begin{equation}\label{b3}
H_{c_2}(\kappa) { :=} \frac{\kappa^2}{\alpha_2}\,,
\end{equation}
for large values of $\kappa$. Here
$$\alpha_2=\lambda_0^{3/2}\min_{x\in\Gamma_{\rm blk}}|\nabla B_0(x)|\,.$$
 Recently, we learn from N. Raymond that his student  J-P.  Miqueu \cite{Miq} is working on the case when \eqref{eq:2ndcase} is satisfied.
 The project in  \cite{Miq} is to give a complete  asymptotics of the first eigenvalue under  the additional  assumption that $\Gamma$ touches the boundary transversally.    If successful, then
a modification of the proof\,\footnote{Personal communication with
S. Fournais.} given in \cite{FoHe07} will yield
 the monotonicity of  $\lambda^N(tA_0)$
for large values of $t$.

Clearly, the condition in \eqref{hypatt} is violated when the intensity of the magnetic field $H$ is
comparable with the critical value $H_{c_3}(\kappa)\approx \kappa^2$, thereby
preventing the  application of the results of Attar \cite{Att}.

\section{The limiting problem}\label{hc2-sec-lbp}

\subsection{The Montgomery operator}~\\

Consider the self-adjoint operator in $L^2(\R^2)$
\begin{equation}\label{eq-P}
P=-\left(\partial_{x_1}-i\frac{x_2^2}2\right)^2-\partial_{x_2}^2\,.\end{equation}
The ground state energy
\begin{equation}
 \lambda_0= \inf \sigma (P)
\end{equation}
 of the operator $P$ is
described using the Montgomery operator as follows.\\
 If $\tau\in\R$,
let $\lambda(\tau)$ be the first eigenvalue of the Montgomery
operator \cite{M},
\begin{equation}\label{deflambda}
P(\tau)=-\frac{d^2}{dx_2^2}+\Big(\frac{x_2^2}{2} +\tau\Big)^2\,,\quad{\rm in
~}L^2(\R)\,.
\end{equation}
Notice that the eigenvalue $\lambda(\tau)$ is positive, simple and
has a unique positive eigenfunction $\varphi^\tau$  of $L^2$ norm $1$.
 There exists a {\bf unique} $\tau_0\in\R$ such that
\begin{equation}\label{eq-lambda0}
\lambda_0=\lambda(\tau_0)\,.
\end{equation}
Hence $\lambda_0 >0$. We write
$$
\varphi_0=\varphi^{\tau_0}
$$
Clearly, the function
\begin{equation}\label{eq-psi0}
\psi_0(x_1,x_2)=e^{-i\tau_0x_1}\,\varphi_0(x_2)\,,
\end{equation}
is a bounded (generalized) eigenfunction of the operator $P$ with eigenvalue
$\lambda_0$. Moreover (see \cite{Hel} and references therein) the minimum of $\lambda$
at $\tau_0$ is non-degenerate.\\

We collect some important properties of the family of operators
$P(\tau)$.

\begin{thm}\label{thm:Hel}(\cite{Hel})
\begin{enumerate}
\item $\tau_0<0\,$.
\item $\lim_{\tau\to\pm\infty}\lambda(\tau)= \infty\,.$
\item The function $\lambda(\tau)$ is increasing on the interval
$[0,\infty)\,$.
\end{enumerate}
\end{thm}
~\\

\subsection{A one dimensional energy}~\\

Let $b>0$ and $\alpha\in\R$. Consider the functional
\begin{equation}\label{eq:f1D}
\mathcal E^{1D}_{\alpha,b}(f)=\int_{-\infty}^\infty\left(|f'(t)|^2+\left(\frac{t^2}2+\alpha\right)^2|f(t)|^2-b\, |f(t)|^2+\frac{b}2\, |f(t)|^4\right)\,dt\,,
\end{equation}
defined over configurations in the space
$$ B^1(\R)=\{f\in H^1(\R;\R)~:~t^2f\in L^2(\R)\}\,.$$
In light of Theorem~\ref{thm:Hel}, we may define  two
functions $z_1(b)$ and $z_2(b)$ satisfying,
\begin{equation}\label{eq:z12}
z_1(b)<\tau_0<z_2(b)\,,\quad \lambda^{-1}\big([\tau_0,b)\big)=(z_1(b),z_2(b))\,.
\end{equation}
Notice that, if $b<\lambda(0)\,$, then $z_2(b)<0\,$. This follows from (3) in
Theorem~\ref{thm:Hel}.
\begin{thm}\label{thm:FH-ch14}(\cite[Sec.~14.2]{FH-b})
\begin{enumerate}
\item {The functional $\mathcal E^{1D}_{\alpha,b}$ has a non-trivial
minimizer in the space $B^1({\R})$ if and only if\break
$\lambda(\alpha)<b\,$. Furthermore, a non-trivial minimizer $f_\alpha$
can be found  which is a positive function and  $\pm f_\alpha$ are the only real-valued minimizers.}
%
\item Let
\begin{equation}\label{eq:be}
\be(\alpha,b)=\inf\{\mathcal E^{1D}_{\alpha,b}(f)~:~f\in B^1(\R)\}\,.
\end{equation}
There exists $\alpha_0\in (z_1(b),z_2(b))$ such that,
\begin{equation}\label{defalpha0}
\be(\alpha_0,b)=\inf_{\alpha\in\R}\be(\alpha,b)\,.
\end{equation}
\item If $b<\lambda(0)$, then $\alpha_0<0\,$.
\item (Feynman-Hellmann)
\begin{equation}\label{FH}
\int_{-\infty}^\infty\left(\frac{t^2}2+\alpha_0\right)|f_{\alpha_0}(t)|^2\,dt=0\,.
\end{equation}
\end{enumerate}
\end{thm}
The proof of this theorem can be obtained by adapting the analysis
of \cite[Sec.~14.2]{FH-b}  devoted to the functional
\begin{equation}\label{degennes}
\mathcal F^{1D}_{\alpha,b}(f)=\int_{0}^\infty\left(|f'(t)|^2+\left(t+\alpha\right)^2|f(t)|^2-b|f(t)|^2+\frac{b}2|f(t)|^4\right)\,dt\,.
\end{equation}
We note for future use that a minimizer of $\mathcal E^{1D}_{\alpha,b}$ satisfies the Euler-Lagrange equation:
\begin{equation}\label{Euler-Lagrange}
- f'' (t) + (\frac{t^2}{2}+\alpha)^2 f(t) - b f(t) + b f(t)^3 =0\,,
\end{equation}
and that $f\in \mathcal S(\mathbb R)$.

According to Theorem~\ref{thm:FH-ch14}, we observe that the
functional $\mathcal E^{1D}_{\alpha,b}$ has non-trivial minimizers
if and only if $\alpha\in (z_1(b),z_2(b))$. \\

\subsection{Reduced Ginzburg-Landau functional}~\\

Let $L>0$, $R>0$, $\mathcal S_R=(-R,R)\times \R$ and
\begin{equation}\label{eq-A-app}
\Ap(x)=\Big(-\frac{x_2^2}2,0\Big)\,,\quad
\Big(x=(x_1,x_2)\in\R^2\Big)\,.\end{equation}
Consider the functional
\begin{equation}\label{eq-gs-er''}
\mathcal E_{L,R}(u)=\int_{\mathcal S_R}\left(|(\nabla-i\Ap)u|^2-L^{-2/3}|u|^2+\frac{L^{-2/3}}{2}|u|^4\right)\,dx\,,
\end{equation}
and the ground state energy
\begin{equation}\label{eq-gs-er'}
\er(L;R)=\inf\{\mathcal E_{L,R}(u)~:~u\in H^1_0(\mathcal S_R)\}\,.
\end{equation}

The functional in \eqref{eq-gs-er''} has a minimizer (see
\cite[Theorem~3.6]{FKP-jmpa} or \cite{Pa02}).  Useful properties of
the minimizers are collected in the next theorem.

\begin{thm}\label{thm:FKP}
Let $L>0$, $R>0$ and $\varphi_{L,R}\in H^1_0(\mathcal S_R)$ be a minimizer of the
functional $\mathcal E_{L,R}$ in \eqref{eq-gs-er''}. There holds,
\begin{equation}\label{eq-phiR<1}
\|\varphi_{L,R}\|_\infty\leq 1\,.
\end{equation}

Furthermore, there exist universal positive  constants $C_1$ and
$C_2$ such that the following is true:
\begin{enumerate}
\item If $L<2^{-\frac 32}$ and  $R>0$, then
\begin{equation}\label{eq-dec-phiR10}
\int_{\mathcal S_R\cap\{|x_2|\geq
4L^{-2/3}\}}\frac{|x_2|^3}{(\ln
|x_2|)^2}|(\nabla-i\Ap)\varphi_{L,R}|^2\,dx \leq
C_1L^{-8/3}|\ln L|^{-2}\,R\,,
\end{equation}
\begin{equation}\label{eq-dec-phiR20}
\int_{\mathcal S_R\cap\{|x_2|\geq 4L^{-2/3}\}}\frac{|x_2|}{(\ln
|x_2|)^2} |\varphi_{L,R}(x)|^2\leq
C_1\,L^{-1/3}|\ln L|^{ -\frac 32}\,R\,.\end{equation}

\item If $L\geq 2^{-\frac 32} $ and $R>0$, then
\begin{equation}\label{eq-dec-phiR1}
\int_{\mathcal S_R\cap\{|x_2|\geq 8\}}\frac{|x_2|^3}{(\ln
|x_2|)^2}|(\nabla-i\Ap)\varphi_{L,R}|^2\,dx \leq C_2\,
L^{2/3}\,R\,,
\end{equation}
\begin{equation}\label{eq-dec-phiR2}
\int_{\mathcal S_R\cap\{|x_2|\geq 8\}}\frac{|x_2|}{(\ln
|x_2|)^2} |\varphi_{L,R}(x)|^2\leq
C_2 \,L^{2/3}\,R\,.\end{equation}
\end{enumerate}
\end{thm}

\begin{proof}~\\
The minimizer $\varphi_{L,R}$ satisfies the Ginzburg-Landau
equation,
\begin{equation}\label{eq-GL-r}
-(\nabla-i\Ap)^2\varphi_{L,R}=L^{-2/3}(1-|\varphi_{L,R}|^2)\varphi_{L,R}\,.\end{equation}
Hence \eqref{eq-phiR<1} results from the strong maximum principle.

Let
\begin{equation}\label{choiceofl}
\ell=\max(2,L^{-2/3})
\end{equation}
 and $\chi(x_2)$ be  a nonnegative even
smooth function satisfying
$$\chi(x_2)=0\quad{\rm if~}|x_2|\leq \ell\,,\qquad
\chi(x_2)=\frac{|x_2|^{3/2}}{\ln|x_2|}\quad{\rm if}~|x_2|\geq 4\ell\,.$$
Clearly,  $\chi'$ satisfies with our choice of $\ell$,
$$
0<|\chi'(x_2)|\leq\frac{3\sqrt{|x_2|}}{2\ln {\Be |} x_2{\Be |}}\quad{\rm if}~|x_2|\geq 4\ell\,,$$
and one can choose $\chi $ on $[\ell, 4 \ell]$   such that $\chi'$ satisfies,
$$|\chi'(x_2)|\leq \widehat C\,  \ell^{1/2}\,|\ln\ell|^{-1}\quad{\rm if~}|x_2|\leq 4\ell\,,$$
for some constant $\widehat C$ independent of $\ell\,$.\\
Multiplying both sides of \eqref{eq-GL-r} by $\chi^2\varphi_{L,R}$
and integrating by parts yield,
\begin{equation}\label{eq-GL-ibp}
\int_{\mathcal S_R}
\left(|(\nabla-i\Ap)\chi\varphi_{L,R}|^2-L^{-2/3}|\chi\varphi_{L,R}|^2+L^{-2/3}\chi^2|\varphi_{L,R}|^4
\right)\,dx=
\int_{\mathcal S_R}|\chi' (x_2) \,\varphi_{L,R}(x)|^2\,dx\,.
\end{equation}
The following inequality is standard in the spectral analysis of
the Schr\"odinger operators with magnetic fields \cite[Lemma~1.4.1]{FH-b},
\begin{equation}\label{eq:lb}
\int_{\mathcal S_R}|(\nabla-i\Ap)\chi\varphi_{L,R}|^2\,dx \geq
\int_{\mathcal S_R} |\curl\Ap (x)|\,|\chi(x_2) \varphi_{L,R}(x)|^2\,dx =
\int_{\mathcal S_R}|x_2|\, |\chi\varphi_{L,R}|^2\,dx\,.
\end{equation}
Since the function $x\mapsto \chi(x_2) \varphi_{L,R}(x)$ is supported in $\mathcal
S_R\cap\{|x_2|\geq \ell\}$ and $$ \ell=\max(2,L^{-2/3})\geq L^{-2/3}\,,$$
we infer from the previous inequality that,
$$
\int_{\mathcal S_R}|(\nabla-i\Ap)\chi \varphi_{L,R} |^2\,dx
 \geq L^{-2/3}\int_{\mathcal S_R}
|\chi(x_2) \varphi_{L,R}(x)|^2\,dx\,.
$$
This inequality allows us to deduce from \eqref{eq-GL-ibp} that,
\begin{equation}\label{eq-GL-ibp'}
L^{-2/3}\int_{\mathcal S_R}
\chi(x_2)^2|\varphi_{L,R}(x)|^4
\,dx\leq
\int_{\mathcal S_R}|\chi' (x_2)\, \,\varphi_{L,R}(x)|^2\,dx\,.
\end{equation}
We obtain an upper bound of the term on the right side as follows.
First, we use the Cauchy-Schwarz inequality to write,
\begin{align}
\int_{\mathcal S_R\cap\{|x_2|\geq 4\ell\}}&\frac{|x_2|}{(\ln|x_2|)^2}\,|\varphi_{L,R}(x)|^2\,dx\nonumber\\
&\leq\left(\int_{\mathcal S_R\cap\{|x_2|\geq 4\ell\}}\frac1{|x_2|(\ln|x_2|)^2}\,dx\right)^{1/2}
\left(\int_{\mathcal S_R\cap\{|x_2|\geq 4\ell\}}
\frac{|x_2|^3}{(\ln|x_2|)^2}\,|\varphi_{L,R}(x)|^4\,dx\right)^{1/2}\nonumber\\
&\leq C \,  |\ln \ell|^{-\frac 12}  R^{1/2}\left(\int_{\mathcal S_R\cap\{|x_2|\geq 4\ell\}}
\frac{|x_2|^3}{(\ln|x_2|)^2}\,|\varphi_{L,R}(x)|^4\,dx\right)^{1/2}\,.\label{eq:R20''}
\end{align}
Next we use the assumptions on $\chi$, \eqref{eq-phiR<1} and
\eqref{eq:R20''},
\begin{align}
&\int_{\mathcal S_R}|\chi' (x_2)\, \,\varphi_{L,R}(x)|^2\,dx\nonumber\\
&\leq
\int_{\mathcal S_R\cap\{|x_2|\geq 4\ell\}}
\frac{9|x_2|}{4(\ln|x_2|)^2}\,|\varphi_{L,R}(x)|^2\,dx+\int_{\mathcal
S_R\cap\{|x_2|\leq 4\ell\}} |\chi' (x_2)|^2\,|\varphi_{L,R}(x)|^2\,dx \nonumber\\
&\leq C\,  |\ln \ell|^{-\frac 12}  R^{1/2}\left(\int_{\mathcal S_R\cap\{|x_2|\geq 4\ell\}}
\frac{|x_2|^3}{(\ln|x_2|)^2}\,|\varphi_{L,R}(x)|^4\,dx\right)^{1/2}+C\, \ell^2 (\ln\ell)^{-2}\, R \,.\label{eq:R20'''}
\end{align}
Now we use the assumption on $\chi$, \eqref{eq-GL-ibp'} and
\eqref{eq:R20'''} to get,
\begin{align*}
L^{-2/3}&\int_{\mathcal S_R\cap\{|x_2|\geq 4\ell\}}
\frac{|x_2|^3}{(\ln|x_2|)^2}\,|\varphi_{L,R}(x)|^4 \,dx \\
&\leq L^{-2/3}\int_{\mathcal S_R}
|\chi(x_2)\,\varphi_{L,R}(x)|^4
\,dx\\
&\leq C \,  |\ln \ell|^{-\frac 12}  R^{1/2} \left( \int_{\mathcal S_R\cap\{|x_2|\geq 4\ell\}}
\frac{|x_2|^3}{4(\ln|x_2|)^2}\,|\varphi_{L,R}(x) |^4\,dx\right)^{1/2}+C \, \ell^2(\ln\ell)^{-2} \, R \,.
\end{align*}
As a consequence we obtain with a new constant $C$  the following
inequality,
\begin{equation}\label{eq:R20-4'}
\int_{\mathcal S_R\cap\{|x_2|\geq 4\ell\}}
\frac{|x_2|^3}{(\ln|x_2|)^2}\,|\varphi_{L,R}(x)|^4 \,dx\leq C \left(   L^\frac 23   |\ln \ell|^{-1}  + \ell^2 (\ln\ell)^{-2}\right)\,L^\frac 23 R \,.
\end{equation}
Next, we insert \eqref{eq:R20-4'} into \eqref{eq:R20''} and get,
\begin{equation}
\int_{\mathcal S_R\cap\{|x_2|\geq 4\ell\}}\frac{|x_2|}{(\ln|x_2|)^2}\,|\varphi_{L,R}(x)|^2\,dx
\leq C\, |\ln \ell|^{-\frac 12}   \left(   L^\frac 23   |\ln \ell|^{-1}  + \ell^2 (\ln\ell)^{-2}\right)^\frac 12 \,L^\frac 13 R\,.\label{eq:R20'}
\end{equation}
Again, the choice of $\ell$ allows us to deduce
\eqref{eq-dec-phiR20} and \eqref{eq-dec-phiR2} from
\eqref{eq:R20''}.\\

Now, we show how to get the two inequalities in
\eqref{eq-dec-phiR10} and  \eqref{eq-dec-phiR1}. A simple
decomposition of the integral below and the inequality in
\eqref{eq-phiR<1} yield,
$$
\int_{\mathcal S_R} |\chi (x_2)\varphi_{L,R} (x)|^2\,dx\leq
\int_{\mathcal S_R\cap \{|x_2|\geq 4\ell\}} |\chi(x_2) \varphi_{L,R}(x)|^2\,dx+C\, \ell^4 |\ln \ell|^{-2}\,  R\,.
$$
The next inequality is an easy consequence of \eqref{eq:lb},
$$
\int_{\mathcal S_R}|(\nabla-i\Ap)\chi\varphi_{L,R}|^2\,dx \geq
4\ell \int_{\mathcal S_R\cap\{|x_2|\geq 4\ell\} } |\chi(x_2) \varphi_{L,R}(x)|^2\,dx\,.
$$
We insert these two inequalities into \eqref{eq-GL-ibp}. We get,
\begin{equation}\label{eq:R3}\frac12\int_{\mathcal
S_R}|(\nabla-i\Ap)\chi\varphi_{L,R}|^2\,dx\leq \int_{\mathcal
S_R}|\chi'(x_2)\,\varphi_{L,R}(x)|^2\,dx+CL^{-2/3} \ell^4 |\ln \ell|^{-2} R\,.\end{equation} Now we use
a simple commutator and get the following inequality,
$$
\int_{\mathcal
S_R}|(\nabla-i\Ap)\chi\varphi_{L,R}|^2\,dx\geq \frac12\int_{\mathcal
S_R}|\chi(x_2) (\nabla-i\Ap)\varphi_{L,R}|^2\,dx-4\int_{\mathcal
S_R}|\chi'(x_2)\varphi_{L,R}(x)|^2\,dx\,.
$$
As a consequence, we infer from \eqref{eq:R3},
\begin{equation}\label{eq:R3'}\frac14\int_{\mathcal
S_R}|\chi(\nabla-i\Ap)\varphi_{L,R}|^2\,dx\leq 3\int_{\mathcal
S_R}|\chi'(x_2)\,\varphi_{L,R}(x)|^2\,dx+CL^{-2/3} \ell^4 |\ln \ell|^{-2} R\,.\end{equation}
   The term on the
right side is controlled using \eqref{eq:R20'} and
\eqref{eq:R20-4'}. In that way we get,
$$\frac12\int_{\mathcal
S_R}|\chi(\nabla-i\Ap)\varphi_{L,R}|^2\,dx\leq
C\sqrt{L^{4/3}+L^{2/3}\ell^2(\ln\ell)^{-2}}\,R+C  \ell^4 |\ln \ell|^{-2} R\,.$$ Thanks to the choice of $\chi$ and $\ell$, we deduce the inequalities in \eqref{eq-dec-phiR10} and \eqref{eq-dec-phiR1}.
\end{proof}

\begin{rem}\label{rem:L<Lambda}
Let $\Lambda>0\,$. The inequalities in
\eqref{eq-dec-phiR10}-\eqref{eq-dec-phiR2} remain true under the
assumption that $L\leq \Lambda$ with the constant $C_1$ replaced by
another constant $C_\Lambda$.
\end{rem}

As a consequence of Theorem~\ref{thm:FKP}, we can obtain a uniform
estimate of the energy of a minimizer $\varphi_{L,R}\,$.

\begin{prop}\label{prop:dec-m}
Let $\Lambda>0\,$. There exists a positive constant $C_\Lambda$  such that, for all $L\in\,(0,\Lambda)$ and $R>0$\,,
\begin{multline*}
\int_{\mathcal S_R} |\varphi_{L,R}(x)|^2
\,dx\leq C_\Lambda L^{-2/3}R\quad  {\rm and}\quad \int_{\mathcal
S_R}|(\nabla-i\Ap)\varphi_{L,R}|^2\,dx\leq C_\Lambda
L^{-4/3}R\,.\end{multline*}
\end{prop}
\begin{proof}
Let $\ell=\max(2,L^{-2/3})$.  As a consequence of the
inequalities in Theorem~\ref{thm:FKP}, we have,
$$
\int_{\mathcal S_R\cap\{|x_2|\geq 4\ell\}}
\frac{|x_2|}{(\ln|x_2|)^2} |\varphi_{L,R}(x)|^2
\,dx\leq C\,\bar \lambda\,R\,,
$$
where
$$\bar \lambda=L^{-2/3}(1+|\ln
L|)^{-2}+L^{4/3}\,.
$$
Since $\ell\geq 2$, then there exists a universal constant $a$ such
that
$$|x_2|\geq 4\ell\implies \frac{|x_2|}{(\ln|x_2|)^2}\geq
a>0\,.$$
Consequently, we get,
$$
\int_{\mathcal S_R\cap\{|x_2|\geq 4\ell\}}
|\varphi_{L,R}(x)|^2
\,dx\leq  \frac{C \bar \lambda}{a}\,R\,.
$$
On the other hand, using \eqref{eq-phiR<1}, we see that,
$$
\int_{\mathcal S_R\cap\{|x_2|\leq 4\ell\}}
|\varphi_{L,R}(x) |^2
\,dx\leq C\ell\,R\,.
$$
A combination of both inequalities lead us to
$$
\int_{\mathcal S_R}
|\varphi_{L,R}(x)|^2
\,dx\leq C\ell\,R+\frac{C\bar \lambda}{a}\,R\,.
$$
We have $\ell\sim L^{-2/3}$ and $\bar \lambda\sim L^{-2/3}|\ln L|^{-2}$
as  $L\to0_+$.  Hence, we get
$$
\int_{\mathcal S_R}
|\varphi_{L,R}(x)|^2
\,dx\leq C_\Lambda L^{-2/3}\,R\,.
$$
To finish the proof of the theorem, we multiply both sides of
\eqref{eq-GL-r} by $\overline{\varphi_{L,R}}$ and integrate by
parts. In that way we obtain
$$\int_{\mathcal S_R}|(\nabla-i\Ap)\varphi_{L,R}|^2\,dx\leq L^{-2/3}\int_{\mathcal S_R}|\varphi_{L,R}(x)|^2\,dx\leq C_\Lambda L^{-4/3}\,R\,.$$
\end{proof}

The value of the ground state energy defined in \eqref{eq-gs-er'} is
connected to the eigenvalue $\lambda_0$ in \eqref{eq-lambda0}.

\begin{prop}\label{eq-gs-lambda0}
For $R>0$, $L>0$, if $\er(L;R)$ denotes the ground state energy in
\eqref{eq-gs-er'}, there holds:
\begin{enumerate}
\item If $L\geq \lambda_0^{-3/2}$, then $\er(L;R)=0$\,.
\item There exist positive constants $C_1$, $C_2$ and $C_3$ such that, if   $L< \lambda_0^{-3/2}$ and $R>0$, then
\begin{equation}
-C_1\, L^{-4/3} R  \leq \frac{\er(L;R)}{(1-\lambda_0L^{2/3}) } \leq -C_2\, L^{-2/3} R+\displaystyle\frac{C_3}{R}\,.
\end{equation}
\end{enumerate}
\end{prop}
\begin{proof}~\\
Suppose that $L\geq \lambda_0^{-3/2}$. Let $u\in H^1_0(\mathcal
S_R)$. The min-max principle and the condition on $L$ tell us that
$\mathcal E_{L,R}(u)\geq 0$ and consequently $\er(L;R)\geq0\,$. But $\er(L;R)\leq \mathcal E_{L,R}(0)=0\,$. This proves the statement in (1).\\
Now, suppose that $L<\lambda_0^{-3/2}$. Let $\theta \in
C_c^\infty(\R)$ be a function satisfying, \begin{equation*} {\rm
supp}\,\theta\subset (-1,1)\,,\quad 0\leq\theta \leq  1\,,\quad
\theta =1~{\rm in~}(-1/2,1/2)\,,
\end{equation*}
 and let
 $$\theta_R(x):= \theta(x/R)\,.$$
Let $t>0$ and $$ u(x_1,x_2)=t\,\theta_R(x_1)\,\psi_0(x_1,x_2)\,,$$ where
$\psi_0$ is the function in \eqref{eq-psi0}. \\
Recall that $\psi_0$ satisfies $-(\nabla-i\Ap)^2\,
\psi_0=\lambda_0\psi_0\,$. An integration by parts yields,
\begin{align*}\int_{\mathcal
S_R}|(\nabla-i\Ap)u|^2\,dx&=t^2\left(\Big\langle \theta_R(x_1)^2
\psi_0\,,\,-(\nabla-i\Ap)^2\psi_0\Big\rangle+\int_{\mathcal
S_R}|\phi_0(x_2)  \theta_R'(x_1) |^2\,dx\right)\\
& =t^2\left(\Big\langle \theta_R(x_1)^2
\psi_0\,,\,-(\nabla-i\Ap)^2\psi_0\Big\rangle+ \frac 1R \int
  \theta'(x_1) ^2\,dx_1\right)\\
&=t^2\left(\lambda_0\int_{\mathcal
S_R}|\theta_R(x_1) \psi_0(x)|^2\,dx+\frac{C}{R}\right)\,.
\end{align*}
As a consequence, we get that,
\begin{align*}
\er(L,R)\leq \mathcal
E_{L,R}(u)&\leq t^2\left(\lambda_0-L^{-2/3}\right)\int_{\mathcal
S_R}|\theta_R\psi_0|^2\,dx+t^2\frac{C}{R}+t^4\frac{L^{-2/3}}2
\int_{\mathcal S_R}|\theta_R(x_1) \psi_0(x)|^4dx\,.\\
&\leq t^2\left(R\left(\lambda_0-L^{-2/3}\right) +\frac{C}{R}\right)+R \nu L^{-2/3}\,t^4\,.
\end{align*}
Here
$$\nu =\displaystyle\int_\R|\varphi_0(x_2)|^4\,dx_2$$ and $\varphi_0$ is
the $L^2$-normalized function introduced in \eqref{eq-phiR<1}.

Selecting $t$ such that
$$\left(\lambda_0-L^{-2/3}\right)+\nu L^{-2/3}\,t^2=\frac12\left(\lambda_0-L^{-2/3}\right)$$
finishes  the proof of the upper bound.

The lower bound is obtained as follows.  Let $\varphi_{L,R}$ be the
minimizer in Theorem~\ref{thm:FKP}. It follows from the min-max
principle that,
$$\er(L;R)=\mathcal E_{L,R}(\varphi_{L,R})\geq
L^{-2/3}(\lambda_0L^{2/3}-1)\int_{\mathcal
S_R}|\varphi_{L,R}(x)|^2\,dx\,.$$ Under the assumption
$L<\lambda_0^{-2/3}$, Proposition~\ref{prop:dec-m} tells us that
$$\int_{\mathcal
S_R}|\varphi_{L,R}(x)|^2\,dx\leq C_1L^{-2/3}R\,.$$
{As a consequence, we get the lower bound.}
\end{proof}

\begin{rem}\label{rem:1}~\\
In light  of Propositions~\ref{prop:dec-m}~and~\ref{eq-gs-lambda0},
we observe that:
\begin{enumerate}
\item If $L\geq \lambda_0^{-2/3}$, then  $\varphi_{L,R}=0$ is the
minimizer of the functional in \eqref{eq-gs-er''} realizing the
ground  state energy in \eqref{eq-gs-er'}.
\item   If $L\leq \lambda_0^{-2/3}$,  every minimizer $\varphi_{L,R}$ satisfies,
\begin{equation}
\int_{\mathcal S_R}|(\nabla-i\Ap)\varphi_{L,R}(x)|^2\,dx\leq CL^{-4/3}\,,\quad \int_{\mathcal S_R} |\varphi_{L,R}(x)|^2\,dx\leq
CL^{-2/3} R\,,
\end{equation}
 where $C$ is a universal constant.
\end{enumerate}
\end{rem}

Notice that the energy $\mathcal E_{L,R}(u)$ in \eqref{eq-gs-er''}
is invariant under translation along the $x_1$-axis. This allows us
to follow the approach in \cite{FK-cpde, Pa02} and obtain that the
limit of $\frac{\er(L;R)}{R}$ as $R\to \infty$ exists. The precise
statement is:

\begin{thm}\label{thm-FK}
{Given $L>0$, there exists $E(L)\leq 0$ such that,
$$\lim_{R\to\infty}\frac{\er(L;R)}{2R}=E(L)\,.$$
The function $(0,\infty)\,\ni L\mapsto
E(L)\in\,(-\infty,0]$ is continuous, monotone increasing and
$$E(L)=0\quad \mbox{ if and only if}~L\geq \lambda_0^{-3/2}\,.$$
Furthermore,
\begin{equation}
\forall~R >0\,,\forall L>0\,,\quad
E(L)\leq \frac{\er(L;R)}{2R} \,,
\end{equation}
and there exists a  constant $C$ such that\begin{equation}
\forall~R\geq 2\,,\forall L>0\,,\quad
 \frac{\er(L;R)}{2R}\leq E(L)+C\left(1+L^{-2/3}\right)R^{-2/3}\,.
\end{equation}
}
\end{thm}
\begin{proof}~\\
There is nothing to prove when $L \geq \lambda_0^{-3/2}$, hence we
assume that
$$
0 < L <  \lambda_0^{-3/2}\,.
$$
{\bf Step~1.} Let $n\geq 2$ be a natural number, $a\in (0,1)$ and
{consider the family of strips }
$$I_j=\big(-n^2-1-a+(2j-1)(1+\frac{a}2)\,,\,-n^2-1+(2j+1)(1+\frac{a}2)\big)\,\times\R\,,\quad (j\in\mathbb Z)\,.$$
{ Notice that the width of each strip in the family $(I_j)$ is
$2(1+a)$, and if two strips in the family overlap, then the width of
the overlapping region is $a$.} Consider a partition of unity of
$\R^2$ such that,
$$\sum_{j}|\chi_j|^2=1\,,\quad0\leq \chi_j\leq1\,,\quad
\sum_j|\nabla\chi_j |^2 \leq \frac{C}{a^2}\,,\quad {\rm supp}\,\chi_j\subset
I_j\,,$$ where $C$ is a universal constant.\\
Define $\chi_{R,j}(x)=\chi_j(x/R)$. That way we obtain the new
partition of unity,
$$\sum_{j}|\chi_{R,j}|^2=1\,,\quad0\leq \chi_{R,j}\leq1\,,\quad
\sum_j|\nabla\chi_{R,j} |^2 \leq \frac{C}{a^2R^2}\,,\quad {\rm
supp}\,\chi_{R,j}\subset I_{R,j}\,,$$ where $I_{R,j}=\{R\,x~:~x\in I_{j}\}$.

Notice that $(I_{R,j})_{j\in\{1,2,\cdots, n^2\}}$  is a covering of $\mathcal
S_{n^2R}=(-n^2R,n^2R)\times \R$ by $n^2$ strips, each having
side-length $ 2(1+a)R\,$.

Let $\varphi_{L,n^2R}\in H^1_0(\mathcal S_{n^2R})$ be the minimizer
in Theorem~\ref{thm:FKP}. There holds the decomposition,
\begin{align*}
\er(L;n^2R)& =\mathcal E_{L,n^2R}(\varphi_{L,n^2R})\\&\geq \sum_{j=1}^{n^2}\left(\mathcal E_{L,n^2R}(\chi_{R,j}\varphi_{ L,n^2R})
-\big\|\,|\nabla\chi_{R,j}|\,\varphi_{L,n^2R}\big\|_{L^2(\mathcal S_{n^2R})}^2\right)\\
&=\left(\sum_{j=1}^{n^2} \mathcal E_{L,n^2R}(\chi_{R,j}\varphi_{ L,n^2R})\right)
-\int_{\mathcal S_{n^2R}}\left(\sum_{j=1}^{n^2}|\nabla\chi_{R,j}|^2\right)|\,\varphi_{L,n^2R}|^2\,dx\\
&\geq \left(\sum_{j=1}^{n^2}\mathcal E_{L,n^2R}(\chi_{R,j}\varphi_{L,n^2R})\right)
-\frac{Cn^2L^{-2/3}}{a^2R}\quad[{\rm By~Remark~\ref{rem:1}}]\,.
\end{align*}
The function $\chi_{R,j}\varphi_{L, n^2R}$ is supported in an
infinite strip of {width $ 2(1+a)R \,$.} Since the energy $\mathcal
E_{L,R}(u)$ is (magnetic) translation-invariant along the $x_1$
direction, we get
$$ \forall~j\,,\quad\mathcal E_{L,n^2R}(\chi_{R,j}\varphi_{ L,n^2R})\geq
\er(L; (1+a) R)$$
and consequently,
$$\er(L;n^2R)\geq
n^2\er(L; (1+a) R)- C\,\frac{n^2L^{-2/3}}{a^2R}\,.$$ Dividing both
sides of the above inequality  by $ n^2R$ and using the estimate in
Proposition~\ref{eq-gs-lambda0}, we get
\begin{equation*}
\frac{\er(L;n^2R)}{n^2R}\geq
\frac{\er(L; (1+a) R)}{R}-C\left(aL^{-2/3}+ \frac{L^{-2/3}}{a^2R^2}\right)\,.
\end{equation*}
Using the trivial inequality $(1+a)\leq (1+a)^2$, we finally
obtain:
\begin{equation}\label{eq-mon}
\frac{\er(L;n^2R)}{n^2R}\geq
\frac{\er(L; (1+a)^2 R)}{(1+a)^2 R}-C\left(aL^{-2/3}+ \frac{L^{-2/3}}{a^2R^2}\right)\,.
\end{equation}

{\bf Step~2.} Let $\ell>0\,$. Let us define,
$$d(\ell,L)=\frac{\er(L;\ell^2)}{2}\,,\quad
f(\ell,L)=\frac{d(\ell,L)}{\ell^2}\,.$$ Clearly, the function $\ell
\mapsto d(\ell,L)$ is decreasing. Thanks to
Proposition~\ref{eq-gs-lambda0}, we observe that $d(\ell,L)\leq 0$
and $f(\ell,L)$ is bounded. Furthermore, \eqref{eq-mon} used with
$R=\ell^2$ tells us that,
$$f(n\ell,L)\geq f\Big((1+a)\ell,L\Big)-C\left(aL^{-2/3}+
\frac{1}{a^2\ell^2}\right)\,.$$ By \cite[Lemma~3.10]{FKP-jmpa}, we get the existence of $E(L)$ such that
$$\lim_{\ell\to\infty}f(\ell,L)=E(L)\,.$$
The simple change of variable $\ell=\sqrt{R}$ gives us,
$$\lim_{R\to\infty}\frac{\er(L;R)}{2R}=E(L)\,.$$
{\bf Step~3.} Using a comparison argument and the translation
invariance of the energy $\mathcal E_{L,R}(u)\,$, we observe that,
$$\forall~n\in\mathbb N\,,\quad \er(L;n^2R)\leq n^2\er(L;R)\,.$$
Dividing both sides of the above inequality by $2n^2R$ and  taking
$n\to\infty\,$, we get,
$$E(L)=\lim_{n\to\infty}\frac{\er(L;n^2R)}{2n^2R}\leq
\frac{\er(L;R)}{2R}\,.$$
The matching lower bound for $E(L)$ is obtained by taking
$n\to\infty$ in \eqref{eq-mon}, selecting $a=R^{-2/3}$ and
replacing $R$ by $(1+a)^2 R\,$.

{\bf Step~4.} Proposition~\ref{eq-gs-lambda0} tells us that $E(L)=0$
if and only if $L\geq\lambda_0^{-3/2}$. The continuity and
monotonicity properties of $E(L)$ are easily obtained through the
study of the energy $\mathcal E_{L,R}(u)$ as a function of $L$. The
details can be found in \cite[Thm~3.13]{FKP-jmpa}.
\end{proof}

As in the case of a constant magnetic field, it would be
desirable to establish a simpler expression of $E(L)$ when $L\in (
\lambda(0)^{-\frac 32}\,,\, \lambda_0^{-\frac 32})\,$.

\begin{conjecture}\label{thm:CN}
Let $\lambda$ be the function introduced in \eqref{deflambda}. If
\begin{equation}\label{condsurL}
\lambda_0<L^{-2/3}<\lambda(0)\,,
\end{equation}
then
$$E(L)=E^{\rm 1D}(L^{-2/3})\,.$$
Here, for $b>0$,  $E^{\rm 1D}(b)=\be(\alpha_0,b)$ and
$\be(\alpha_0,b)$ is  defined in \eqref{defalpha0}.
\end{conjecture}
\begin{rem}
In \cite{Hel},  the following numerical estimate is given:
$\lambda_0\approx 0.57\,$. Furthermore, the lower bound:
\begin{equation}\label{proplambda}
\lambda(0)\leq \left(\frac 34\right)^\frac 43 <1\,,
\end{equation}
 is proved. Finally the strict inequality
$\lambda_0<\lambda(0)$ is a consequence of the uniqueness  of the
point of minimum of the function $\lambda(\tau)$.
\end{rem}

\subsection{The approximate functional}

Let $\nu\in[0,2\pi)$ be a given angle. Define the magnetic
potential:
\begin{equation}\label{eq-A-app-nu}
\App(x)=-\frac{|x|^2}2\nb\,,\quad
\nb=(\cos\nu,\sin\nu)\,,\quad (x\in\R^2)\,.\end{equation}
Let $\kappa>0$, $ \ell\in(0,1)$, $\mathcal D_\ell=D(0,\ell)$ the disc
centered at $0$ and of radius $\ell$, and $L>0$. Consider the
functional:
\begin{equation}\label{eq-GL-r-new}
\mathcal G(\psi)=\int_{\mathcal D_\ell}\left(|(\nabla-iL\kappa^3  \App)\psi|^2-\frac{\kappa^2}2|\psi|^2+\frac{\kappa^2}2|\psi|^4\right)\,dx\,,
\end{equation}
together with the ground state energy
\begin{equation}\label{eq-gs-Er}
\Er(\kappa,L,\nu;\ell)=\inf\{\mathcal G(\psi)~:~\psi\in H^1_0(\mathcal D_\ell)\}\,.
\end{equation}
The change of variable $x\mapsto \sqrt{m}\,\kappa\,x$ yields
\begin{equation}\label{eq-gs-er}
\Er(\kappa,L;\ell)=\er(\nu,L;R)\,,
\end{equation}
where $m=L^{2/3}$, $R=\sqrt{m}\,\kappa\,\ell$, $\mathcal
D_R=D(0,R)$,
\begin{equation}\label{eq-gs-er-nu''}
\mathcal E_{\nu,L,R}(u)=\int_{\mathcal D_R}\left(|(\nabla-i\App)u|^2-L^{-2/3}|u|^2+\frac{L^{-2/3}}{2}|u|^4\right)\,dx\,,
\end{equation}
and
\begin{equation}\label{eq-gs-er-nu'}
\er(\nu,L;R)=\inf\{\mathcal E_{\nu,L,R}(u)~:~u\in H^1_0(\mathcal D_R)\}\,.
\end{equation}
We now show that  the ground state energy $\er(\nu,L;R)$ is
independent of $\nu$. Let $u$ be a given function in $H^1_0(\mathcal
D_R)$. We perform the rotation
$$(x_1,x_2)\mapsto
\big(x_1\cos\nu+x_2\sin\nu\,,\,-x_1\sin\nu + x_2\cos\nu \big)\,,$$ which
transforms the function $u$ to a new function $\widetilde u$,  then
 the gauge transformation \break $\widetilde u\mapsto v=
e^{ix_1^3/6}\widetilde u$ and get
$$\mathcal E_{\nu,L,R}(u)=\int_{\mathcal
D_R}\left(|(\nabla-i\Ap)v|^2-L^{-2/3}|v|^2+\frac{L^{-2/3}}{2}|v|^4\right)\,dx=:G_{L,R}(v)\,,$$ where $\Ap$ is introduced in \eqref{eq-A-app}. \\ Hence  we get,
\begin{equation}\label{eq-G}
\er(\nu,L;R)=\inf\{G_{L,R}(v)~:~v\in H^1_0(\mathcal D_R)\}\,.
\end{equation}
This simple observation allows us to prove the following theorem:
\begin{thm}\label{thm-tdl-e} For $\nu\in[0,2\pi)$,   $L>0$ and $ R>0$,
 we have,
\begin{equation}\label{errnu}
 \er(\nu,L;R)=  \er(0,L;R) \geq \er(L;R)\,,
 \end{equation}
where   $\er(L;R)$ and $\er(\nu,L;R)$
are the ground state energies introduced  in \eqref{eq-gs-er'} and
\eqref{eq-gs-er-nu'}.\\
Moreover, there exists a constant $C$ such that, for   $L>0$, $\rho
\in (0,1)$, and $$R\geq \max\left(2^{1/\rho}L^{1/(3\rho)},
2^{1/(1-\rho)}  L^{-1/(1-\rho)}, 2\right)\,,$$ we have
\begin{equation}\label{errnu2}
\er(0,L;R)\leq \er\Big(L;\big(1-L^{2/3}R^{-2\rho}\big)R\Big)+CL^{-4/3}|\ln L|^{-1}|\ln(L^{1/3}R^{-\rho})|^2\,R^\rho\,.
\end{equation}
\end{thm}
\begin{proof}~\\
The independence of $\nu$ was observed in \eqref{eq-G}. From now on we can take $\nu=0\,$.\\~\\
{\bf Lower bound.} Let $u\in H^1_0(\mathcal D_R)$ be a minimizer of
the functional $G_{L,R}$. The function $u$ can be extended by $0$ to
a function in $H^1_0(\mathcal S_R)$. Thus,
$$\er(0,L;R)=G_{L,R}(u)=\mathcal E_{L,R}(u)\geq \er(L;R)\,.$$
{\bf Upper bound.}  Let $a\in\,(0,\frac12)$, $\widetilde R=(1-a)R$ and
$$v=\varphi_{L,\widetilde R}\in H^1_0(\mathcal S_{\widetilde R})$$ a
minimizer of  $G_{L,\widetilde R}$. Remember that
$\varphi_{L,\widetilde R}=0$ when $L\geq \lambda_0^{-3/2}$
(Proposition~\ref{eq-gs-lambda0}). We impose the condition
\begin{equation}\label{hypaa}
\sqrt{a}\,R>2L^{-2/3}\,.
\end{equation}

Consider a test function $\chi\in C_c^\infty(\R)$ such that,
$$\left\{
\begin{array}{l}
0\leq\chi\leq 1\,,\quad {\rm supp}\chi\subset
\,(-\sqrt{a(2-a)}\,R,\,\sqrt{a(2-a)}\,R)\,,\\
 \chi=1{\rm
~in~} (-\sqrt{a(1-a)}\,R,\,\sqrt{a(1-a)}\,R)\,,\\
|\chi'|\leq \frac{C}{\sqrt{a}\,R}\quad{\rm and}\quad | \chi''|\leq \frac{C}{a\,R^2}\,.
\end{array}
\right.$$
Let
$$u(x_1,x_2)=\chi(x_2)\,v(x_1,x_2)\,,\quad (x_1,x_2)\in\R^2\,.$$
Clearly, $u\in H^1_0(\mathcal D_R)$. Thus,
\begin{align*}
\er(0,L;R)&\leq G_{L,R}(u)=\mathcal E_{L,R}(u)\\
&=\int_{\mathcal S_R}\left(\chi(x_2)^2|(\nabla-i\Ap)v|^2-\chi(x_2)\chi''(x_2)|v|^2 -L^{-2/3}|\chi(x_2) v|^2
+\frac{L^{-2/3}}2|\chi(x_2) v|^4\right)\,dx\\
&\leq \int_{\mathcal S_R}\left(|(\nabla-i\Ap)v|^2-L^{-2/3}| v|^2+\frac{L^{-2/3}}2| v|^4\right)\,dx\\
&\qquad  + L^{-\frac 23} \int_{\mathcal S_R}(1-\chi^2)|v|^2\,dx
+\frac{CL^{-1/3}|\ln L|^{-3/2}\big(\ln(\sqrt{a}\,R)\big)^2}{a^{3/2}R^2}\\
&=\int_{\mathcal S_{\widetilde R}}\left(|(\nabla-i\Ap)v|^2-L^{-2/3}| v|^2+\frac{L^{-2/3}}2| v|^4\right)\,dx\\
&\qquad+L^{-2/3}\int_{\mathcal S_R}(1-\chi^2)|v|^2\,dx { +\frac{CL^{-1/3}|\ln L|^{ -\frac 32}\big(\ln(\sqrt{a}\,R)\big)^2}{a^{3/2}R^2}}\\
&=\er(L;\widetilde R)
+\frac{CL^{-1}|\ln L|^{ -\frac 32 }\big(\ln(\sqrt{a}\,R)\big)^2}{a^{1/2}}+\frac{CL^{-1/3}|\ln L|^{-\frac 32}\big(\ln(\sqrt{a}\,R)\big)^2}{a^{3/2}R^2}\,.
\end{align*}
The two terms $L^{-\frac 23} \int_{\mathcal
S_R}(1-\chi(x_2)^2)|v(x_1,x_2)|^2\,dx$ and $\int_{\mathcal
S_R}\chi(x_2) \chi''(x_2)\, |v(x_1,x_2)|^2\,dx$ have been controlled
by using  the decay of $v=\varphi_{L,\widetilde R}$ established  in
Theorem~\ref{thm:FKP} {(Formula \eqref{eq-dec-phiR20})}. Here, we
have used Assumption \eqref{hypaa}.

We select $a=L^{2/3}\,R^{-2\rho}$. Under the assumptions on $R$, $L$
and $\rho$, we see that $0<a<\frac12 $ and $\sqrt{a}\,R>2L^{-2/3}$.
Remembering that $\widetilde R=R-a\,$,  { this achieves}  the proof of
Theorem~\ref{thm-tdl-e}.
\end{proof}
~\\
\subsection{A useful function}~\\

In this subsection, we recall the construction of a function that
describes the energy of the Ginzburg-Landau model with constant
 magnetic field \cite{FK-cpde, SS02}. Consider $b\in\,(0,\infty)$, $r>0\,$, and
$Q_r=\,(-r/2,r/2)\,\times\,(-r/2,r/2)$\,. Define the functional,
\begin{equation}\label{eq:rGL}
F_{b,Q_r}(u)=\int_{Q_r}\left(b|(\nabla-i\Ab_0)u|^2-|u|^2+\frac{1}2|u|^4\right)\,dx\,,
\quad \mbox{ for } u\in H^1(Q_r)\,.
\end{equation}
Here, $\Ab_0$ is the magnetic potential,
\begin{equation}\label{eq:A0}
\Ab_0(x)=\frac12(-x_2,x_1)\,,\quad \big(x=(x_1,x_2)\in \R^2\big)\,.
\end{equation}
Define the two ground state energies,
\begin{align}
&e_D(b,r)=\inf\{F_{b,Q_r}(u)~:~u\in H^1_0(Q_r)\}\,,\label{eq:eD}\\
&e_N(b,r)=\inf\{F_{b,Q_r}(u)~:~u\in H^1(Q_r)\}\,.\label{eq:eN}
\end{align}
It is known \cite{Att, FK-cpde, SS02}  that,
\begin{equation}\label{eq:g}
\forall~b>0\,,\quad g(b)=\lim_{r\to\infty}\frac{e_D(b,r)}{|Q_r|}=\lim_{r\to\infty}\frac{e_N(b,r)}{|Q_r|}\,,
\end{equation}
where  $|Q_r|$ denotes the area of $Q_r$ ($|Q_r|=r^2$) and $g$ is a continuous function such that
\begin{equation} \label{propg}
g(0)=-\frac12 \mbox{ and } g(b)=0 \mbox{ when } b\geq 1\,.
\end{equation}
Furthermore,   there exists a  constant $C$ such that, for all
$r\geq 1$ and  $b>0$,
\begin{equation}\label{eq:g'}
  g(b)-C\frac{\sqrt{b}}{r}\leq\frac{e_N(b,r)}{|Q_r|}\leq\frac{e_D(b,r)}{|Q_r|}\leq g(b)+C\frac{\sqrt{b}}{r}\,.
\end{equation}
We will use the function $g(\cdot)$ to prove the following important
theorem:
\begin{thm}\label{thm:Lto0}
There exist two
positive constants $C_1$ and $C_2$ such that, if
$L\in\,(0,\lambda_0^{-3/2}]\,$, then
$$-\frac{C_1 (1-\lambda_0L^{2/3}) }{L^{4/3}}\leq E(L)\leq -\frac{C_2   (1-\lambda_0L^{2/3}) }{L^{4/3}}\,.$$
\end{thm}
\begin{proof}
The  lower bound follows immediately by sending $R$ to $\infty$ in
the lower bound in Proposition~\ref{eq-gs-lambda0} (see also Theorem
\ref{thm-FK}). The upper bound in the second item of
Proposition~\ref{eq-gs-lambda0} gives  us the upper bound
\begin{equation}\label{eq:Lto0}
E(L)\leq -\frac{C {(1-\lambda_0L^{2/3}) }}{L^{ 2/3}}\,,\end{equation} valid for all
$L\in(0,\lambda_0^{-3/2})\,$. We have just to improve it as $L\rightarrow  0\,$.\\
The improved upper
bound with order $L^{-4/3}$ follows from the construction of a test
function as follows.
Let us cover $\R^2$ by a lattice of squares $\overline{
Q_{\ell,j}}$, where $Q_{\ell,j}=(-\ell+a_j\,,\,\ell-a_j)$ and
\begin{equation}\label{deflm}
\ell= mL^{1/3}\,.
\end{equation}
The choice of the  positive constant $m$  will be specified later.
Notice that the magnetic potential $\Ap$ (cf \eqref{eq-A-app})
satisfies
$${\mathbf B}_{\rm app}=\curl\Ap=x_2\,,$$
and that  the gradient of the magnetic field ${\mathbf B}_{\rm app}$
is
bounded.\\
There exists a constant $C$ such that, for any $j$, we can select a gauge $\phi_j$, such that, in the square
$Q_{\ell,j}$, we have,
$$\Big|\Ap(x)-\Big(a_{j,2}\Ab_0(x-a_j)-\nabla \phi_j\Big)\Big|\leq C\,\ell^2\,,$$
where $a_j=(a_{j,1},a_{j,2})\,$.

Now, we define the test function as follows,
 \begin{equation}\label{tf:Lto0} v(x)=\left\{
\begin{array}{ll}
e^{i\phi_j(x)}\,u_{r}\big(\sqrt{a_{j,2}}\,(x-a_j)\big)&{\rm if}~a_{j,2}>0 {~\rm and ~}\\
&\hskip1cm x\in Q_{\ell,j}\subset\{|x_1|<R{\rm ~and~}  \frac{\epsilon}2L^{-2/3}<|x_2|< \epsilon L^{-2/3}\}\,,\\
e^{i\phi_j(x)}\,\overline{u_{r}\big( \sqrt{|a_{j,2}|}\,(x-a_j)\big)}&{\rm if ~}a_{j,2}<0{~\rm and ~}\\
&\hskip1cm x\in Q_{\ell,j}\subset\{|x_1|<R{\rm ~and~}  \frac{\epsilon}2L^{-2/3}<|x_2|< \epsilon L^{-2/3}\}\,,\\
0&{\rm otherwise\,,}
\end{array}
\right.
\end{equation}
where the function $u_{r}\in H^1_0(Q_{r})$ is a minimizer  of the ground state energy $F_{b,Q_r}$ introduced  in \eqref{eq:eD} and
$\epsilon\in(0,1)$ is a positive constant (to be determined in \eqref{choiceofepsilon}). \\
 We impose the
following condition on  $m$ and $\epsilon$
\begin{equation}\label{mepsilon}
m\sqrt{\frac\epsilon2}\geq 1\,.
\end{equation}
We will use the notation
$$\mathcal E_{L,R}(v;Q_{\ell,j})=\int_{
Q_{\ell,j}}\left(|(\nabla-i\Ap)v|^2-L^{-2/3}|v|^2+\frac{L^{-2/3}}{2}|v|^4\right)\,dx\,.$$
Notice that, if $a_{j,2}>0$ and $Q_{\ell,j}\subset\{|x_1|<R{\rm
~and~} \frac{\epsilon}2L^{-2/3}<|x_2|<\epsilon L^{-2/3}\}$, then,
{for all $\eta
>0$},
\begin{align*}
&\mathcal E_{L,R}(v;
Q_{\ell,j})\\
&\leq L^{-2/3}\left(\int_{
Q_{\ell,j}}\left(L^{2/3}(1+\eta)|(\nabla-ia_{j,2}\Ab_0(x-a_j))v|^2-L^{-2/3}|v|^2+\frac{L^{-2/3}}{2}|v|^4\right)\,dx\right)\\
&\hskip2cm+C\eta^{-1}\ell^6\\
&=\frac{L^{-2/3}}{a_{j,2}}\left(\int_{
Q_{\sqrt{a_{j,2}}\,\ell}}\left(L^{2/3}a_{j,2}(1+\eta)|(\nabla-i\Ab_0(x))u_r(x)|^2-L^{-2/3}|v|^2+\frac{L^{-2/3}}{2}|v|^4\right)\,dx\right)\\
&\hskip2cm+C\eta^{-1}\ell^6\\
&\leq\frac{L^{-2/3}}{a_{j,2}}\left(g\Big(L^{2/3}a_{j,2}(1+\eta)\Big)|a_{j,2}|\ell^2+C{\sqrt{L^{2/3}a_{j,2}(1+\eta)}}\,\sqrt{a_{j,2}}\,\ell\right)+{C\eta^{-1} \ell^6}\,.
\end{align*}
To write the last inequality, \eqref{eq:g'} is used with
$b=L^{2/3}a_{j,2}(1+\eta)$ and $r=\sqrt{a_{j,2}}\,\ell\,$. (Thanks to the condition \eqref{mepsilon}, we have $r\geq 1\,$). \\
 Similarly,
if $a_{j,2}<0$ and $Q_{\ell,j}\subset\{|x_1|<R{\rm ~and~}
\frac\epsilon2L^{-1/3}<|x_2|<\epsilon L^{-2/3}\}$, then,
$$\mathcal E_{L,R}(v;
Q_{\ell,j})\leq \frac{L^{-2/3}}{|a_{j,2}|}\left(g\Big(L^{2/3}|a_{j,2}|(1+\eta)\Big)|a_{j,2}|\ell^2+C{\sqrt{L^{2/3}(1+\eta)}}\,|a_{j,2}|\,\ell\right)+C\eta^{-1}\ell^6\,.$$
Notice  the simple decomposition of the energy of
$v$,
$$\mathcal E_{L,R}(v)=\sum_{j\in\mathcal J}\mathcal E_{L,R}(v;
Q_{\ell,j})\,,
$$
where $\mathcal J=\big\{j~:~Q_{\ell,j} \subset\{|x_1|<R{\rm ~and~}
\frac\epsilon2L^{-2/3}<|x_2|<\epsilon L^{-2/3}\}\big\}$.\\
 Let $$n={\rm
Card}\,\mathcal J\,.
$$
The numbers   $L$ and $\ell$ are small enough
such that,
$$
 \frac\epsilon4 L ^{-2/3}R\leq n\,\ell^2\leq \frac\epsilon2 L ^{-2/3}R< L^{-2/3}R\,.$$
 Now, we have the following upper bound on the energy of $v\,$,
$$
\mathcal E_{L,R}(v)\leq L^{-2/3}\, \left( \sum_j
g\Big(L^{2/3}|a_{j,2}|(1+\eta)\Big) \ell^2+C \,n \,
\sqrt{L^{2/3}(1+\eta)}\,\ell \right)+C\, n\, \eta^{-1}\ell^6\,.$$
We select $\eta=\frac12$.  Having in mind \eqref{propg} we can select $\epsilon$ sufficiently small such that
\begin{equation}\label{choiceofepsilon}
g(t) \leq -\frac14\,, \quad\forall\, t \in [0,2\epsilon]\,.
\end{equation}
Observing
$$ L^{2/3}\, |a_{j,2}|\, (1+\eta)\leq 2\epsilon\,,
$$
 we get, for $ R \geq  n\, \ell^2 L^\frac 23\,$,
$$
\frac{\er(L;R)}{R}\leq \frac{\mathcal E_{L,R}(v)}{R}\leq
L^{-2/3} \left(-\frac{\epsilon}{16}\, L^{-2/3}+2C L^{1/3} \,\frac{L^{-2/3}}{\ell}\right)+2 C
L^{-2/3}\ell^4\,.$$ Sending $R\to\infty\,$, we deduce that,
$$E(L)\leq -\frac\epsilon{32}L^{-4/3}+C\frac{L^{-1}}{\ell}+CL^{-2/3}\ell^4\,.
$$
Having in mind \eqref{deflm}, we get
\begin{equation}\label{estfin}
E(L)\leq \left( -\frac\epsilon{32}+ \frac{C}{m}\right) L^{-4/3} +C m^4 L^{2/3}\,.
\end{equation}
Recalling \eqref{mepsilon} and  \eqref{choiceofepsilon}, we select $
m $ such that
$$-\frac\epsilon{32}+\frac{C}{m}< 0 \mbox{ and } m >\sqrt{\frac{2}\varepsilon} \,.$$
In that way, \eqref{estfin} gives the claimed  upper bound as $L\rightarrow 0\,$.
\end{proof}

\section{A priori estimates and gauge transformation}

Let $\kappa>0$, $H>0$ and $(\psi,\Ab)$ be a critical point of the
functional in \eqref{eq-3D-GLf}, i.e. $(\psi,\Ab)$ satisfies,
\begin{align}
&-(\nabla-i\kappa H\Ab)^2\psi=\kappa^2(1-|\psi|^2)\psi\,,\label{eq:GL}\\
&-\nabla^\bot\curl(\Ab-\Fb)=\frac1{\kappa H}\left(\overline\psi\,(\nabla-i\kappa H\Ab)\psi\right)\quad{\rm in~}\Omega\,,\label{eq:GL2}
\end{align}
and the two boundary conditions
\begin{align*}
&\nu\cdot(\nabla-i\kappa H\Ab)\psi=0\\
&\curl(\Ab-\Fb)=0\quad\quad{\rm
on}~\partial\Omega\,,\end{align*}
where $\nu$ is the unit exterior normal vector of $\partial\Omega$.

We note for further use the following identity. Multiplying both the equation in \eqref{eq:GL} by $\overline\psi$ then integrating over $\Omega$, we get,
\begin{equation}\label{eq:GLen}
\mathcal E_0(\psi,\Ab;\Omega):=\int_\Omega\left(|(\nabla-i\kappa
H\Ab)\psi|^2-\kappa^2|\psi|^2+\frac{\kappa^2}{2}|\psi|^4\right)\,dx=-
\frac{\kappa^2}2\int_\Omega|\psi|^4\,dx\leq0\,.
\end{equation}

We need the following estimates on
$\psi$ and $\Ab$ that we take from \cite{FH-b}. Earlier versions of
these estimates are given in \cite{LuPa99} when the magnetic field
is constant.

\begin{prop}\label{prop:FH-b}
Let $\alpha\in(0,1)$. There exists a constant $C=C(\alpha,\Omega)>0$
such that, if $\kappa>0$, $H>0$ and $(\psi,\Ab)$ a critical point of
the functional in \eqref{eq-3D-GLf}, then,
\begin{align}
&\|\psi\|_\infty\leq 1\,,\label{eq:psi<1}\\
&\|\curl(\Ab-\Fb)\|_2\leq \frac{C}{H}\,\|\psi\|_2\,,\label{eq:curl}\\
&\|(\nabla-i\kappa H\Ab)\psi\|_2\leq \kappa\, \|\psi\|_2\,,\label{eq:grad<kappa}\\
&\|\Ab-\Fb\|_{C^{1,\alpha}(\overline\Omega)}\leq C\, \frac{1+\kappa H+\kappa^2}{\kappa H}\, \|\psi\|_\infty\, \|\psi\|_2\,.
\label{eq:A-F}
\end{align}
\end{prop}

Using the regularity of the curl-div system, we obtain the following
improved estimates of $\Ab-\Fb$.

\begin{prop}\label{prop:A-F}
Let $\alpha\in(0,1)\,$. There exists a constant $C=C(\alpha,\Omega)>0$
such that, if $\kappa>0\,$, $H>0$ and $(\psi,\Ab)$ a critical point of
the functional in \eqref{eq-3D-GLf}, then,
$$\|\Ab-\Fb\|_{C^{0,\alpha}(\overline\Omega)}\leq
C\left(\|\curl(\Ab-\Fb)\|_2+\frac1{\kappa H}\|(\nabla-i\kappa H\Ab)\psi\|_2\, \|\psi\|_\infty\right)\,.$$
\end{prop}
\begin{proof}
Let $a=\Ab-\Fb$. Notice that $a$ satisfies ${\rm div}\,a=0$  in $\Omega$ and
$\nu\cdot a=0$ on $\partial \Omega\,$. Thus, there exists $C(\Omega) >0$ such that for all $a$ satisfying the previous condition
$$\|a\|_{H^2(\Omega)}\leq C(\Omega) \|\curl a\|_{H^1(\Omega)}\,.$$
Since $(\psi,\Ab)$ is a critical point of the functional in
\eqref{eq-3D-GLf}, then
$$
\nabla^\bot\curl a=\frac1{\kappa H}{\rm
Im}\,(\overline\psi\,(\nabla-i\kappa H)\psi)\,.
$$
Consequently, we get
$$\|a\|_{H^2(\Omega)}\leq
C\, \left(\|\curl(\Ab-\Fb)\|_2+\frac1{\kappa H}\|(\nabla-i\kappa
H\Ab)\psi\|_2\, \|\psi\|_\infty\right)\,.$$ This finishes the proof of
the proposition in light of the {continuous}  embedding of
$H^2(\Omega)$ in $C^{0,\alpha}(\overline\Omega)\,$.
\end{proof}

The next proposition provides us with a useful   gauge
transformation.

\begin{prop}\label{prop:guage} Given $\Omega$ and $B_0$ as in the introduction,
there exists a constant  $C>0$  such that the following is true.
\begin{enumerate}
\item Let $\ell>0$, $a_j\in \Omega$, $D(a_j,\ell)\subset\Omega$ and
$x_j\in D(a_j,\ell)$. There exists a function $\varphi_j\in
C^1(\overline{D(a_j,\ell})$)   such that, for all~$x\in D(a_j,\ell)$\,,
\begin{equation}\label{eq4.3(1)}
| \Fb(x)-\left(B_0(x_j)\Ab_0(x-a_j)+\nabla\varphi_j\right)|\leq C\, \ell^2\,.
\end{equation}
\item Let $\ell>0$, $a_j\in \Gamma$ and $x_j\in
\overline{D(a_j,\ell)}\cap\Gamma$. There exist $\nu_j\in
[0,2\pi)$ and a function $\phi_{j}\in C^1(\overline
{D(a_j,\ell)\cap\Omega})$ such that, , for all~$x\in D(a_j,\ell)\cap \Omega$\,,
\begin{equation} \label{eq24.3(2)}
|\Fb(x)-\left(|\nabla B_0(x_j)|\Ab_{{\rm app},\nu_j}(x-a_j)+\nabla\phi_j\right)|\leq C\, \ell^3\,.
\end{equation}
\end{enumerate}
\end{prop}
\begin{proof}~\\
The function $\varphi_j$ in (1) is constructed in \cite{Att}.\\
We give the construction of the function $\phi_j$ announced in (2).
The vector field $\Fb$ and the function $B_0$ are defined in a
neighborhood of $\overline\Omega$ {(w.l.o.g. we can even assume that
they are defined in $\mathbb R^2$)}.  In particular, $\Fb(x)$ and
$B_0(x)$ are defined in $D(a_j,\ell)$ even when
$D(a_j,\ell)\not\subset\Omega$.

Select $\nu_j\in[0,2\pi)$ such that,
$$\nabla B_0(a_j)=|\nabla B_0(a_j)|(\cos\nu_j,\sin\nu_j)\,.$$
We apply Taylor's formula to the function $B_0$ near $a_j\,$. Since
$a_j\in\Gamma$, we get,
\begin{equation}\label{eq:B0=Aapp'}
B_0(x)=|\nabla B_0(a_j)|(\cos\nu_j,\sin\nu_j)\cdot
(x-a_j)+f_j(x)\,,\end{equation}  where
$$
|f_j(x)|\leq C|x-a_j|^2\leq C\, \ell^2\,,\quad (x\in D(a_j,\ell))\,.
$$
Taylor's formula applied to the function $|\nabla B_0|$ near $a_j$
yields,
$$|\nabla B_0(x_j)|=|\nabla B_0(a_j)|+ e_j\,,$$
where $$|e_j|\leq C|x_j-a_j|\leq C\ell\,.$$
 In that way,
\eqref{eq:B0=Aapp'} becomes,
\begin{equation}\label{eq:B0=Aapp}
B_0(x)=|\nabla B_0(x_j)|(\cos\nu_j,\sin\nu_j)\cdot
(x-a_j)+g_j(x)\,,\end{equation} where
$g_j(x)=f_j(x)+e_j(\cos\nu_j,\sin\nu_j)\cdot (x-a_j)$ and satisfies
$$|g_j(x)|\leq C\, \ell^2\,.$$
Define the vector field:
\begin{align*}
&\mathbf{G}_j(y)=\left(\displaystyle\int_0^1 sg_j(sy+a_j)\,ds\right)\,(-y_2,y_1)\,,\, \mbox{ for }  y=(y_1,y_2)\,.
\end{align*}
Clearly, $|\mathbf G_j(y)|\leq C\ell^3\,$, when $y\in D(0,\ell)$ and
$y+a_j\in\Omega\,$.

We perform the translation $y=x-a_j$ and define
$$
\widetilde \Fb(y)=\Fb(y+a_j)\,,\, \mbox{ for } y\in D(0,\ell)\,.$$
In that
way, the formula in \eqref{eq:B0=Aapp} reads as follows,
$$\curl\left(\widetilde \Fb-|\nabla B_0(x_j)|\Ab_{\rm app,\nu_j}\right)=\curl
{\mathbf G}_j\quad{\rm in}\quad D(0,\ell)\,,$$
where $\Ab_{\rm app,\nu_j}$ is introduced in \eqref{eq-A-app-nu}. \\
 Consequently, we
deduce the existence of  a function $\widetilde \phi_j\in
C^1(\overline{D(0,\ell)})$ such that,
$$\widetilde \Fb-|\nabla B_0(x_j)|\Ab_{\rm app,\nu_j}={\mathbf G}_j+\nabla\widetilde\phi_j\,,\quad{\rm in~}D(0,\ell)\,.$$ The
function $\phi_j$ is defined by $\phi_j(x)=\widetilde\phi_j(x-a_j)\,$, for
$x\in D(a_j,\ell)$.
\end{proof}

\section{Energy upper bound}
 For the statement of the next proposition, we introduce  the quantity
$\psi(\mu_1,\mu_2,\mu_3,\rho)$ which is defined for $\mu_1>0$, $\mu_2>0$,
$\mu_3>0$  and $\rho\in(0,1)$ by
$$
\psi(\mu_1,\mu_2,\mu_3, \rho) := \max \left(   2^{1/\rho} \mu_1^{\frac{1-\rho}{3\rho}}\,,\, 2^{1/(1-\rho)} \mu_2^{-\frac{4-\rho}{3(1-\rho)}}\,,\,2\mu_3^{-1/3}   \right)
$$
\begin{prop}\label{prop:ub}
Let $\Lambda>0$, $\eta\in(0,1/2)$ and $b:\R\to\R_+$ such that
$\displaystyle\lim_{\kappa\to \infty}b(\kappa)= \infty$
 and $\displaystyle\lim_{\kappa\to \infty}\kappa^{-1/2}b(\kappa)=0\,$. There exist positive constants
 $C$,
$ \kappa_0$  and $\ell_0$ such that if $\rho\in\,(0,1)\,$,
$\ell\in\,(0,\ell_0)\,$, $\delta\in\,(0,1)\,$,
$\kappa \geq \kappa_0\,$,
\begin{equation}\label{condiprop}
\kappa\ell\geq \psi\left(\frac{H}{\kappa^2}\sup_{x\in\Gamma}|\nabla
B_0(x)|\,,\frac{H}{\kappa^2}\inf_{x\in\Gamma}|\nabla
B_0(x)|\,,\, \frac{H}{\kappa^2}\inf_{x\in\Gamma}|\nabla
B_0(x)| \,,\,\rho\right)\,,
\end{equation}
 and
$$ b(\kappa)\kappa^{3/2} \leq H\leq \Lambda\kappa^2\,,$$
 then the
ground state energy in \eqref{eq-gse} satisfies,
\begin{equation}\label{eq:ub1}
\E \leq \kappa \int_{\Gamma}\left(|\nabla
B_0(x)|\frac{H}{\kappa^2}\right)^{1/3}\,E\left(|\nabla
B_0(x)|\frac{H}{\kappa^2}\right)\,ds(x)+ a + b  \,,
\end{equation}
where
\begin{itemize}
\item  $ds$ is the arc-length measure on $\Gamma$,
\item
\begin{equation}\label{deflambda}
\check{\lambda}=\left(\sup_{x\in\Gamma}|\nabla
B_0(x)|\right)\frac{H}{\kappa^2}\,,
\end{equation}
\item
$ a  = C\left( \check{\lambda}^{-5/9}\kappa^{1/3}\ell^{-2/3}+
\check{\lambda}^{-4/3}|\ln\check{\lambda}|^{-1}|\ln(\check{\lambda}
\kappa^{-1}\ell^{-1})|^2\kappa^\rho\ell^{\rho-1}+
\left(\delta\kappa^2+\delta^{-1}\kappa^2 H^2\ell^6\right)\ell
\right)$
\item
$b= C\left(\eta+\delta+
\check{\lambda}^{\frac{2(1-\rho)}3}\kappa^{-2\rho}\,\ell^{-2\rho}\right)\displaystyle\frac{\kappa^3}{H}
$\,.
\end{itemize}
\end{prop}~
\begin{proof}~\\
{\bf Step~1.} {\it Existence of $\ell_0$.}

Recall the assumption that $\Gamma$ is the union of a finite number
of simple smooth curves and $\Gamma\cap\partial\Omega$ is a finite
set. Given $\eta>0$, there exists a constant $\ell_1\in(0,1)$ such
that, for all $a\in \Gamma$ and $\ell\in(0,\ell_1)$ with
$\overline{D(a,\ell)}\subset\Omega$, then $D(a,\ell) \cap \Gamma$ is
connected and \begin{equation}\label{eq:arc=euc}
2\ell-\frac{\eta}2\ell\leq \int_{D(a,\ell)\cap\Gamma}ds(x)\leq
2\ell+\frac{\eta}2\ell\,.\end{equation}
Notice that $\displaystyle\int_{D(a,\ell)\cap\Gamma}ds(x)$ is the
arc-length (along $\Gamma$) of $D(a,\ell)\cap\Gamma$. Thus, the
choice of $\ell_1$ is such that the arc-length of
$D(a,\ell)\cap\Gamma$ is approximately $2\ell$, whenever
$\ell\in(0,\ell_1)$.

The arc-length measure of $\Gamma$ is denoted by $|\Gamma|$. By
assumption, $\Gamma$ consists of a finite number of simple smooth
curves $(\Gamma_i)_{i=1}^k$. Let
$$\ell_0=\min\left(\frac{\eta}{16}\min_{1\leq i\leq k}|\Gamma_i|\,,\,\frac{\ell_1}{16}\left(1+\frac{\eta}4\right)^{-1}\right)\,.$$
If $\ell\in(0,\ell_0)$, then $\ell<\ell_1$, \eqref{eq:arc=euc} is
satisfied and
\begin{equation}\label{eq:ell0}
\frac{2\ell}{|\Gamma_i|}<\frac{\eta}4\,,\quad
\left(1+\eta\right)\ell<\frac{\ell_1}4\,.\end{equation}

{\bf Step 2.} {\it A covering of $\Gamma$.}

In the sequel, we suppose that $\ell\in(0,\ell_0)$.  Consider
$i\in\{1,\cdots,k\}$ and the curve $\Gamma_i$. Let $\n_i\in\mathbb
 N$ be the unique natural number satisfying
\begin{equation}\label{eq:ni}
\frac{|\Gamma_i|}{2\ell}\left(1+\frac{\eta}4\right)^{-1}-1<\n_i\leq\frac{|\Gamma_i|}{2\ell}\left(1+\frac{\eta}4\right)^{-1}\,.
\end{equation}
Select $\n_i$ distinct points $(b_{j,i})_j$ on $\Gamma_i$ such that,
$$\forall~j\,,\quad{\rm dist}_{\Gamma_i}(b_{j,i},b_{j+1,i})=\frac{|\Gamma_i|}{\n_i}\,,$$
where ${\rm dist}_{\Gamma_i}$ is the arc-length measure on
$\Gamma_i$.

Obviously, the Euclidean distance $e_j:=|b_{j+1,i}-b_{j,i}|$
satisfies $e_j\leq {\rm
dist}_{\Gamma_i}(b_{j,i},b_{j+1,i})=\frac{|\Gamma_i|}{\n_i}$. Thanks
to \eqref{eq:ni} and \eqref{eq:ell0}, we have,
$$e_j\leq 2\ell\left(1+\eta\right)<\ell_1\,.$$
Thus, if $\overline{D(b_{j,i},e_j)}\subset \Omega$, we can use
\eqref{eq:arc=euc} with $\ell=e_j$ and get,
$$2e_j\left(1-\frac{\eta}4\right)\leq 2\frac{|\Gamma_i|}{\n_i}\leq
2e_j\left(1+\frac{\eta}4\right)\,.$$ Thanks to \eqref{eq:ni}, this
leads to
$$e_j\geq \frac{|\Gamma_i|}{\n_i}\left(1-\frac{\eta}4\right)\geq2\ell\,.$$
Now, define the index set
$$\mathcal J_i=\{j~:~\overline{D(b_{j,i},e_j)}\subset\Omega\}\,,$$
and $N_i={\rm Card}\,\mathcal J_i$. Notice that, if $j\in\mathcal
J_i$, then $e_j\geq 2\ell$ and $\overline{D(b_{j,i},\ell)}\subset
\overline{D(b_{j,i},e_j/2)}\subset\Omega$. The sets
$(D(b_{j,i},\ell))_{j\in\mathcal J_{i}}$ are pairwise disjoint.

Since $\Gamma_i\cap\partial\Omega$ is a finite set, then there
exists a constant $c$ such that,
$${\rm If}~a\in\Gamma_i\; {\rm and}~{\rm dist}(a,\partial\Omega)\geq
c\ell\,,{~\rm then~}\overline{D(a,\ell)}\subset\Omega\,.$$ Consequently, the number $N_i$ satisfies
$$\n_i-C\leq N_i\leq \n_i\,,$$
where $C>0$ is a constant.  Thus, thanks to \eqref{eq:ni} and
\eqref{eq:ell0},
$$|\Gamma_i|\left(1+\frac{\eta}4\right)^{-1}-C\ell\leq N_i\times 2\ell\leq
|\Gamma_i|\left(1+\frac{\eta}4\right)^{-1}\,.$$
Now, collecting the points $(b_{j,i})_{j\in\mathcal
J_i,i\in\{1,\cdots,k\}}$, we get the collection of points on
$\Gamma$,
$$(a_j)_{j\in\mathcal J}=(b_{j,i})_{j\in\mathcal
J_i,i\in\{1,\cdots,k\}}\,,$$ such that,
\begin{align}
&\forall~j\in\mathcal J\,,\quad a_j\in\Gamma\quad {\rm and}\quad \overline{D(a_j,\ell)}\subset\Omega\,,\nonumber\\
&N={\rm Card}\,\mathcal J=\sum_{i=1}^kN_i\quad{\rm and}\quad |\Gamma|=\sum_{i=1}^k|\Gamma_i|\,,\nonumber\\
&|\Gamma|\left(1+\frac{\eta}4\right)^{-1}-C\ell\leq N\times 2\ell\leq |\Gamma|\left(1+\frac{\eta}4\right)^{-1}\,.\label{estN}
\end{align}
Notice that
$$\bigcup_j\Big(\Gamma\cap\overline{D(a_j,\ell)}\Big)\subset \Gamma \,.$$
and the arc-length measure
$$\left|\bigcup_j\Big(\Gamma\cap\overline{D(a_j,\ell)}\Big)\right|=\int_{\bigcup_j\Big(\Gamma\cap\overline{D(a_j,\ell)}\Big)}ds(x)
=\sum_j\int_{\Gamma\cap\overline{D(a_j,\ell)}}ds(x)$$
satisfies
$$|\Gamma|-C\eta
\leq 
\left|\bigcup_j\Big(\Gamma\cap\overline{D(a_j,\ell)}\Big)\right|\leq 
 |\Gamma|+C\eta\,.$$

Thus, the arc-length measure of the set
$\Gamma\setminus\bigcup_j\Big(\Gamma\cap\overline{D(a_j,\ell)}\Big)$
satisfies,
\begin{equation}\label{eq:lenG}
\left|\Gamma\setminus\bigcup_j\Big(\Gamma\cap\overline{D(a_j,\ell)}\Big)\right|\leq
C\eta\,.\end{equation}

{\bf Step~3.} {\it Construction of a test configuration.}  For each
$j$, select an arbitrary point  $x_j\in \overline{D(a_j,\ell)}\cap
\Gamma$ and write
$$\nabla B_0(x_j)=|\nabla B_0(x_j)|(\cos\nu_j,\sin\nu_j)\,,$$
with $\nu_j\in[0,2\pi)$.\\
 Define
\begin{equation}\label{eq:L,R}
L=L_j=|\nabla B_0(x_j)|\,\frac{H}{\kappa^2}\,,\quad R=R_j=L^{1/3}\kappa\ell\,.
\end{equation}
Thanks to the assumption in \eqref{condiprop}, there holds the
condition,
\begin{equation}\label{eq:condR}
R \geq \max\left(2^{1/\rho}L^{1/(3\rho)}\,,\,
2^{1/(1-\rho)}L^{-1/(1-\rho)}\,,2\right)\,.
\end{equation}
%
 We define a function $w\in
H^1(\Omega)$ as follows. Consider the set of indices $\mathcal
J=\{j~:~D(a_j,\ell)\subset\Omega\}\,$. Let $x\in\Omega\,$ and
$j\in\mathcal J$. If $x\in D(a_j,\ell)\,$, define,
\begin{equation}\label{eq:fw}
w(x)=e^{i\kappa H\phi_j}u_{L,\ell,\nu_j}(x-a_j)\,,
\end{equation}
where $u_{R,L,\nu_j}\in H^1_0(D(0,\ell))$ is a minimizer of the
functional in \eqref{eq-GL-r-new} with $\nu=\nu_j\,$, and $\phi_j$
is the function constructed in Proposition~\ref{prop:guage}.  If
$x\not\in\displaystyle\bigcup _{j\not\in\mathcal J} \overline{
D(a_j,\ell)}\,$, we set $w(x)=0\,$. \\ Clearly, $w\in
H^1(\Omega)\,$.

{\bf Step~4.} {\it Upper bound of $\mathcal E (w,\Fb)$.}

Notice that $\curl\Fb=B_0$ and that  the magnetic energy term in
\eqref{eq-3D-GLf} vanishes for $\Ab=\Fb$. Thus, we have,
\begin{equation}\label{eq:dec;en(w,A)}
\mathcal E(w,\Fb)=\mathcal E_0(w,\Fb;\Omega)=\sum_j \mathcal
E_0(w,\Fb; D(a_j,\ell))\,,\end{equation} where the functional
$\mathcal E_0$ is defined in \eqref{eq:E0}.

Recalling the definition of $w$,  we observe that,
$$\mathcal E_0(w,\Fb; D(a_j,\ell))=
\mathcal E_0(u_{L,\ell,\nu_j}(x-a_j),\Fb-\nabla\phi_j;
D(a_j,\ell))\,.$$ Thanks to the choice of $\phi_j$, we infer from Proposition~\ref{prop:guage},
\begin{equation}
\left|\,|\nabla B_0(x_j)|\Ab_{{\rm
app},\nu_j}(x-a_j)-(\Fb-\nabla\phi_j)\right|\leq  C\,
\ell^3\,.\end{equation}
As a consequence, applying the Cauchy-Schwarz inequality, we get
that,  for any $\delta >0\,$,
$$\mathcal E_0(w,\Fb; D(a_j,\ell))\leq
(1+\delta)\mathcal E_0(u_{L,\ell,\nu_j}(x-a_j), |\nabla B_0(x_j)|\Ab_{{\rm app},\nu_j}(x-a_j);D(a_j,\ell))
+r_1\,,$$
 where
$$r_1=C\left(\delta\kappa^2+
\delta^{-1}\kappa^2H^2\ell^6\right)\int_{D(a_j,\ell)}|u_{L,\ell,\nu_j}(x-a_j)|^2\,dx
\,.$$  Recall that $u_{L,\ell,\nu_j}$ being a minimizer, it
satisfies
 $$|u_{L,\ell,\nu_j}|\leq 1\,.
 $$
  Thus,
\begin{equation}\label{eq:r1} r_1\leq C\left(\delta\kappa^2+
\delta^{-1}\kappa^2H^2\ell^6\right)\ell^2\,.\end{equation}
Now,  performing the translation $x\mapsto x-a_j$, we observe that,
$$\mathcal E_0(w,\Fb; D(a_j,\ell))\leq
(1+\delta)\mathcal E_0(u_{L,\ell,\nu_j},|\nabla B_0(x_j)|\Ab_{{\rm
app},\nu_j}; D(0,\ell))+r_1\,.$$
 With $L=L_j$ and $R=R_j$  in
\eqref{eq:L,R}, we get in light of Theorem~\ref{thm-tdl-e}:
%
%
$$
\mathcal E_0(w,\Ab; D(a_j,\ell))\leq
(1+\delta)\er\left(L_j;\big(1-L_j^{2/3}R_j^{-2\rho}\big)R_j\right)+CL_j^{-4/3}|\ln L_j|^{-1}|\ln(L_j^{1/3}R_j^{-\rho})|^2R_j^\rho+r_1\,.
$$
Thanks to Theorem~\ref{thm-FK}, we deduce that,
\begin{multline}\label{eq:D(0,ell)}
\mathcal E_0(w,\Fb; D(a_j,\ell))\leq 2 (1+\delta)\,
\big(1-L_j^{2/3}R_j^{-2\rho}\big)R_j\,E(L_j)+C\,
(1+L_j^{-2/3})R_j^{1/3}\\+ CL_j^{-4/3}|\ln
L_j|^{-1}|\ln(L_j^{1/3}R_j^{-\rho})|^2R_j^\rho+r_1\,.
\end{multline}
Recall the definition of $L_j$ and $R_j$ in \eqref{eq:L,R}, and that
the number of disks $D(a_j,\ell)$  is inversely proportional to
$\ell$, i.e. of order $\ell^{-1}\,$.

In the sequel, the following remark will be used. The two terms
$$\check\lambda=\left(\sup_{x\in\Gamma}|\nabla B_0(x)|\right)\frac{H}{\kappa^2}\quad{\rm and}\quad
\underline\lambda=\left(\inf_{x\in\Gamma}|\nabla
B_0(x)|\right)\frac{H}{\kappa^2}$$ are of the same order, i.e. there
exist constants $c_1>0$ and $c_2>0$ such that $c_1\check\lambda\leq
\underline\lambda\leq c_2\check\lambda$. Thus, terms controlled by
$\underline \lambda$ are controlled by $\check\lambda$ and vice
versa.  We will express all terms controlled by $\underline\lambda$
and $\check\lambda$ in the form $\mathcal O(\check\lambda)$.

Substituting \eqref{eq:D(0,ell)} into \eqref{eq:dec;en(w,A)} yields,
\begin{multline}\label{eq:ub}
\mathcal E(w,\Fb)\leq 2\kappa \ell(1+\delta)
\big(1-\check{\lambda}^{\frac{2(1-\rho)}{3}}\kappa^{-2\rho}\ell^{-2\rho}\big)\left(\sum_j\left(|\nabla
B_0(x_j)|\,\frac{H}{\kappa^2}\right)^{1/3}\,
E\left(|\nabla B_0(x_j)|\frac{H}{\kappa^2}\right)\right) \\
+C\, \check{\lambda}^{-5/9}
\kappa^{1/3}\ell^{-2/3}+C\check{\lambda}^{-4/3}|\ln\check{\lambda}|^{-1}|\ln(\check{\lambda}
\kappa^{-1}\ell^{-1})|^2\kappa^\rho\ell^{\rho-1}+C\, r_1\ell^{-1}\,.
\end{multline}
Thanks to \eqref{estN} and the upper bound on $E(\cdot)$ obtained in
Theorem~\ref{thm:Lto0}, the term
\begin{equation}\label{eq:Rsum}
\sum_j2\ell\left(|\nabla
B_0(x_j)|\,\frac{H}{\kappa^2}\right)^{1/3}\, E\left(|\nabla
B_0(x_j)|\frac{H}{\kappa^2}\right)\end{equation} is of order
$\kappa^2/H$. Thus, \eqref{eq:ub} becomes,
\begin{multline}\label{eq:ub'}
\mathcal E(w,\Fb)\leq \kappa  \left(\sum_j2\ell\left(|\nabla
B_0(x_j)|\,\frac{H}{\kappa^2}\right)^{1/3}\,
E\left(|\nabla B_0(x_j)|\frac{H}{\kappa^2}\right)\right) +C\Big(\delta+\check{\lambda}^{\frac{2(1-\rho)}{3}}\kappa^{-2\rho}\ell^{-2\rho}\Big)\frac{\kappa^3}{H}\\
+C\, \check{\lambda}^{-5/9}
\kappa^{1/3}\ell^{-2/3}+C\check{\lambda}^{-4/3}|\ln\check{\lambda}|^{-1}|\ln(\check{\lambda}
\kappa^{-1}\ell^{-1})|^2\kappa^\rho\ell^{\rho-1}+C\, r_1\ell^{-1}\,.
\end{multline}
In \eqref{eq:Rsum}, replacing $2\ell$ by the arc-length measure of
$D(a_j,\ell)\cap\Gamma$ produces an error $\eta\ell/2$ and the sum
becomes a Riemann sum
over $V_\ell =\bigcup_{j\in\mathcal J} \big(\Gamma\cap \overline{D(a_j,\ell)}\big)$.\\
The points $x_j$ can be selected such that the Riemann sum is a {\it
lower} Riemann sum. Thus,
\begin{multline*}
\sum_j2\ell\left(|\nabla
B_0(x_j)|\,\frac{H}{\kappa^2}\right)^{1/3}\, E\left(|\nabla
B_0(x_j)|\frac{H}{\kappa^2}\right)\\
\leq \int_{V_\ell}\left(|\nabla
B_0({x})|\,\frac{H}{\kappa^2}\right)^{1/3}\, E\left(|\nabla
B_0({x})|\frac{H}{\kappa^2}\right)\,ds(x)+C\eta\frac{\kappa^2}H\,.
\end{multline*}
Inserting this into \eqref{eq:ub'}, we get,
\begin{multline*}
\mathcal E(w,\Fb)\leq \kappa\, \left(\int_{V_\ell}\left(|\nabla
B_0({x})|\,\frac{H}{\kappa^2}\right)^{1/3}\,
E\left(|\nabla B_0({x})|\frac{H}{\kappa^2}\right)\,ds(x)\right)+C\Big(\eta+\delta+\check{\lambda}^{\frac{2(1-\rho)}{3}}\kappa^{-2\rho}\ell^{-2\rho}\Big)\frac{\kappa^3}{H}\\
+C{\check{\lambda}^{-5/9}}\kappa^{1/3}\ell^{-2/3}+C\check{\lambda}^{-4/3}|\ln\check{\lambda}|^{-1}|\ln(\check{\lambda}
\kappa^{-1}\ell^{-1})|^2\kappa^\rho\ell^{\rho-1}+Cr_1\ell^{-1}\,.
\end{multline*}
As pointed earlier, the arc-length measure of the set
$\Gamma\setminus V_\ell $ does not exceed $C\eta$. Recall the upper
bound on $E(\cdot)$ obtained in Theorem~\ref{thm:Lto0}. In that way,
we get,
$$\left|\int_{\Gamma\setminus V_\ell}\left(|\nabla
B_0({x})|\,\frac{H}{\kappa^2}\right)^{1/3}\, E\left(|\nabla
B_0({x})|\frac{H}{\kappa^2}\right)\,ds(x)\right|\leq
C\eta\frac{\kappa^2}{H}\,.$$
Consequently, we deduce the following
upper bound,
\begin{multline*}
\mathcal E(w,\Fb)\leq \kappa\, \left(\int_\Gamma\left(|\nabla
B_0({x})|\,\frac{H}{\kappa^2}\right)^{1/3}\,
E\left(|\nabla B_0({x})|\frac{H}{\kappa^2}\right)\,ds(x)\right)+C\Big(\eta+\delta+\check{\lambda}^{\frac{2(1-\rho)}{3}}\kappa^{-2\rho}\ell^{-2\rho}\Big)\frac{\kappa^3}{H} \\
+C{\check{\lambda}^{-5/9}}\kappa^{1/3}\ell^{-2/3} +C\check{\lambda}^{-4/3}|\ln\check{\lambda}|^{-1}|\ln(\check{\lambda}
\kappa^{-1}\ell^{-1})|^2\kappa^\rho\ell^{\rho-1}+Cr_1\ell^{-1}+C\eta\frac{\kappa^3}{H}\,.
\end{multline*}

The definition of the ground state energy in \eqref{eq-gse} tells us
that $\E\leq \mathcal E(w,\Fb)$. Recalling the definition of $r_1$
in \eqref{eq:r1} finishes the proof of \eqref{eq:ub1}.
\end{proof}

As a straightforward application of Proposition~\ref{prop:ub}, we
deduce the following upper bound on the ground state energy in
\eqref{eq-gse}, valid in the regime $\kappa^{3/2}\ll H\lesssim
\kappa^2$.

\begin{thm}\label{thm:up}
Let $\Lambda>0$ and $\epsilon:\R\to\R_+$ such that
$\displaystyle\lim_{\kappa\to \infty}\kappa\epsilon(\kappa)= \infty$
 and $\displaystyle\lim_{\kappa\to \infty}\epsilon(\kappa)=0\,$.
If $\epsilon(\kappa)\kappa^2 \leq H\leq \Lambda\kappa^2$, then the
ground state energy in \eqref{eq-gse} satisfies,
\begin{equation}\label{eq:ub1*}
\E \leq \kappa\left(\int_{\Gamma}\left(|\nabla
B_0(x)|\frac{H}{\kappa^2}\right)^{1/3}\,E\left(|\nabla
B_0(x)|\frac{H}{\kappa^2}\right)\,ds(x)\right)+\frac{\kappa^3}{H}\,o(1)
\,,\quad(\kappa\to \infty)\,,
\end{equation}
where $ds$ is the arc-length measure on $\Gamma$.
\end{thm}
\begin{proof}
We use the upper bound in Proposition~\ref{prop:ub} with the
following choice of the parameters:
\begin{equation}\label{eq:ub-ell}\left\{
\begin{array}{l}
0<\rho<\frac1{8}\,,\quad \ell=
\epsilon_1^{1/8}\kappa^{-3/4}\check{\lambda}^{-1/2}\,,\quad
\delta=\epsilon_1^{7/8}|\ln\kappa|\, \kappa^{-1/4}\check{\lambda}^{-1/2}\,,\\ \epsilon_1=
\kappa^{-2/7}|\ln\kappa|^{-8/7}\,.
\end{array}
\right.
\end{equation}
Clearly, $\epsilon_1$ satisfies,
$$\epsilon_1\ll1\quad{\rm as ~}\kappa\to\infty\,.$$
Recall that $\check{\lambda}$ (introduced in \eqref{deflambda})
satisfies
$$ \kappa^{-1/2}\ll\check{\lambda}\lesssim 1\,.$$
In that way, the parameters $\delta$ and $\ell$ satisfy,
\begin{align*}
&\delta\ll\epsilon_1^{7/8}|\ln\kappa\, |\kappa^{1/4}=1\,,\\
&\ell\ll\epsilon^{1/8}\kappa^{-3/4}\kappa^{1/2}=\epsilon^{1/8}\kappa^{-1/4}\ll1\,.
\end{align*}
Let us show that the three conditions in \eqref{condiprop} are
satisfied. We rewrite
$$
\kappa\ell=\epsilon_1^{1/8}\kappa^{1/4}\check\lambda^{-1/2}=\kappa^{3/14}|\ln\kappa|^{-1/7}\check\lambda^{-1/2}
=\kappa^{\varsigma/14}|\ln\kappa|^{-1/7}\kappa^{(3-\varsigma)/14}\check\lambda^{-1/2}\,,
$$
with $\varsigma:=\frac{1-8\rho}{3(1-\rho)}\,$.\\ Thanks to the
condition $0<\rho<\frac18$, we have $0<\varsigma<1\,$.\\
The conditions in \eqref{condiprop} will follow from
$$
\kappa \ell \gg \max \{ \check \lambda^{\frac{1-\rho}{2 \rho}}, \check \lambda^{-\frac{4-\rho}{3 (1-\rho)}}, \check \lambda^{-\frac 13}\}\,.
$$
We first observe
$$
\kappa \ell
\gg\check\lambda^{-(3-\varsigma)/7}\check\lambda^{-1/2}=\check\lambda^{-\frac{37-23\rho}{14\times
3(1-\rho)}}\gtrsim 1\,.$$ Then, we deduce,
\begin{align*}
&\kappa\ell\,\check\lambda^{\frac{4-\rho}{3(1-\rho)}}\gg
\check\lambda^{\frac{-29+21\rho}{14\times 3(1-\rho)}}\gtrsim 1\,,\\
&\kappa\ell\,\check\lambda^{1/3}\gg \check\lambda^{\frac{-23+9\rho}{14\times3(1-\rho)}}\gtrsim 1\,,
\end{align*}
and we can apply Proposition~\ref{prop:ub}.

Now, the remainder $a$ in Proposition~\ref{prop:ub}
satisfies\footnote{We recall that $ \kappa \check{\lambda}^{-1}$ is
of the order of $\kappa^3/H$.}
$$a\ll \kappa \check{\lambda}^{-1}\,.
$$
Notice indeed that,
\begin{align*}
&\delta^{-1}\kappa^2H^2\ell^7\, {\approx}\, |\ln\kappa|^{-1}\kappa\check\lambda^{-1}\ll\kappa\check\lambda^{-1}\,,\\
&\delta\kappa^2\ell=\epsilon_1|\ln\kappa|\kappa\check\lambda^{-1}\ll\kappa\check\lambda^{-1}\,,\\
&|\ln(\check\lambda\kappa\ell)|^2\kappa^\rho\ell^{\rho-1}=|\ln(\check\lambda\kappa\ell)|^2({\epsilon_1^\frac 18 } \kappa^{1/4})^{\rho-1}\check\lambda^{\frac{1-3\rho}2}\kappa\check\lambda^{-1}
\lesssim|\ln(\check\lambda\kappa\ell)|^2({\epsilon_1^{\frac 18}} \kappa^{1/4})^{\rho-1}\kappa\check\lambda^{-1}\ll\kappa\check\lambda^{-1}\,,\\
&\check\lambda^{-5/9}\kappa^{1/3}\ell^{-2/3}=\check\lambda^{7/9}\kappa^{-\frac{5}{6\times7}}|\ln\kappa|^{\frac2{3\times
7}}\kappa\check\lambda^{-1}\ll\kappa\check\lambda^{-1}\,.
\end{align*}
The last point to verify is that
$$\check{\lambda}^{\frac{2(1-\rho)}3}\kappa^{-2\rho}\,\ell^{-2\rho}\ll1\,,$$
which follows easily since we know that
$$1-\rho>0\,,\,\check\lambda\lesssim1 \mbox{ and  and
}\kappa\ell\gg1\,.$$ Now sending $\kappa\to\infty$, the upper bound
in Proposition~\ref{prop:ub} becomes,
$$\limsup_{\kappa\to\infty}\frac{H}{\kappa^3}\Big\{\E -2\kappa\left(\int_{\Gamma}\left(|\nabla
B_0(x)|\frac{H}{\kappa^2}\right)^{1/3}\,E\left(|\nabla
B_0(x)|\frac{H}{\kappa^2}\right)\,ds(x)\right)\Big\}\leq C\eta\,.$$
Since this is true for all $\eta\in(0,1/2)$, we get by sending
$\eta\to0_+\,$,
$$\limsup_{\kappa\to\infty}\frac{H}{\kappa^3}\Big\{\E -2\kappa\left(\int_{\Gamma}\left(|\nabla
B_0(x)|\frac{H}{\kappa^2}\right)^{1/3}\,E\left(|\nabla
B_0(x)|\frac{H}{\kappa^2}\right)\,ds(x)\right)\Big\}\leq0\,,$$ and
the conclusion in Theorem~\ref{thm:up} follows.
\end{proof}

In the next theorem, we derive an upper bound of the ground state
energy in \eqref{eq-gse} valid in the regime $\kappa^{-1}\ll
H\lesssim\kappa^{3/2}$.

\begin{thm}\label{thm:up*}
Let $\Lambda>0$ and $\epsilon:\R+\to\R_+$ be a function satisfying
$\displaystyle\lim_{\kappa\to\infty}\epsilon(\kappa)=0$ and
$\displaystyle\lim_{\kappa\to\infty}\kappa\epsilon(\kappa)=\infty$.

If $\epsilon(\kappa)\kappa\leq H\leq \Lambda\kappa^{3/2}$, then the
ground state energy in \eqref{eq-gse} satisfies,
\begin{equation}\label{eq:ub2}
\E \leq \kappa^2\int_\Omega
g\left(\frac{H}{\kappa}|B_0(x)|\right)\,dx+
\frac{\kappa^3}{H}\,o(1)\,,
\end{equation}
where $g(\cdot)$ is the function introduced in \eqref{eq:g}.
\end{thm}
\begin{proof}
Here, we construct a test function as in \eqref{eq:fw} below. Let
$\zeta=\kappa^{-1/16} H^{-1/2}$ and $(\mathcal Q_{k,\zeta})$ be the
lattice of squares generated by the square
$$\mathcal Q_\zeta=\left(\,-\frac\zeta2,\frac\zeta2\,\right)\times
\left(\,-\frac\zeta2,\frac\zeta2\,\right)\,.$$ Notice that $\zeta$ satisfies $\zeta\ll\frac{\kappa}{H}\ll1$. Define
$$\mathcal I=\big\{k~:~\mathcal
Q_{k,\zeta}\subset\{{\rm dist}(x,\Gamma)\leq
M\frac{\kappa}{H}\}\setminus\big(\bigcup_{j\in\mathcal
J}D(a_j,\ell)\big)\quad{\rm and}\quad{\rm dist}
(Q_{k,\zeta},\Gamma)\geq M\zeta\big\}\,,$$ where $M>0$ is a constant
selected sufficiently large so that, if ${\rm
dist}(x,\Gamma)>M\frac{\kappa}{H}$, then $|B_0(x)|\geq \frac\kappa
H$. Notice that, since $B_0$ vanishes non-degenerately on $\Gamma$,
if $k\in\mathcal I$, then  $$ |B_0(x)|\geq M'\zeta>0 \mbox{ in }  \overline
{\mathcal Q_{k,\zeta}}\,.$$
For all $k\in \mathcal I$, let $a_k$ be the center of the square
$\mathcal Q_{k,\zeta}$ and select an arbitrary point
$x_k\in\overline{\mathcal Q_{k,\zeta}}$.

If $r>0$ and $b>0$, let $u_r\in H^1_0(\mathcal Q_{r})$ be the
minimizer of the ground state energy $F_{b,\mathcal Q_r}$ introduced
in \eqref{eq:eD}. For all $k\in\mathcal I$, let
$r_k=\zeta\sqrt{\kappa H |B_0(x_k)|}$, $b_k=\frac
H\kappa|B_0(x_k)|$, $u_k=u_{r_k}$ and $\varphi_k$ be the gauge
function satisfying (see Proposition~\ref{prop:guage}),
$$|\Fb(x)-\big(B_0(x_k)\Ab_0(x-a_k)+\nabla\varphi_k\big)|\leq
C\zeta^2\,,\quad{\rm in~}\overline{\mathcal Q_{k,\zeta}}\,.$$
Define the test function $v$ as follows,
\begin{equation}\label{eq:fv}
v(x)=\left\{\begin{array}{ll}
e^{-i\phi_k(x)}\,u_{k}\big(\frac{r_k}{\zeta}\,(x-a_k)\big)&{\rm if}~x\in\mathcal Q_{k,\zeta}\subset\{B_0(x)>0\}\,,\\
e^{-i\phi_j(x)}\,\overline{u_{k}}\big(\frac{r_k}{\zeta}\,(x-a_k)\big)&{\rm if}~x\in\mathcal Q_{k,\zeta}\subset\{B_0(x)<0\}\,,\\
0&{\rm otherwise\,.}
\end{array}
\right.
\end{equation}
We outline the computation of $\mathcal E_0(v,\Fb)$. The details of
the computations are given in \cite{Att}. In every square $\mathcal
Q_{k,\zeta}$, we have,
$$\mathcal E_0(v,\Fb;\mathcal Q_{k,\zeta})\leq (1+\kappa^{-1/16})\,\frac{F_{b_k,\mathcal
Q_{r_k}}(u_k)}{b_k}+\kappa^{1/16}\kappa^2H^2\zeta^6\,.$$ Thanks to
the assumption on $H$ and the definition of
$\zeta=\kappa^{-1/16}H^{-1/2}$, we have $r_k\gg 1$. Thus, we may use
 \eqref{eq:g'} and write,
\begin{align*}\mathcal E_0(v,\Fb;\mathcal Q_{k,\zeta})&\leq
(1+\kappa^{-1/16})\,\frac{r_k^2}{b_k}\left(g(b_k)+C\frac{\sqrt{b_k}}{r_k}\right)+\kappa^{1/16}\kappa^2H^2\zeta^6\\
&=(1+\kappa^{-1/16})\zeta^2\kappa^2\left(g\left(\frac H\kappa |B_0(x_k)|\right)+C\frac1{\zeta\kappa}\right)+\kappa^{1/16}\kappa^2H^2\zeta^6\,.
\end{align*}
We sum over $k$ and select the points $x_k$ as follows
$$|B_0(x_k)|=\min\{|B_0(x)|~:~x\in\overline{\mathcal
Q_{k,\zeta}}\,\}\,.$$ In that way, we obtain,
\begin{align*}
\mathcal E_0(v,\Fb)&=\sum_{k\in\mathcal I}\mathcal E_0(v,\Fb;\mathcal Q_{k,\zeta})\\
&\leq (1+\kappa^{-1/16})\kappa^2
\int_{\bigcup_{k\in\mathcal I}\overline{ Q_{k,\zeta}}}
\left(g\left(\frac H\kappa |B_0(x)|\right)+C\frac1{\zeta\kappa}+\kappa^{1/16}H^2\zeta^4\right)\,dx\\
&\leq (1+\kappa^{-1/16})\kappa^2\int_{\bigcup_{k\in\mathcal I}\overline{ Q_{k,\zeta}}} g\left(\frac H\kappa
|B_0(x)|\right)\,dx+C\kappa^{1/16}\frac{\kappa^2}{H\zeta}+M\kappa^{-1/16}\frac{\kappa^3}H\,.
\end{align*}
Notice that, since $g(b)=0$ for $b\geq 1$ and $B_0$ vanishes
non-degenerately on $\Gamma$, then
$$
\int_{\bigcup_{k\in\mathcal I}\overline{ Q_{k,\zeta}}}
g\left(\frac H\kappa |B_0(x)|\right)\,dx\leq \int_{\Omega}
g\left(\frac H\kappa |B_0(x)|\right)\,dx +\frac{\kappa}{H}\,o(1)\,.
$$
Thus,
$$\mathcal E_0(v,\Fb)\leq (1+\kappa^{-1/16})\kappa^2\int_{\Omega} g\left(\frac H\kappa
|B_0(x)|\right)\,dx+C\kappa^{1/16}\frac{\kappa^2}{H\zeta}+M\kappa^{-1/16}\frac{\kappa^3}H+\frac{\kappa^3}{H}\,o(1)\,.$$
Since $H\lesssim \kappa^{3/2}$, then
$$\kappa^{1/16}\frac{\kappa^2}{H\zeta}=\kappa^{1/8}\frac{H^{1/2}}{\kappa}\frac{\kappa^3}{H}\ll1$$
and
$$\mathcal E_0(v,\Fb)\leq (1+\kappa^{-1/16})\kappa^2\int_{\Omega} g\left(\frac H\kappa
|B_0(x)|\right)\,dx+\frac{\kappa^3}{H}\,o(1)\,.$$
Since $\E\leq \mathcal E(v,\Fb)=\mathcal E_0(v,\Fb)$, then we get
the upper bound in \eqref{eq:ub2}.
\end{proof}
\section{Exponential decay of the order parameter}\label{s6}

The aim of this section is to prove that the order parameter $\psi$  is
exponentially small (in the $L^2$-norm) away from the points where
the magnetic field vanishes. This bound is needed in
Section \ref{s7} to obtain a lower bound of the ground state energy in
\eqref{eq-gse}.

\subsection{A rough bound}
In this subsection we give a rough bound valid for any order parameter $\psi$.
\begin{thm}\label{thm:rb}
Let $\Lambda>0$ and $\epsilon:\R\to\R_+$ such that
$\displaystyle\lim_{\kappa\to
\infty}\kappa\,\epsilon(\kappa)=\infty$
 and $\displaystyle\lim_{\kappa\to \infty}\epsilon(\kappa)=0~$. There exist constants $C$ and
$\kappa_0$ such that, if
$(\psi,\Ab)$ is a critical point of the functional in
\eqref{eq-3D-GLf}, $\kappa\geq\kappa_0$ and
\begin{equation}\label{hypH}
\epsilon(\kappa)\kappa^2 \leq H\leq \Lambda\kappa^2\,,
\end{equation}
 then
\begin{equation}\label{l2est}
\|\psi\|_2\leq C\left(\frac{\kappa}{H}\right)^{1/6}\,,
\end{equation}
\begin{equation}\label{l2curl} \|\curl (\Ab-\Fb)\|_2\leq \frac C H
\left(\frac{\kappa}{H}\right)^{1/6}\,,
\end{equation}
and
\begin{equation}\label{l2est:grad}
\|(\nabla-i\kappa H\Ab)\psi\|_2\leq  C \kappa  \left(\frac{\kappa}{H}\right)^{1/6}\,.
\end{equation}
\end{thm}

An important ingredient in the proof of this theorem is:
\begin{prop}\label{prop:lb-blk*}
Let $\Lambda>0$ and $\epsilon:\R\to\R_+$ such that
$\displaystyle\lim_{\kappa\to \infty}\kappa\,\epsilon(\kappa)=
\infty$
 and $\displaystyle\lim_{\kappa\to \infty}\epsilon(\kappa)=0~$. There exist positive constants $C$, $\ell_0$ and
$\kappa_0$ such that the following is true:\\
For  $\ell\in\,(0,\ell_0)\,$, $a\in\,(0,1]$ and $h\in C_c^\infty(\Omega)$
such that
$$
{\rm supp}\,h\subset \{x\in{\Omega}~:~{\rm
dist}(x,\partial\Omega)> \ell~\&~{\rm dist}(x,\Gamma)>
\sqrt{a}\,\ell\}{\rm ~and~}
\|h\|_\infty\leq 1\,,$$ if
$(\psi,\Ab)$ is a critical point of the functional in
\eqref{eq-3D-GLf}, $\kappa\geq\kappa_0$ and
$\epsilon(\kappa)\kappa^2 \leq H\leq \Lambda\kappa^2$,  then
\begin{equation}\label{blk1}
\int_\Omega|(\nabla-i\kappa H\Ab)\psi|^2\,dx\geq {\frac 1C}
\kappa\big(H\sqrt{a}\,\ell-{C^2}\big)\int_\Omega|h\psi|^2\,dx - C\kappa\int_\Omega(1-h^2)|\psi|^2\,dx\,.
\end{equation}
\end{prop}
\begin{proof}~\\
The support of the function $h\psi$ does not meet the boundary of
$\Omega$  and $\Gamma$. We can use the celebrated inequality
$$\int_{\Omega}|(\nabla-i\kappa H\Ab)h\psi|^2\,dx\geq \kappa
H\int_\Omega|\curl\Ab|\,|h\,\psi|^2\,dx\,.$$
The simple decomposition $\curl\Ab=\curl\Fb+\curl(\Ab-\Fb)$ and the triangle inequality yield,
\begin{equation}\label{eq:lb-t}
\int_{\Omega}|(\nabla-i\kappa H\Ab)h\psi|^2\,dx\geq \kappa
H\int_\Omega|\curl\Fb|\,|h\psi|^2\,dx-\kappa H\int_\Omega|\curl\Ab-\curl\Fb|\,|h\psi|^2\,dx\,.
\end{equation}
By assumption  $\nabla B_0$ does not vanish  on $\Gamma$, hence
\begin{equation}\label{eq:cond-B0}
|\curl\Fb|=|B_0(x)|\geq \frac 1 M \sqrt{a}\,\ell\quad{\rm in~}\{{\rm dist}(x,\Gamma)\geq\sqrt{a}\,\ell\}\,,
\end{equation}
for some constant $M >0\,$. \\
Thus,
\begin{equation}\label{eq:curlF}
\int_\Omega|\curl\Fb|\,|h\psi|^2\,dx\geq \frac{1}{M}\, \sqrt{a}\,\ell\int_\Omega|h\psi|^2\,dx\,.
\end{equation}
Next we use the Cauchy-Schwarz inequality and the inequality in \eqref{eq:curl} as
follows
\begin{align*}
\kappa H\, \int_\Omega|\curl\Ab-\curl\Fb|\,|h\psi|^2\,dx&\leq
 \kappa H\|\curl\Ab-\curl\Fb\|_{2}\left(\int_\Omega|h\psi|^4\,dx\right)^{1/2}\\
&\leq C\kappa\|\psi\|_2\left(\int_\Omega|h\psi|^4\,dx\right)^{1/2}\,.
\end{align*}
Since $\|\psi\|_\infty\leq1$ and $\|h\|_\infty\leq 1$, we get further,
$$\|\psi\|_2\left(\int_\Omega|h\psi|^4\,dx\right)^{1/2}\leq \int_\Omega|\psi|^2\,dx=\int_\Omega|h\psi|^2\,dx+\int_\Omega(1-h^2)|\psi|^2\,dx\,.$$
Therefore, we have,
\begin{equation}\label{eq:h=1}
\kappa H\int_\Omega|\curl\Ab-\curl\Fb|\,|h\psi|^2\,dx\leq C\kappa \int_\Omega|\psi|^2\,dx+ C \kappa\int_\Omega(1-h^2)|\psi|^2\,dx\,.\end{equation}
Inserting \eqref{eq:h=1}  and \eqref{eq:curlF} into \eqref{eq:lb-t} finishes the proof of the proposition.
\end{proof}

\begin{proof}[Proof of Theorem~\ref{thm:rb}]

Let $\ell>0$ and $\Omega_\ell=\{x\in\Omega~:~{\rm dist}(x,\partial\Omega)>\ell~\&~{\rm dist}(x,\Gamma)>\ell\}$.
Select a function $h\in C_c^\infty(\Omega)$ satisfying
$$0\leq h\leq 1{\rm ~in~}\Omega\,,\quad h=1{\rm ~in~}\Omega_{2\ell}\,,\quad h=0{~\rm in~}\Omega_\ell\,,$$
and
$$|\nabla h|\leq \frac{C}{\ell}\quad{\rm in~}\Omega\,,$$
where $C$ is a constant.

Thanks to the bound $\|\psi\|_\infty\leq 1$ and the assumptions on $h$, we have,
\begin{align}
&\int_\Omega|\psi|^2\leq \int|h\psi|^2+C\ell\,,\label{eq:1}\\
&\int_\Omega|(\nabla-i\kappa H\Ab)\psi|^2\,dx\geq \int_\Omega|h(\nabla-i\kappa H\Ab)\psi|^2\,dx\\
&\hskip1cm\geq\frac12\int_\Omega|(\nabla-i\kappa H\Ab)h\psi|^2\,dx-C\int_\Omega|\nabla h|^2\,|\psi|^2\,dx\,.
\end{align}
%
Thanks to the estimate on the gradient of $h$, we may write,
$$\frac12\int_\Omega|(\nabla-i\kappa H\Ab)h\psi|^2\,dx-
\kappa^2\int_\Omega|h\psi|^2\,dx-C\big(\ell+\ell^{-1}\big)\kappa^2 \leq \mathcal E_0(\psi,\Ab;\Omega)\leq 0\,,$$
where $\mathcal E_0(\psi,\Ab;\Omega)$ is introduced in \eqref{eq:GLen}.

Now, we use Proposition~\ref{prop:lb-blk*} with $a=1$ and get,
$$\left(\frac{\kappa}{2C}(H\ell-C^2)
-\kappa^2\right)\int_\Omega|h\psi|^2\,dx\leq C\big(\ell+\ell^{-1}\big)\kappa^2\,.$$
Selecting $\ell=(\kappa/H)^{1/3}$, we get for $\kappa$ large and $H$ satisfying \eqref{hypH}
$$\int_\Omega|h\psi|^2\,dx\leq C\left(\frac{\kappa}{H}\right)^{1/3}\,.$$
Now, thanks to \eqref{eq:1}, the first inequality \eqref{l2est} in
Theorem~\ref{thm:rb} is proved. Now, the  inequality \eqref{l2curl}
(resp.  \eqref{l2est:grad}) is simply a consequence of
\eqref{eq:curl}) (resp.  \eqref{eq:grad<kappa}).
\end{proof}
\subsection{Exponential bound}

In the next theorem, we establish that any order parameter decays
exponentially fast away from the set $\Gamma$  where the magnetic
field vanishes.

\begin{thm}\label{thm:exdec}
Let $\Lambda>0$ and $\epsilon:\R\to\R_+$ such that
$\displaystyle\lim_{\kappa\to \infty}\kappa\,\epsilon(\kappa)=
\infty$ and $ \displaystyle\lim_{\kappa\to
\infty}\epsilon(\kappa)=0\,$. There exist positive constants $C$,
$m_0$
 and $\kappa_0$ such that, if $(\psi,\Ab)$ is a
critical point of the functional in \eqref{eq-3D-GLf},
$\kappa\geq\kappa_0\,$, $\epsilon(\kappa)\kappa^2 \leq H\leq
\Lambda\kappa^2$, then
$$
\int_{\Omega} \exp\left(2m_0\frac H\kappa\,
t(x)\right)\left(\frac1{\kappa^2}|(\nabla-i\kappa
H\Ab)\psi|^2+|\psi(x)|^2\right)\,dx\leq C\, \int_{\{t(x)\leq
C\frac \kappa H\}}|\psi(x)|^2\,dx\,,
$$
where $t(x)={\rm dist}(x,\Gamma)$.
\end{thm}
\begin{proof}
Let
\begin{equation}\label{zeta}
\zeta=(\kappa H)^{-1/3}\,.
\end{equation}
The assumption on $\kappa$ and $H$
ensures that
\begin{equation}\label{estzeta}
\kappa^{-1} \lesssim \zeta \ll 1\,.
\end{equation}
We will prove Theorem~\ref{thm:exdec} by establishing the following
two estimates {(away from the boundary or in a neighborhood of the
boundary)},
\begin{equation}
\int_{\{{\rm dist}(x,\partial\Omega)\geq\zeta\}}
e^{2m_0\frac H \kappa\, t(x)}\left(\frac1{\kappa^2}|(\nabla-i\kappa H\Ab)\psi|^2+|\psi(x)|^2\right)\,dx\leq
C_1\int_{\{t(x)\leq C\frac{\kappa}{H}\}}|\psi(x)|^2\,dx\,,\label{eq:ed-int}
\end{equation}
and
\begin{equation}
\int_{\{{\rm dist}(x,\partial\Omega)\leq\zeta\}}
e^{2 m_0 \frac H\kappa\, t(x)}\left(\frac1{\kappa^2}|(\nabla-i\kappa H\Ab)\psi|^2+|\psi(x)|^2\right)\,dx\leq
C_2\int_{\{t(x)\leq
C\frac{\kappa}{H}\}}|\psi(x)|^2\,dx\,,\label{eq:ed-bd}
\end{equation}
expressing the localization of the energy of $\psi$ near $\Gamma$.\\

The proof of \eqref{eq:ed-int} and \eqref{eq:ed-bd} is divided into several steps.\\
{\bf Step~1.}\\
Consider the parameters
\begin{equation}\label{eq:exp}
\xi \in (1,\infty)\,,\quad \sigma=\frac{H}{\kappa^2}\,,\quad
\ell=\frac{\xi}{\sigma\kappa}\,.
\end{equation}
Let
$$f(x)=\chi(x)\,\exp(\ell^{-1}\,t(x))\,,$$
and
$$g(x)=\eta(x)\,\exp(\ell^{-1}t(x))\,.$$
The functions $\chi\in C_c^\infty(\Omega)$ and $\eta\in
C^\infty(\overline\Omega)$ satisfy,
\begin{equation}
\left\{\begin{array}{ll}
&0\leq \chi\leq 1\quad{\rm in~}\Omega\,,\\
&\chi=1\quad{\rm in~}\{{\rm dist}(x,\partial\Omega)\geq \zeta\}\bigcup\{t(x)\geq \ell\}\,,\\
&\chi=0\quad{\rm in~}\{{\rm dist}(x,\partial\Omega)\leq \frac12\zeta\}\bigcup\{t(x)\leq \frac12\ell\}\,,\\
&|\nabla\chi|\leq C {\kappa} \quad{\rm in~}\Omega\,.
\end{array}
\right.
\end{equation}
and
\begin{equation}
\left\{
\begin{array}{ll}
&0\leq \eta\leq 1\quad{\rm in~}\Omega\,,\\
&\eta=1\quad{\rm in~}\{{\rm dist}(x,\partial\Omega)\leq \zeta\}\bigcup\{t(x)\geq \ell\}\,,\\
&\eta=0\quad{\rm in~}\{{\rm dist}(x,\partial\Omega)\geq 2\zeta\}\bigcup\{t(x)\leq \frac12\ell\}\,,\\
&{|\nabla\eta|\leq C\kappa\quad{\rm in~}\Omega\,.}
\end{array}
\right.
\end{equation}
Here we have used for the control of the  gradient \eqref{estzeta} and that
$$\kappa^{-1}\ll\ell\lesssim1\,.$$
%
%
%
Using the Ginzburg-Landau equation in \eqref{eq:GL}, we write,
\begin{multline}\label{eq:jdeFK}
\int_\Omega\Big(|(\nabla-i\kappa H\Ab)f\psi|^2-|\nabla
f|^2|\psi|^2\Big)\,dx=\kappa^2\int_\Omega\big(|\psi|^2-|\psi|^4\big)f^2\,dx
\leq\kappa^2\int_\Omega|f\psi|^2\,dx\,,
\end{multline}
and
\begin{multline}\label{eq:jdeFK*}
\int_\Omega\Big(|(\nabla-i\kappa H\Ab)g\psi|^2-|\nabla
g|^2|\psi|^2\Big)\,dx=\kappa^2\int_\Omega\big(|\psi|^2-|\psi|^4\big)g^2\,dx
\leq\kappa^2\int_\Omega|g\psi|^2\,dx\,.
\end{multline}
%
%
%
%
%
%
%
{\bf Step~2.}\\
In this step, we determine a lower bound of
$\displaystyle\int_\Omega|(\nabla-i\kappa H\Ab)f\psi|^2\,dx$. Notice
that $f\psi\in C_c^\infty(\Omega)$. Consequently, we may write (see
\eqref{eq:lb-t}),
$$
\int_\Omega|(\nabla-i\kappa H\Ab)f\psi|^2\,dx
\geq \kappa H\int_\Omega|\curl\Fb|\,|f\psi|^2\,dx-\kappa H\int_\Omega|\curl\Ab-\curl\Fb|\,|f\psi|^2\,dx\,.
$$
We use the following estimates,
\begin{align*}
&\int_\Omega|\curl\Fb|\,|f\psi|^2\,dx\geq \frac 1 M \,\ell\int_\Omega|f\psi|^2\,dx&[{\rm by~\eqref{eq:cond-B0}}]&\\
&  \int_\Omega|\curl\Ab-\curl\Fb|\,|f\psi|^2\,dx\leq \frac{C}{H}\left(\frac{\kappa}{H}\right)^{1/6}\|f\psi\|_4^2&[{\rm by}~\eqref{l2curl}]&\,,
\end{align*}
and obtain
$$\int_\Omega|(\nabla-i\kappa H\Ab)f\psi|^2\,dx\geq \frac 1M\, \kappa
H\ell\int_\Omega|f\psi|^2\,dx-C\kappa\left(\frac{\kappa}{H}\right)^{1/6}\|f\psi\|_4^2\,.$$
Notice that $f\psi\in C_c^\infty(\Omega)\subset H^1(\R^2)\,$. By the continuous
Sobolev embedding of $H^1(\R^2)$ in $L^4(\R^2)$ and a scaling, we
get for all $\eta\in(0,1)$,
\begin{align*}
\|f\psi\|_4^2&=\big\|\,|f\psi|\,\big\|_4^2&&\\
&\leq C_{\rm Sob}\Big(\eta\|\nabla |f\psi|\|_2^2+\eta^{-1}\|f\psi\|_2^2\Big)&&\\
&\leq C_{\rm Sob}\Big(\eta\|(\nabla-i\kappa H\Ab)f\psi\|_2^2+\eta^{-1}\|f\psi\|_2^2\Big)&[\text{\rm By~the~diamagnetic~inequality}]&\,.
\end{align*}
We select $\eta=\frac{ 1}{CC_{\rm Sob}}  \kappa^{-1} \left(\frac
\kappa H\right)^{-\frac 16}$ and obtain,
\begin{equation}\label{eq:g1*}
\int_\Omega|(\nabla-i\kappa H\Ab)f\psi|^2\,dx\geq \Big(\frac{\kappa
H\ell}{2M}-\widehat C\kappa^2\left(\frac{\kappa}{H}\right)^{1/3}\Big)\int_\Omega|f\psi|^2\,dx\,.
\end{equation}
Thanks to the choice of the parameters in \eqref{eq:exp}, the lower
bound in \eqref{eq:g1*} becomes,
\begin{equation}\label{eq:g1}
\int_\Omega|(\nabla-i\kappa H\Ab)f\psi|^2\,dx\geq \Big(\frac{\xi\kappa^2}{2M}
- \widehat C\kappa^2 \left(\frac{\kappa}{H}\right)^{1/3}\Big)\int_\Omega|f\psi|^2\,dx\,.
\end{equation}
{\bf Step~3.}\\
We insert \eqref{eq:g1} into \eqref{eq:jdeFK} and use that
\begin{align}\label{eq:nablaf} \int_\Omega|\nabla
f|^2|\psi|^2\,dx&\leq
2\ell^{-2}\int_\Omega|f\psi|^2\,dx+2 \int_\Omega|\nabla\chi|^2\exp(2\ell^{-1}t(x))|\psi|^2\,dx\nonumber\\
&\leq 2\ell^{-2}\int_\Omega|f\psi|^2\,dx+C\kappa^2\int_\Omega|g\psi|^2\,dx+C\kappa^2\int_{\{\ell^{-1}t(x)\leq1\}}|\psi|^2\,dx\,,
\end{align}
to obtain,
\begin{multline}\label{eq:expdec1}
\int_\Omega\left(\frac12|(\nabla-i\kappa
H\Ab)f\psi|^2+\frac12\Big(\frac{\xi\kappa^2}{2M}
- 2\frac{\sigma^2}{\xi^2}\kappa^2-\widehat C\kappa^2\left(\frac{\kappa}{H}\right)^{1/3}\Big)|f\psi|^2\,dx\right)\\
\leq \widehat C  \kappa^2\int_\Omega|g\psi|^2\,dx
+ \widehat C \kappa^2\int_{\{\ell^{-1}t(x)\leq1\}}|\psi|^2\,dx\,.
\end{multline}
%
%
%
{\bf Step~4.}\\
%
%
We will  determine a lower bound of
$\displaystyle\int_\Omega|(\nabla-i\kappa H\Ab)g\psi|^2\,dx$. We
cover the set
$$\Omega_{\zeta,\ell}=\{x\in\overline{\Omega}~:~{\rm dist}(x,\partial\Omega)\leq
2\zeta\,,~{\rm dist}(x,\Gamma)\geq \ell\}$$ by a family of squares (in tubular coordinates),
$$\mathcal K(a_j,\zeta)=\{x~\in\overline{\Omega}~:~{\rm
dist}(x,\partial\Omega)\leq 2\zeta,~{\rm
dist}_{\partial\Omega}(p(x),a_j)\leq 2\zeta\}\,,$$ where:
\begin{itemize}
\item ${\rm dist}_{\partial\Omega}$ is the arc-length distance
along $\partial\Omega$\,.
\item If $x\in\Omega_{\zeta,\ell}$ and $\zeta$ is sufficiently small, $p(x)$ is the unique point on $\partial\Omega$ satisfying
${\rm dist}(x,p(x))={\rm dist}(x,\partial\Omega)$\,.
\item For all $j$,  $a_j\in\partial\Omega\cap\overline{\Omega_{\zeta,\ell}}$.
\end{itemize}
Let $(\chi_j)$ be a partition of unity such that
$$\sum_{j}\chi_j^2=1\,,\quad \sum_j|\nabla\chi_j|^2 \leq
C\zeta^{-2}\,,\quad {\rm supp}\,\chi_j\subset\mathcal K(a_j,2\zeta)\,.
$$
There holds the decomposition formula
\begin{align}
\int_{\Omega}|(\nabla-i\kappa
H\Ab)g\psi|^2\,dx&=\sum_j\int_\Omega|(\nabla-i\kappa
H\Ab)\chi_jg\psi|^2\,dx-\sum_j\int_\Omega|\nabla\chi_j|^2\,|g\psi|^2\,dx\nonumber\\
&\geq\sum_j\int_\Omega|(\nabla-i\kappa
H\Ab)\chi_jg\psi|^2\,dx-C\zeta^{-2}\int_\Omega|g\psi|^2\,dx\,.\label{eq:half-plane*}
\end{align}
Next we define the gauge function
$$\alpha_j=\big(\Ab(a_j)-\Fb(a_j)\big)\cdot(x-a_j)\,.$$
Using the Cauchy-Schwarz inequality, Proposition \ref{prop:A-F} and
Theorem \ref{thm:rb},   we may write,
\begin{multline}\label{eq:half-plane1*}
\int_\Omega|(\nabla-i\kappa
H\Ab)\chi_jg\psi|^2\,dx=\int_\Omega|(\nabla-i\kappa
H(\Ab-\nabla\alpha_j))e^{-i\kappa
H\alpha_j}\chi_jg\psi|^2\,dx\\
\geq\frac12\int_\Omega|(\nabla-i\kappa H\Fb)e^{-i\kappa
H\alpha_j}\chi_jg\psi|^2\,dx-C\kappa^2\left(\frac{\kappa}{H}\right)^{1/3}\zeta^{2\alpha}\int_\Omega|\chi_jg\psi|^2\,dx\,.
\end{multline}
Next, we observe that there exists a gauge function $\varphi_j$
satisfying
$$\big|\Fb(x)-\big(B_0(x_j)\Ab_0(x-a_j)+\nabla\varphi_j\big)\big|\leq
C\zeta^2\quad{\rm in ~}\mathcal K_j(a_j,\zeta)\,.$$ Again, using
the Cauchy-Schwarz inequality, we may write,
\begin{multline}\label{eq:half-plane2*}
\int_\Omega|(\nabla-i\kappa H\Fb)e^{-i\kappa
H\alpha_j}\chi_jg\psi|^2\,dx\\
\geq \frac12\int_\Omega|(\nabla-i\kappa
HB_0(a_j)\Ab_0(x-a_j))e^{-i\kappa H\varphi_j}e^{-i\kappa
H\alpha_j}\chi_jg\psi|^2\,dx-\kappa^2H^2\zeta^4\int_\Omega|\chi_jg\psi|^2\,dx\,.
\end{multline}
Now, we are reduced to the analysis of the Neumann realization of
the  Schr\"odinger operator  with a constant magnetic field equal to
$\kappa HB_0(a_j)$ in our case. In the half-plane case, the ground
state energy of this operator is $\Theta_0\kappa H|B_0(a_j)|$, where
the constant $\Theta_0$ is universal and satisfies
$\Theta_0\in\,(\frac12,1)$\,. The result remains asymptotically true
in general domains with smooth and compact boundary \cite{HelMo1}.
More precisely, there exists a function
$${\rm err}:\R_+\to\R_+\,,$$
such that $\lim_{|b|\to\infty}{\rm err}(b)=0$ and
$$\forall~b\,,\quad \lambda^N(b)\geq \Theta_0 |b|-|b|\,{\rm
err}(b)\,,$$ where $\lambda^N(b)$ is the lowest eigenvalue of the
operator $-(\nabla-ib\Ab_0)^2$ in $L^2(\Omega)$ with Neumann
boundary condition.

Notice that by the assumptions on $\ell$ and the points $(a_j)$, we
may use \eqref{eq:cond-B0} with $x=a_j$ and get,
$$\forall~j\,,\quad \kappa H|B_0(a_j)|\geq  \frac 1M \ell\,\kappa H\gg
1\,.$$
Moreover, the magnetic potentials $\Ab_0(x)$ and $\Ab_0(x-a_j)$ are
gauge equivalent since
$$\Ab_0(x-a_j)=\Ab_0(x)-\Ab_0(a_j)=\Ab_0(x)-\nabla u_j(x)\,,$$
 with $u_j(x) = A_0(a_j)\cdot x\,$.\\
In that way,  when $\kappa$ is sufficiently large, we may write,
\begin{multline}\label{eq:half-plane3*}
\int_\Omega|(\nabla-i\kappa HB_0(a_j)\Ab_0(x-a_j))e^{-i\kappa
H\varphi_j}e^{-i\kappa H\alpha_j}\chi_jg\psi|^2\,dx
\\
\geq \frac{\Theta_0}{2} \kappa H
|B_0(a_j)|\,\int_\Omega|\chi_jg\psi|^2\,dx\geq  \frac{1}{4M}\,
\ell\kappa H\, \int_\Omega|\chi_jg\psi|^2\,dx\,.
\end{multline}
%
%
%
Collecting the estimates in \eqref{eq:half-plane*},
\eqref{eq:half-plane1*}, \eqref{eq:half-plane2*} and
\eqref{eq:half-plane3*}, we get,
\begin{equation}\label{eq:g2}
\int_\Omega|(\nabla-i\kappa H\Ab)g\psi|^2\,dx\geq \kappa\left({\frac{H\ell}{4M}} -C\kappa H^2\zeta^4-\frac{C}{\kappa\zeta^2}-C\kappa\left(\frac\kappa{H}\right)^{1/3}\zeta^{2\alpha}\right)
\int_\Omega|g\psi|^2\,dx\,.
\end{equation}
Recall the definition of the  parameters in \eqref{eq:exp} and
\eqref{zeta}:
$$ \zeta=(H\kappa)^{-1/3}=\sigma^{-1/3}\kappa^{-1}\,.
$$ We insert
\eqref{eq:g2} into \eqref{eq:jdeFK*} and use that
\begin{align*}
\int_\Omega|\nabla g|^2|\psi|^2\,dx&\leq
2 \ell^{-2}\int_\Omega|g\psi|^2\,dx+ 2 \int_\Omega|\nabla\eta|^2\exp(2\ell^{-1}t(x))|\psi|^2\,dx\\
&\leq
2 \ell^{-2}\int_\Omega|g\psi|^2\,dx+C\kappa^2\int_\Omega|f\psi|^2\,dx+
C\kappa^2\int_{\{\ell^{-1}t(x)\leq1\}}|\psi|^2\,dx\,,
\end{align*}
to obtain,
\begin{multline}\label{eq:expdec2}
\int_\Omega\left(\frac12|(\nabla-i\kappa
H\Ab)g\psi|^2+\frac12\Big(\frac{\xi\kappa^2}{4M}
- 2 \ell^{-2}-C\sigma^{2/3}\kappa-C\kappa\left(\frac{\kappa}{H}\right)^{1/3}\zeta^{2\alpha}\Big)|g\psi|^2\,dx\right)\\
\leq \kappa^2\int_\Omega|f\psi|^2\,dx
+C\kappa^2\int_{\{\kappa\,t(x)\leq1\}}|\psi|^2\,dx\,.
\end{multline}
{\bf Step~5.}\\
Adding the two inequalities in \eqref{eq:expdec1} and
\eqref{eq:expdec2}, we get,
\begin{align}
&\frac12\int_\Omega\left(|(\nabla-i\kappa
H\Ab)g\psi|^2+|(\nabla-i\kappa H\Ab)f\psi|^2\right)\,dx\nonumber\\
&\hskip1cm+\frac12\int_\Omega\Big(\frac{\xi\kappa^2}{4M}
-C\frac{\sigma^2}{\xi^2}\kappa^2-C\kappa^2-C\sigma^{2/3}\kappa-C\kappa\left(\frac{\kappa}{H}\right)^{1/3}\zeta^{2\alpha}\Big)\Big(|f\psi|^2+|g\psi|^2\Big)
\nonumber\\
&\hskip2cm\leq C\kappa^2\int_{\{\kappa\,t(x)\leq1\}}|\psi|^2\,dx\,.\label{eq:expdec*}
\end{align}
Recall that $\sigma$ satisfies $\kappa^{-1}\ll\sigma\leq\Lambda$. We
select $\xi$ sufficiently large
 such that
$$\frac{\xi}{4M} -C\frac{\Lambda^2}{\xi^2}-C>2\,.$$
Since $\left(\frac{\kappa}{H}\right)^{1/3}\ll 1$ and $ \zeta\ll1$,
we get,
$$\int_\Omega\left(\frac12|(\nabla-i\kappa H\Ab)f\psi|^2+\frac{\kappa^2}2|f\psi|^2\,dx\right)\leq
C\kappa^2\int_{\{\ell^{-1}\,t(x)\leq1\}}|\psi|^2\,dx\,,
$$
and
$$\int_\Omega\left(\frac12|(\nabla-i\kappa H\Ab)g\psi|^2+\frac{\kappa^2}2|g\psi|^2\,dx\right)\leq C\kappa^2\int_{\{\ell^{-1}\,t(x)\leq1\}}|\psi|^2\,dx\,.
$$ Thanks to the definitions of $f$ and $g$, the two aforementioned inequalities yield the inequalities in
\eqref{eq:ed-int} and \eqref{eq:ed-bd} with $m_0= 1 /\xi$.
\end{proof}

As a consequence of Theorem~\ref{thm:exdec}, we get an improvement
of the bound given in Theorem~\ref{thm:rb}.
\begin{prop}\label{thm:ob}
Under the assumptions of Theorem~\ref{thm:exdec}, there holds,
$$
\|\psi\|_2\leq C\sqrt{\frac \kappa H}\,.$$
\end{prop}

Combining the results in Propositions~\ref{prop:FH-b},
 \ref{prop:A-F} and \ref{thm:ob}, we obtain the improved estimates:
\begin{prop}\label{thm:oe}
Under the assumptions of Theorem~\ref{thm:exdec} and
Proposition~\ref{prop:A-F}, there holds,
\begin{align*}
& \|\curl\Ab-\curl\Fb\|_2\leq  \frac{C}{H}\sqrt{\frac\kappa H} \,,\\
&\|\Ab-\Fb\|_{C^{1,\alpha}(\overline\Omega)}\leq C_\alpha \sqrt{\frac \kappa H}\,,\\
&\|\Ab-\Fb\|_{C^{0,\alpha}(\overline\Omega)}\leq \frac{\widehat C_\alpha}{H}\sqrt{\frac\kappa H}\,.
\end{align*}
\end{prop}

\section{Energy lower bound}\label{s7}
In this section, we will derive lower bounds of the following
energy,
\begin{equation}\label{eq:GLf-U}
\mathcal E_0(\psi,\Ab;U)=\int_U\left(|(\nabla-i\kappa H\Ab)\psi|^2-\kappa^2|\psi|^2+\frac{\kappa^2}2|\psi|^4\right)\,dx\,,
\end{equation}
where $U\subset \R^2$ is an open set such that $\overline{U} \subset
\Omega$.

\begin{prop}\label{prop:lb}
Let $\Lambda>0$ and $\epsilon:\R\to\R_+$ such that
$\displaystyle\lim_{\kappa\to \infty}\kappa\,\epsilon(\kappa)=
\infty$
 and $\displaystyle\lim_{\kappa\to \infty}\epsilon(\kappa)=0~.$ For $\alpha \in (0,1)$, there exist positive constants $C$ and
$\kappa_0$ such that,
for $\ell\in\,(0,1)\,$,  $\delta\in\,(0,1)\,$,
$a_j\in\Gamma$, $D(a_j,\ell)\subset\Omega$,  $x_j\in
\overline{D(a_j,\ell)}\cap \Gamma$,   $h\in
C_c^\infty\big(D(a_j,\ell)\big)$  a function satisfying
$\|h\|_\infty\leq1$,
  $(\psi,\Ab)$  a critical point of the
functional in \eqref{eq-3D-GLf}, $\kappa\geq\kappa_0\,,$ and
$$ \epsilon(\kappa)\kappa^2 \leq H\leq \Lambda\kappa^2\,,$$
the following holds:
$$\mathcal E_0\big(h\,\psi,\Ab; D(a_j,\ell)\big)\geq
(1-\delta)\,2\ell\,\kappa\,\left(|\nabla
B_0(x_j)|\frac{H}{\kappa^2}\right)^{1/3}\,E\left(|\nabla
B_0(x_j)|\frac{H}{\kappa^2}\right)- r  \,,$$ where
$$r=C\Big(\delta\kappa^2+\delta^{-1}\left(\frac{\kappa^3}{H}\ell^{2\alpha}+\kappa^2H^2\ell^6\right)\Big)\int_{D(a_j,\ell)}|h\psi|^2\,dx\,.$$
\end{prop}
\begin{proof}
Let $\alpha_j=\big(\Ab(a_j)-\Fb(a_j)\big)\cdot(x-a_j)$. Thanks to
Proposition~\ref{thm:oe}, we have
\begin{equation}\label{eq:lb-gauge}
\left|\Ab-(\Fb+\nabla\alpha_j)\right|\leq
C\, \|\Ab-\Fb\|_{C^{0,\alpha}(\overline\Omega)}{ |x-a_j|^\alpha} \leq
\frac{C}{H}\sqrt{\frac \kappa H}\,\ell^\alpha \quad{\rm in}\quad D(a_j,\ell)\,.
\end{equation}
Notice that,
\begin{align}
\mathcal E_0\big(h\,\psi,\Ab; D(a_j,\ell)\big)&=\mathcal E_0\big(h\,\psi\,e^{-i\kappa H\alpha_j},\Ab-\nabla\alpha_j; D(a_j,\ell)\big)\nonumber\\
&\geq (1-\delta)\mathcal E_0\big(h\,\psi\,e^{-i\kappa H\alpha_j},\Fb; D(a_j,\ell)\big)\nonumber\\ & \quad
-C\left(\delta\kappa^2\int_{D(a_j,\ell)}|h\psi|^2\,dx+\delta^{-1}\kappa^2H^2\, \int_{D(a_j,\ell)} |\Ab-(\Fb+\nabla\alpha_j)|^2 |h\psi|^2\,dx\right)\nonumber\\
\end{align}
Using \eqref{eq:lb-gauge}, we get,
\begin{equation}\label{eq:lb-A-F}
\mathcal E_0\big(h\,\psi,\Ab; D(a_j,\ell)\big)\geq (1-\delta)\mathcal E_0\big(h\,\psi\,e^{-i\kappa H\alpha_j},\Fb; D(a_j,\ell)\big)-C
\left(\delta\kappa^2+\delta^{-1}\frac{\kappa^3}{H}\ell^{2\alpha}\right)\int_{D(a_j,\ell)}|h\psi|^2\,dx\,.
\end{equation}
 Let  $$f_j=h\,\psi\,e^{-i\kappa
H\alpha_j}e^{i\kappa H\phi_j}\,,
$$
where  $\phi_j$ is defined in
Proposition~\ref{prop:guage}. \\
 Notice that $f_j\in
H^1_0(D(a_j,\ell))$, $\|f_j\|_\infty\leq 1$ and, using \eqref{eq24.3(2)},
\begin{align}
&\mathcal E_0\big(h\,\psi\,e^{-i\kappa H\alpha_j},\Fb; D(a_j,\ell)\big)\nonumber\\
&\quad=\mathcal E_0\big(f_j,\Fb-\nabla\phi_j; D(a_j,\ell)\big)\nonumber\\
&\quad\geq
(1-\delta)\mathcal E_0\big(f_j,|\nabla B_0(x_j)|\Ab_{{\rm app},\nu_j}(x-a_j); D(a_j,\ell)\big)
-C\left(\delta\kappa^2+\delta^{-1}\kappa^2H^2\ell^6\right)\int_{D(a_j,\ell)}|f_j|^2\,dx\,.\label{eq:lb-F-A0}
\end{align}
We will use Theorem~\ref{thm-tdl-e} to
get a lower bound of the energy
$$\mathcal E_0\big(f_j,|\nabla
B_0(x_j)|\Ab_{{\rm app},\nu_j}(x-a_j); D(a_j,\ell)\big)\,.$$
Define
\begin{equation}\label{eq:Lj}
L=L_j=|\nabla B_0(x_j)|\,\frac{H}{\kappa^2}\,.\end{equation}
Performing the translation $x\mapsto x+a_j$, we get that
\begin{equation}\label{eq:GLj}
\mathcal E_0\big(f_j,|\nabla B_0(x_j)|\Ab_{{\rm
app},\nu_j}(x-a_j); D(a_j,\ell)\big)=\mathcal G(f_j) \geq
\Er(\kappa,L;\ell)\,.\end{equation} Here $\mathcal G$ is the
functional in \eqref{eq-GL-r-new} and $\Er(\kappa,L;\ell)$ is the
ground state energy in \eqref{eq-gs-Er}.\\
Let $$R=L^{1/3}\kappa\ell \,.$$
 Now, Theorems~\ref{thm-tdl-e} and
\ref{thm-FK} applied successively  tell us that
$$\mathcal E_0\big(f_j,|\nabla
B_0(x_j)|\Ab_{{\rm app},\nu_j}(x-a_j); D(a_j,\ell)\big)\geq
\er(\nu,L;R)\geq 2R\,E(L)=L^{1/3}\kappa\ell\,E(L)\,.$$ Recall the
definition of $L$ in \eqref{eq:Lj}. We insert the aforementioned
estimate into \eqref{eq:GLj}. In that way, we infer from
\eqref{eq:lb-F-A0} and \eqref{eq:lb-A-F} the lower bound of
Proposition~\ref{prop:lb}.
\end{proof}

\begin{prop}\label{prop:lb-bnd0}
For $r>0$,  $h\in C^\infty(\R^2)$ satisfying
$\|h\|_\infty\leq 1\,,$  and $(\psi,\Ab)$   a critical point of the functional in
\eqref{eq-3D-GLf},  the following lower bound holds,
\begin{equation}\label{bnd2} \mathcal E_0\big(h\,\psi,\Ab;D(a_j,r)\cap\Omega\big)\geq
- \pi \kappa^2r^2\,.
\end{equation}
\end{prop}
\begin{proof}
Notice that all terms in $\mathcal
E_0\big(h\,\psi,\Ab;D(a_j,r)\cap\Omega\big)$ are positive except the
integral of $|h\psi|^2$. Thus,
$$\mathcal
E_0\big(h\,\psi,\Ab;D(a_j,r)\cap\Omega\big)\geq-\kappa^2\int_{\Omega\cap D(a_j,r)} |h\psi|^2\,dx\,.$$
This finishes the proof of the proposition upon using
$\|h\psi\|_\infty\leq 1$ and $\|\psi\|_\infty\leq1\,$.
\end{proof}

\begin{thm}\label{thm:lb}
Let $\Lambda>0$ and $\epsilon:\R\to\R_+$ such that
$\displaystyle\lim_{\kappa\to
\infty}\kappa\,\epsilon(\kappa)=\infty$
 and $\displaystyle\lim_{\kappa\to \infty}\epsilon(\kappa)=0\,$.

There exist $\kappa_0>0$ and a function ${\rm err}:\R\to \R$ such
that the following is true:
\begin{enumerate}
\item $\displaystyle\lim_{\kappa\to \infty}{\rm err}(\kappa)=0$\,.
\item  Let  $D\subset \Omega$ be a regular open set,
$h\in C^\infty(\overline{D})$, $\|h\|_\infty\leq 1$,
$(\psi,\Ab)$  a critical point of the functional in
\eqref{eq-3D-GLf}, $\kappa\geq\kappa_0$ and
$\epsilon(\kappa)\kappa^2 \leq H\leq \Lambda\kappa^2$.
\begin{enumerate}
\item If $H\gg \kappa^{3/2}$, then,
\begin{equation}\label{lb1}
\mathcal E_0(h\psi,\Ab;D)\geq
\kappa\left(\int_{\Gamma\cap D}\left(|\nabla
B_0(x)|\frac{H}{\kappa^2}\right)^{1/3}\,E\left(|\nabla
B_0(x)|\frac{H}{\kappa^2}\right)\,ds(x)\right)+\frac{\kappa^3}{H}{\rm err}(\kappa)
\,.
\end{equation}
\item If $H\lesssim\kappa^{3/2}$, then,
\begin{equation}\label{lb1*}
\mathcal E_0(h\psi,\Ab;D)\geq \kappa^2\int_{D}
g\left(\frac{H}{\kappa}|B_0(x)|\right)\,dx+\frac{\kappa^3}{H}{\rm
err}(\kappa) \,.
\end{equation}
\end{enumerate}
\end{enumerate}
\end{thm}
\begin{proof}
Consider three parameters
$$a\in\,(0,1)\,,\quad \ell\in\,(0,1)\,,\quad \delta\in\,(0,1)\,,$$
and define the following sets,
\begin{align*}
&D_{1}=\{x\in\Omega~:~{\rm dist}(x,\Gamma)< 2\sqrt{a}\,\ell\}\,,\\
&D_{2}=\{x\in\Omega~:~{\rm dist}(x,\Gamma)> \sqrt{a}\,\ell\}\,.\\
\end{align*}
Let $(\chi_j)$ be a partition of unity satisfying
$$ \sum_{j=1}^2\chi_j^2=1\,,\quad \sum_{j=1}^2|\nabla\chi_j|^2 \leq
C(a\ell^2)^{ -1}\,,\quad {\rm supp}\,\chi_j\subset D_j\quad(j\in\{1,2\})\,.$$
There holds the following decomposition of the energy,
$$
\mathcal E_0(h\psi,\Ab;D)\geq \mathcal
E_0(\chi_1h\psi,\Ab;D_1)+\mathcal
E_0(\chi_2h\psi,\Ab;D_j)-\sum_{j=1}^2\int_\Omega|\nabla
\chi_j|^2\,|h\psi|^2\,dx\,.
$$
The error terms are controlled using the pointwise bounds on $|h|$,
$|\psi|$, $|\nabla\chi_j|$,  and the conditions on the support of
$\chi_j$. We obtain the following lower bound,
\begin{equation}\label{eq:lb-IMS}
\mathcal E_0(h\psi,\Ab;D)\geq \mathcal
E_0(\chi_1h\psi,\Ab;D_1)+\mathcal
E_0(\chi_2h\psi,\Ab;D_2)
-C(\sqrt{a}\,\ell)^{-1}\,.
\end{equation}
~\\
\subsection*{The regime $H\gg\kappa^{3/2}$}~\\
In this regime, we shall see that  $\mathcal
E_0(\chi_1h\psi,\Ab;D_1)$ is the leading term and
$\mathcal E_0(\chi_2h\psi,\Ab;D_2)$ is an error term.\\

\subsubsection*{Lower bound of the term $\mathcal
E_0(\chi_1h\psi,\Ab;D_1)$.}~\\
Consider a constant $c>0$ and distinct points $(a_j)$ in $\Gamma$
such that,
$$\forall~j\,,\quad \ell-a\ell\leq {\rm dist}(a_j,a_{j+1})\leq
\ell-a\ell\,.$$
Choose the constant $a$ sufficiently small so that
$$D_1=\{x\in\Omega~:~{\rm dist}(x,\Gamma)<
2\sqrt{a}\,\ell\}\subset \bigcup_jD(a_j,\ell)\,.$$ Consider a partition of unity satisfying
$$ \sum_jf_j^2=1~{\rm in~}D_1\,,\quad{\rm supp}\,f_j\subset D(a_j,\ell)\,,\quad
\sum_j|\nabla f_j|^2\leq \frac{C}{a^2\ell^2}\,.$$ Notice that
the support of each $\nabla f_j$ is in $D(a_j,\ell)\cap
D(a_{j+1},\ell)\cap D(a_{j-1},\ell)$ and that the points $(a_j)$ are
selected such that the last domain  has an area proportional to
$\sqrt{a}\,\ell\times a\ell=a\sqrt{a}\,\ell^2$.

The partition of unity $(f_j)$ allows us to decompose the energy as
follows,
\begin{align}\label{eq:lb-bnd1}
\mathcal E_0(\chi_1h\psi,\Ab;D_1)&\geq \sum_j\mathcal E_0(f_j\,\chi_1h\psi,\Ab;D_1)-\sum_j\big\|\,|\nabla f_j|\,\chi_1h\psi\big\|_2^2\nonumber\\
&\geq \sum_j\mathcal E_0(h_j\,\psi,\Ab;D_1)-\frac{C}{\sqrt{a}\,\ell}\,,
\end{align}
where $h_j=f_j\,\chi_1h\psi$ is supported in $D\cap D(a_j,\ell)\,$. \\
If
$D(a_j,\ell)\cap\partial\Omega\not=\emptyset$, then we can apply
Proposition~\ref{prop:lb-bnd0}. Since $\Gamma\cap\partial\Omega$ is
a finite set, then we get,
$$\sum_{D(a_j,\ell)\cap\partial\Omega\not=\emptyset}\mathcal E_0(h_j\,\psi,\Ab;D_1)\geq
-C\kappa^2\ell^2\,.$$
We select the parameter $\ell$ as follows,
\begin{equation}\label{eq:lb-ell}
\ell=\delta H^{-1/3}\,.
\end{equation}%
In that way, we obtain,
$$\ell\ll1\,,\quad\kappa^2\ell^2\ll\frac{\kappa^3}{H}\,,\quad \frac1{\sqrt{a}\,\ell}\ll\frac{\kappa^3}{H}\,,$$
and
\begin{equation}\label{eq:lb-1}
\sum_{D(a_j,\ell)\cap\partial\Omega\not=\emptyset}\mathcal E_0(h_j\,\psi,\Ab;D_1)\geq
\frac{\kappa^3}{H}\,o(1)\quad(\kappa\to\infty)\,.
\end{equation}
 If $D(a_j,\ell)\subset
D^c$, then $h_j=0$ and
$$\mathcal E_0(h_j\,\psi,\Ab;D_1)=0\,.$$
Now,  if $j\in \mathcal I=\{j~:~D(a_j,\ell)\subset\Omega~{\rm
and}~D(a_j,\ell)\cap D\not=\emptyset\}$, then
we can apply Proposition~\ref{prop:lb} and get,
\begin{multline*}
\sum_{j\in\mathcal I}\mathcal E_0(h_j\,\psi,\Ab;D_1)\geq
(1-\delta)\,2\ell\,\kappa\sum_{j\in\mathcal I}\left(|\nabla
B_0(x_j)|\frac{H}{\kappa^2}\right)^{1/3}\,E\left(|\nabla
B_0(x_j)|\frac{H}{\kappa^2}\right) \\
-C\left(\delta\kappa^2+\delta^{-1}\left(\frac{\kappa^3}{H}\ell^{2\alpha}+\kappa^2H^2\ell^6\right)\right)\int_\Omega|h\psi|^2\,dx\,,\end{multline*}
where, for all $j$, $x_j$ is an arbitrary point
in $\overline{D(a_j,\ell)}$.\\
Thanks to Proposition~\ref{thm:ob} and the choice of $\ell$ in
\eqref{eq:lb-ell}, we get further,
\begin{multline*}
\sum_{j\in\mathcal I}\mathcal E_0(h_j\,\psi,\Ab;D_1)\geq
(1-\delta)\,2\ell\,\kappa\sum_{j\in\mathcal I}\left(|\nabla
B_0(x_j)|\frac{H}{\kappa^2}\right)^{1/3}\,E\left(|\nabla
B_0(x_j)|\frac{H}{\kappa^2}\right) \\
-C\left(\delta+\delta^{2\alpha-1}\frac\kappa H\,H^{-2\alpha/3}\right)\frac{\kappa^3}{H}\,.
\end{multline*}
Theorem~\ref{thm:Lto0} allows us to write,
\begin{equation}\label{eq:thmLto0}
\left(|\nabla
B_0(x_j)|\frac{H}{\kappa^2}\right)^{1/3}\,E\left(|\nabla
B_0(x_j)|\frac{H}{\kappa^2}\right)\leq
C\frac{\kappa^2}{H}\,.\end{equation} Consequently,
\begin{multline*}
\sum_{j\in\mathcal I}\mathcal E_0(h_j\,\psi,\Ab;D_1)\geq
2\ell\,\kappa\sum_{j\in\mathcal I}\left(|\nabla
B_0(x_j)|\frac{H}{\kappa^2}\right)^{1/3}\,E\left(|\nabla
B_0(x_j)|\frac{H}{\kappa^2}\right)
\\-C\left(\delta+\delta^{2\alpha-1}\frac\kappa H\,H^{-2\alpha/3}\right)\frac{\kappa^3}{H}\,.\end{multline*}
Inserting this and \eqref{eq:lb-1} into \eqref{eq:lb-bnd1}, and using that $(\sqrt{a}\,\ell)^{-1}\ll\frac{\kappa^3}{H}$, we get,
\begin{multline}\label{eq:lb-Rsum} \mathcal
E_0(\chi_1h\psi,\Ab;D_1)\geq \kappa\sum_{j\in\mathcal I}2\ell\left(|\nabla
B_0(x_j)|\frac{H}{\kappa^2}\right)^{1/3}\,E\left(|\nabla
B_0(x_j)|\frac{H}{\kappa^2}\right)\\
 -C\left(\delta+\delta^{2\alpha-1}\frac\kappa H\,H^{-2\alpha/3}\right)\frac{\kappa^3}{H}\,.\end{multline} Thanks to \eqref{eq:thmLto0},
the sum
\begin{equation}\label{eq:Rsum'}
\sum_{j\in\mathcal I}2\ell\left(|\nabla
B_0(x_j)|\frac{H}{\kappa^2}\right)^{1/3}\,E\left(|\nabla
B_0(x_j)|\frac{H}{\kappa^2}\right)
\end{equation}
is of order $\kappa^2/H$. Let $\eta>0$. Select $\ell_0$ sufficiently
small such that, for all $\ell\in (0,\ell_0)$, the arc-length
measure of $D(a_j,\ell)\cap\Gamma$ along $\Gamma$ satisfies,
$$2\ell-\ell\frac{\eta}2\leq |D(a_j,\ell)\cap\Gamma|\leq
2\ell+\ell\frac{\eta}2\,.$$ Thus, replacing $2\ell$ by
$|D(a_j,\ell)\cap\Gamma|$ in the sum in \eqref{eq:Rsum'} produces an
error of order $\eta\ell$. Now, select
$x_j\in\overline{D(a_j,\ell)}$ such that
$$|\nabla
B_0(x_j)|^{1/3}\,E\left(|\nabla
B_0(x_j)|\frac{H}{\kappa^2}\right)=\max_{\overline
{D(a_j,\ell)}}|\nabla B_0(x)|^{1/3}\,E\left(|\nabla
B_0(x)|\frac{H}{\kappa^2}\right)\,.$$ In that way, the sum in
\eqref{eq:Rsum'} satisfies,
\begin{multline}\label{eq:sum=int}
\sum_{j\in\mathcal I}2\ell\left(|\nabla
B_0(x_j)|\frac{H}{\kappa^2}\right)^{1/3}\,E\left(|\nabla
B_0(x_j)|\frac{H}{\kappa^2}\right)\\
\geq \sum_{j\in\mathcal
I}\left(\int_{D(a_j,\ell)\cap\Gamma}\left(|\nabla
B_0(x_j)|\frac{H}{\kappa^2}\right)^{1/3}\,E\left(|\nabla
B_0(x_j)|\frac{H}{\kappa^2}\right)\,dx\right)-C\eta\frac{\kappa^2}{H}\,.
\end{multline}
Recall \eqref{eq:thmLto0}. Since the balls $\big(D(a_j,\ell)\big)$
overlap in a region of length $\mathcal O(a\ell)$, and the number of
these balls is inversely proportional to $\ell$, then
\begin{multline*}
\sum_{j\in\mathcal I}\left(\int_{D(a_j,\ell)\cap\Gamma}\left(|\nabla
B_0(x_j)|\frac{H}{\kappa^2}\right)^{1/3}\,E\left(|\nabla
B_0(x_j)|\frac{H}{\kappa^2}\right)\,dx\right)\\
\geq \int_{D\cap\Gamma}\left(|\nabla
B_0(x_j)|\frac{H}{\kappa^2}\right)^{1/3}\,E\left(|\nabla
B_0(x_j)|\frac{H}{\kappa^2}\right)\,dx-Ca\frac{\kappa^2}{H}\,.
\end{multline*}
Inserting this into \eqref{eq:sum=int}, then inserting the resulting
inequality into \eqref{eq:lb-Rsum}, we get,
\begin{multline}
\mathcal E_0(\chi_1h\psi,\Ab;D_1)\geq \kappa\int_{\Gamma\cap D}\left(|\nabla
B_0(x)|\frac{H}{\kappa^2}\right)^{1/3}\,E\left(|\nabla
B_0(x)|\frac{H}{\kappa^2}\right)\,ds(x)\\
 -C\left(a+\delta+\delta^{2\alpha-1}\frac\kappa H\,H^{-2\alpha/3}+\eta\right)\frac{\kappa^3}{H}\,.
\end{multline}
Recall that $\alpha>0$. Taking $\kappa\to\infty$, we get,
\begin{multline*}
\liminf_{\kappa\to\infty}\frac{H}{\kappa^3}\left\{\mathcal
E_0(\chi_1h\psi,\Ab;D_1) -\kappa\int_{\Gamma\cap D}\left(|\nabla
B_0(x)|\frac{H}{\kappa^2}\right)^{1/3}\,E\left(|\nabla
B_0(x)|\frac{H}{\kappa^2}\right)\,ds(x)\right\}\\
\geq
-C(a+\delta+\eta)\,. \end{multline*} Taking $\eta\to0_+$, we obtain,
$$\liminf_{\kappa\to\infty}\frac{H}{\kappa^3}\left\{\mathcal E_0(\chi_1h\psi,\Ab;D_1)
-\kappa\int_{\Gamma\cap D}\left(|\nabla
B_0(x)|\frac{H}{\kappa^2}\right)^{1/3}\,E\left(|\nabla
B_0(x)|\frac{H}{\kappa^2}\right)\,ds(x)\right\}\geq -C(a+\delta)\,,$$ i.e. when $\kappa$ is sufficiently large,
\begin{equation}\label{eq:lb-bnd2}
\mathcal E_0(\chi_1h\psi,\Ab;D_1)\geq \kappa\int_{\Gamma\cap D}\left(|\nabla
B_0(x)|\frac{H}{\kappa^2}\right)^{1/3}\,E\left(|\nabla
B_0(x)|\frac{H}{\kappa^2}\right)\,ds(x)  -2C(\delta+a)\frac{\kappa^3}H\,.
\end{equation}
~\\
\subsubsection*{Lower bound of the term $\mathcal
E_0(\chi_2h\psi,\Ab;D_2)$}~\\
%
%
Since $H\gg\kappa^{3/2}$, the parameter $\ell$ defined in
\eqref{eq:lb-ell} satisfies,
$$
\ell=\delta H^{-1/3}=\delta\frac \kappa H \frac{H^{1/3}}\kappa\gg \frac{\kappa}H\,.
$$
Thanks to the exponential decay in Theorem~\ref{thm:exdec}, there
holds,
$$\kappa^2\int_{D_2}|\psi|^2\,dx\ll\frac{\kappa^3}H\,,
$$ and
\begin{equation}\label{eq:lb-enD2-C1}
\mathcal E_0(\chi_2h\psi,\Ab;D_2)\geq -\frac{\kappa^3}H \,o(1)\quad(\kappa\to\infty)\,.
\end{equation}
Inserting \eqref{eq:lb-enD2-C1} and \eqref{eq:lb-bnd2} into
\eqref{eq:lb-IMS}, we get,
$$
\mathcal E_0(h\psi,\Ab;D)\geq \kappa\int_{\Gamma\cap D}\left(|\nabla
B_0(x)|\frac{H}{\kappa^2}\right)^{1/3}\,E\left(|\nabla
B_0(x)|\frac{H}{\kappa^2}\right)\,ds(x)  -C\big(\delta+a+o(1)\big)\frac{\kappa^3}H\,.
$$
Taking the successive limits
$$\liminf_{\kappa\to\infty}\,,\quad \lim_{\delta\to0_+}\,,\quad
\lim_{a\to0+}\,,$$ we arrive at
\begin{equation}\label{eq:lb-conc-C1}
\mathcal E_0(h\psi,\Ab;D)\geq \kappa\int_{\Gamma\cap D}\left(|\nabla
B_0(x)|\frac{H}{\kappa^2}\right)^{1/3}\,E\left(|\nabla
B_0(x)|\frac{H}{\kappa^2}\right)\,ds(x)  -\frac{\kappa^3}H\,o(1)\,.
\end{equation}
This finishes the proof of Theorem~\ref{thm:lb} in the case $H\gg
\kappa^{3/2}$.\\
\subsection*{The regime $H\lesssim \kappa^{3/2}$.}~\\
In this regime,  we shall see that   $\mathcal
E_0(\chi_1h\psi,\Ab;D_1)$ is an error term and $\mathcal
E_0(\chi_2h\psi,\Ab;D_2)$ is the leading term.

Since $ H\lesssim\kappa^{3/2}$, the parameter $\ell$ introduced in
\eqref{eq:lb-ell} satisfies
$$\ell\lesssim \delta\frac\kappa H\,.$$
Consequently, we have,
\begin{equation}\label{eq:lb-1*}
\mathcal E_0(\chi_1h\psi,\Ab;D_1)\geq -\kappa^2\int_\Omega|\chi_1h\psi|^2\,dx
\geq C \kappa^2\ell \geq -\delta\frac{\kappa^3}{H}\,.
\end{equation}
Unlike the regime $H\gg\kappa^{3/2}$, we can no more ignore the
energy in $\{\sqrt{a}\,\ell\leq {\rm dist}(x,\Gamma)\leq
\frac{\kappa}H\}$.\\
We introduce the two parameters,
\begin{equation}\label{eq:lb-m}
m>1\quad {\rm and}\quad \zeta\in(0,\ell)\,,
\end{equation}
and the domain,
\begin{equation}\label{eq:lb-dom}
U=\big\{\,x\in D_2~:~{\rm dist}(x,\Gamma)\geq m\frac\kappa H\,\big\}\,.
\end{equation}
Thanks to the exponential decay in Theorem~\ref{thm:exdec}, we get,
$$\kappa^2\int_{U}|\psi|^2\,dx\leq
C{e^{-2m\,m_0}}\, \frac{\kappa^3}H\,.
$$ and
\begin{equation}\label{eq:lb-enU}
\mathcal E_0(\chi_2h\psi,\Ab;U)\geq -C{e^{-2m\,m_0}}\,\frac{\kappa^3}H\,.
\end{equation}
Our next task is to determine a lower bound of the energy $\mathcal
E_0(\chi_2h\psi,\Ab;D_2\setminus U)$. Consider  for $\zeta\in (0,1)$ the lattice of
squares $(\mathcal Q_{k,\zeta})_k$ generated by the square
$$\mathcal Q_\zeta=\big(-\frac\zeta2,\frac\zeta2\big)\times
\big(-\frac\zeta2,\frac\zeta2\big)\,.$$
Let
\begin{align*}
&\mathcal J_{\rm blk}=\{k~:~\mathcal Q_{k,\zeta}\subset
D_2\setminus U\quad{\rm and}~\mathcal
Q_{k,\zeta}\cap\partial\Omega=\emptyset\}\,,\\
&\mathcal J_{\rm bnd,1}=\{k~:~\mathcal Q_{k,\zeta}\cap
(D_2\setminus U)\not=\emptyset\quad{\rm and}\quad \mathcal
Q_{k,\zeta}\cap\partial\Omega=\emptyset\}\,,\\
&\mathcal J_{\rm bnd,2}=\{k~:~\mathcal Q_{k,\zeta}\subset
D_2\setminus U\quad{\rm and}~\mathcal
Q_{k,\zeta}\cap\partial\Omega\not=\emptyset\}\,.
\end{align*} We have the obvious
decomposition,
\begin{equation}\label{eq:lb-dec}
\mathcal
E_0(\chi_2h\psi,\Ab;D_2\setminus U)\geq \sum_{k\in \mathcal J_{\rm blk}}\mathcal
E_0(\chi_2h\psi,\Ab;\mathcal Q_{k,\zeta})+\sum_{j=1}^2\sum_{k\in\mathcal J_{{\rm bnd},j}}\mathcal
E_0(\chi_2h\psi,\Ab;\mathcal Q_{k,\zeta}\cap(D_2\setminus U))\,.
\end{equation}
Since $\Gamma\cap\partial\Omega$ is a finite set, then $N={\rm
Card}\mathcal J_{\rm bnd,2}$ is bounded independently of $\kappa$.
Now, the terms corresponding to $k\in\mathcal J_{{\rm bnd},j}$ are
easily estimated as follows,
\begin{equation}\label{eq:lb-dec1}
\sum_{j=1}^2\sum_{k\in\mathcal J_{{\rm bnd},j}}\mathcal
E_0(\chi_2h\psi,\Ab;\mathcal Q_{k,\zeta}\cap(D_2\setminus U))\geq -\kappa^2\int_{\{{\rm dist}(x,\partial\Gamma)\leq C\zeta}|\psi|^2\,dx-N\kappa^2\zeta^2
\geq -C\kappa^2\zeta\,.
\end{equation}
 For all $k$, let $x_k$ be the center of the
square $\mathcal Q_{k,\zeta}$ and $a_k$ an arbitrary point in
$\overline{\mathcal Q_{k,\zeta}}$. Repeating the proof of
Proposition~\ref{prop:lb}, we get, for all $k\in\mathcal J$ and
$\eta\in(0,1)$,
\begin{multline}\label{eq:lb-ensq}
\mathcal E_0(\chi_2h\psi,\Ab;\mathcal Q_{k,\zeta})\geq
(1-\eta)\mathcal E_0(\chi_2h\psi\,e^{-i\kappa H
u_k},B_0(a_k)\Ab_0(x-x_k);\mathcal Q_{k,\zeta})
\\-C\left(\eta\kappa^2+\eta^{-1}\left(\frac{\kappa^3}{H}\zeta^{2\alpha}+\kappa^2H^2\zeta^4\right)\right)\|\psi\|_{L^2(\mathcal
Q_{k,\zeta})}^2\,,
\end{multline}
where $u_k$ is a gauge function.\\
We select the parameter $\zeta$ as follows
\begin{equation}\label{eq:lb-zeta}
\zeta=\eta H^{-1/2}\,.
\end{equation}
Clearly,  $\zeta$ satisfies,
$$ \zeta\ll\ell\lesssim \delta\frac\kappa H \ll 1\,,\quad\kappa^2\zeta\ll\frac\kappa H\,,\quad  H^2\zeta^4=\eta^4\,,\quad
\zeta\sqrt{\kappa H|B_0(a_k)|}\gtrsim \zeta\sqrt{\kappa H\sqrt{a}\,\ell}\,\gtrsim\eta\delta^{1/2} a^{1/4}\,.$$
Applying a scaling and a translation, we may use \eqref{eq:g'} and
get,
$$
\begin{array}{l}
\frac{1}{{\left(\zeta\sqrt{\kappa H|B_0(a_k)|}\right)^2} }  \mathcal E_0(\chi_2h\psi\,e^{-i\kappa H
u_k},B_0(a_k)\Ab_0(x-x_k);\mathcal
Q_{k,\zeta})\qquad \\ \quad\quad \quad\quad
\geq\frac{\kappa}{H|B_0(a_k)|}\left(g\left(\frac{H}{\kappa}|B_0(a_k)|\right)- C\frac{\sqrt{\frac{H}{\kappa}|B_0(a_k)|}}{\zeta\sqrt{\kappa
H|B_0(a_k)|}}\right)\,.
\end{array}
$$
We insert this  into \eqref{eq:lb-ensq}, sum over $k\in\mathcal
J_{\rm blk}$ and use Proposition~\ref{thm:ob} to get,
$$
\sum_{k\in\mathcal J_{\rm blk}}\mathcal E_0(\chi_2h\psi,\Ab;\mathcal Q_{k,\zeta})\geq
\zeta^2\kappa^2\sum_{k\in\mathcal J_{\rm blk}}\left(g\left(\frac{H}{\kappa}|B_0(a_k)|\right)-\frac{C}{\zeta\kappa}\right)
-\left(C\eta+o(1)\right)\frac{\kappa^3}{H}\,.
$$
The sum in the inequality above  becomes a lower Riemann sum if for each $k$  the
point $(a_k)$ is selected in $\overline{Q_{k,\ell}}$ as follows,
$$|B_0(a_k)|= \max\{|B_0(x)|~:~x\in\overline{\mathcal
Q_{k,\ell}}\}\,.$$ Notice that $\mathcal N_{\rm blk}={\rm Card}\,\mathcal J_{\rm blk}$ satisfies
$\mathcal N_{\rm blk}\times\zeta^2\approx | D_2\setminus U|$ as $\zeta\to0$ and
$$|D_2\setminus U|=|\{\sqrt{a}\,\ell\leq {\rm dist}(x,\Gamma)\leq m\frac\kappa H\}|\leq Cm\frac\kappa H\,.$$
Consequently, we have,
$$
\sum_{k\in\mathcal J_{\rm blk}}\mathcal E_0(\chi_2h\psi,\Ab;\mathcal Q_{k,\zeta})\geq
\kappa^2\int_{\mathcal D_\zeta} g\left(\frac{H}{\kappa}|B_0(x)|\right)\,dx- m\frac{C\kappa}{\zeta}\frac\kappa H-\left(C\eta+o(1)\right)\frac{\kappa^3}{H}\,,
$$
where
$$\mathcal D_{\zeta}=\bigcup_{k\in\mathcal J_{\rm blk}}\overline{\mathcal
Q_{k,\zeta}}\subset D_2\setminus U\,.$$ Since the function $g$ is non-positive, then
we get that,
\begin{align}\label{eq:lb-enD2-C2} \sum_{k\in\mathcal
J_{\rm blk}}\mathcal E_0(\chi_2h\psi,\Ab;\mathcal Q_{k,\zeta})&\geq
\kappa^2\int_{\mathcal D_2\setminus
U}g\left(\frac{H}{\kappa}|B_0(x)|\right)\,dx- Cm \, \frac{\kappa}{\zeta}\, \frac\kappa
H-\left(C\eta+o(1)\right)\frac{\kappa^3}{H}\nonumber\\
&\geq \kappa^2\int_{\mathcal D_2\setminus
U}g\left(\frac{H}{\kappa}|B_0(x)|\right)\,dx-C\, m\eta^{-1}\frac{H^{1/2}}{\kappa}\frac{\kappa^3}{H}-\left(C\eta+o(1)\right)\frac{\kappa^3}{H}\,.
\end{align}
%
We insert \eqref{eq:lb-enD2-C2} 
and \eqref{eq:lb-dec1} into
\eqref{eq:lb-dec}. Since $g\left(\frac{H}{\kappa}|B_0(x)|\right)=0$
in $\{|B_0(x)|\geq \frac H\kappa\}$ and $H\lesssim \kappa^{3/2}$, it results the inequality,
\begin{equation}\label{eq:lb-enD-C2'}
\mathcal E_0(\chi_2h\psi,\Ab;D_2\setminus U)\geq \kappa^2\int_{D}
g\left(\frac{H}{\kappa}|B_0(x)|\right)\,dx-\left(C\eta+o(1)\right)\frac{\kappa^3}{H}\,,\quad(\kappa\to\infty)\,.
\end{equation}
Combining \eqref{eq:lb-enD-C2'} and \eqref{eq:lb-enU}, we get
$$\mathcal E_0(\chi_2h\psi,\Ab;D_2)\geq
\kappa^2\int_{D}
g\left(\frac{H}{\kappa}|B_0(x)|\right)\,dx-C{e^{-2m\,m_0}} \, \frac{\kappa^3}{H}
-\left(C\eta+o(1)\right)\frac{\kappa^3}{H}\,.$$ Now, we insert this inequality and \eqref{eq:lb-1*} into \eqref{eq:lb-IMS} to get,
$$\mathcal E_0(h\psi,\Ab;D)\geq\kappa^2\int_{D}
g\left(\frac{H}{\kappa}|B_0(x)|\right)\,dx
-C\big(a+\delta+\eta+{e^{-2m\,m_0}}+o(1)\big)\frac{\kappa^3}{H}\,.$$ By taking the successive limits,
$$\liminf_{\kappa\to\infty}\,,\quad \lim_{a\to0_+}\,,\quad
\lim_{\delta\to0_+}\,,\quad \lim_{\eta\to0_+}\,,\quad
\lim_{m\to\infty}\,,$$ we get
$$\mathcal E_0(h\psi,\Ab;D)\geq\kappa^2\int_{D}
g\left(\frac{H}{\kappa}|B_0(x)|\right)\,dx
-\frac{\kappa^3}{H}\,o(1)\,,$$ which finishes the proof of Theorem~\ref{thm:lb} in the regime $H\lesssim\kappa^{3/2}$.
%
%
\end{proof}

We get by applying Theorem~\ref{thm:lb} with $D=\Omega$\,:

\begin{corol}\label{corl:lb}
Let $\Lambda>0$ and $\epsilon:\R\to\R_+$ such that
$\displaystyle\lim_{\kappa\to
\infty}\kappa\,\epsilon(\kappa)=\infty$
 and $\displaystyle\lim_{\kappa\to \infty}\epsilon(\kappa)=0\,$.

There exist $\kappa_0>0$ and a function ${\rm err}:\R\to \R$ such
that the following is true:
\begin{enumerate}
\item $\displaystyle\lim_{\kappa\to \infty}{\rm err}(\kappa)=0$\,.
\item  Let $\kappa\geq\kappa_0$,
$\epsilon(\kappa)\kappa^2 \leq H\leq \Lambda\kappa^2$ and
$(\psi,\Ab)$  be a critical point of the functional in
\eqref{eq-3D-GLf}.
\begin{enumerate}
\item If $H\gg \kappa^{3/2}$, then,
\begin{equation}\label{lb2}
\mathcal E_0(h\psi,\Ab)\geq
\kappa\left(\int_{\Gamma}\left(|\nabla
B_0(x)|\frac{H}{\kappa^2}\right)^{1/3}\,E\left(|\nabla
B_0(x)|\frac{H}{\kappa^2}\right)\,ds(x)\right)+\frac{\kappa^3}{H}{\rm err}(\kappa)
\,.
\end{equation}
\item If $H\lesssim\kappa^{3/2}$, then,
\begin{equation}\label{lb2*}
\mathcal E_0(h\psi,\Ab)\geq \kappa^2\int_{\Omega}
g\left(\frac{H}{\kappa}|B_0(x)|\right)\,dx+\frac{\kappa^3}{H}{\rm
err}(\kappa) \,.
\end{equation}
\end{enumerate}
\end{enumerate}
\end{corol}

We conclude this section with the
\begin{proof}[Proof of Theorem~\ref{thm:HK}]
We have just to combine the conclusions of Theorem~\ref{thm:up} and
Corollary~\ref{corl:lb}.\end{proof}

\section{Local energy estimates}

\subsection{Preliminaries} Let $D\subset \Omega$ be an open set with a smooth boundary such that $\partial D\cap\Gamma$ is a finite set. Let
$\rho_0\in(0,1)$, $\rho\in(0,\rho_0)$ and
$$D_\rho=\{x\in \Omega~:~{\rm dist}(x, D)<\rho\}\,.$$
We select $\rho_0$ sufficiently small so that the boundary of
$\partial D_\rho$ is smooth.\\
Let $h_1\in C_c^\infty(D_\rho)$ and $h_2\in C^\infty(\mathbb R^2)$
be functions satisfying
$$0\leq h_1\leq 1\,,\quad|\nabla h_1| + |\nabla h_2| \leq \frac{C}\rho\quad{\rm in~}\R^2\,,\quad h_1=1\quad{\rm in}~
D_{\rho}\,,\,{\rm and}\;  h_1^2 +h_2^2 =1\,.$$
Notice that
$${\rm supp}\,h_2\subset \overline D^c\,.$$
Let $(\psi,\Ab)$ be a  minimizer of \eqref{eq-3D-GLf}. We will
estimate the following energy
$$\mathcal
E_0(\psi,\Ab;D)=\int_D\left(|\nabla-i\kappa H\Ab)\psi|^2-{\kappa^2}|\psi|^2+\frac{\kappa^2}2|\psi|^4\right)\,dx\,.$$
Notice that we have the following decomposition of the energy
$$\mathcal E_0(\psi,\Ab;\Omega)\geq \mathcal
E_0(h_1\psi,\Ab;D_\rho)+\mathcal E_0(h_2\psi,\Ab;\overline
D^c)-\frac{C}{\rho^2}\int_\Omega|\psi|^2\,dx\,.$$ Now we use the
estimate in Proposition~\ref{thm:ob} and write,
\begin{equation}\label{eq:dec}
\mathcal E_0(\psi,\Ab;\Omega)\geq \mathcal
E_0(h_1\psi,\Ab;D_\rho)+\mathcal E_0(h_2\psi,\Ab;\overline
D^c)-\frac{C}{\rho^2}\frac\kappa H\,.
\end{equation}
Recall that we deal with two separate regimes:\\
\begin{equation*}\left\{
\begin{array}{lll}
&{\bf Regime~I:}~&\kappa\ll H\lesssim\kappa^{3/2}\,;\\
&{\bf Regime~II:}~&\kappa^{3/2}\ll H\lesssim\kappa^2\,.
\end{array}
\right.
\end{equation*}
We define the
quantity $C_0(\kappa,H;D)$ as follows:
\begin{equation}\label{eq:gse-D}
C_0(\kappa,H;D)=\left\{\begin{array}{ll}
\displaystyle\kappa\left(\int_{D\cap\Gamma}\left(|\nabla
B_0(x)|\frac{H}{\kappa^2}\right)^{1/3}\,E\left(|\nabla
B_0(x)|\frac{H}{\kappa^2}\right)\,ds(x)\right)&{\rm in~}{\rm Regime~I}\,,\\
\displaystyle\kappa^2\int_{D}g\left(\frac H\kappa|B_0(x)|\right)\,dx&{\rm in~Regime~II}\,.
\end{array}\right.
\end{equation}
Notice that, in Regimes~I~and~II, the result of Theorem~\ref{thm:HK} reads as follows,
$$\E=C_0(\kappa,H;\Omega)+\frac{\kappa^3}{H}\,o(1)\,,\quad(\kappa\to\infty)\,,$$

\subsection{Upper bound}
The results in this section are valid under the assumption that
$(\psi,\Ab)$ is a minimizer of the functional in
\eqref{eq-3D-GLf}.\\
%
We have $\mathcal E_0(\psi,\Ab;\Omega)\leq \E$. Using $|\psi|\leq 1$
and the upper bound in Theorem~\ref{thm:up}, we get,
$$
\mathcal E_0(h_1\psi,\Ab;D_\rho)+\mathcal E_0(h_2\psi,\Ab;\overline D^c)
\leq C_0(\kappa,H;\Omega)+\frac{\kappa^3}{H}{\rm
err}(\kappa)+\frac{C}{\rho^2}\frac\kappa H\,.$$
Using Theorem~\ref{thm:lb}, we may write,
$$\mathcal E_0(h_2\psi,\Ab;D^c)\geq
C_0(\kappa,H;\overline D^c)+\frac{\kappa^3}{H}{\rm err}(\kappa)\,.$$
As a consequence, we have,
$$
\mathcal E_0(h_1\psi,\Ab;D_\rho) \leq
C_0(\kappa,H;D)+\frac{\kappa^3}{H}{\rm err}(\kappa)
+\frac{C}{\rho^2}\frac\kappa H\,.
$$
Since $h_1=1$ in $D$, we get the simple decomposition of the energy,
$$\mathcal E_0(\psi,\Ab;D)=\mathcal E_0(h_1\psi,\Ab;D_\rho)-\mathcal E_0(h_1\psi,\Ab;D_\rho\setminus
D)\,.$$
Since $\|h_1\|_\infty\leq 1$ and
%
 the boundary of $D_\rho\setminus D$ is smooth,
we get in light of Theorem~\ref{thm:lb},
$$\mathcal E_0(h_1\psi,\Ab;D_\rho\setminus D)\geq C_0(\kappa,H;D_\rho\setminus D)+\frac{\kappa^3}{H}{\rm
err}_\rho(\kappa)\,.$$ In light of the upper bound in
Theorem~\ref{thm:Lto0}, we have, $$\left(\int_{(D_\rho\setminus
D)\cap\Gamma}\left(|\nabla
B_0(x)|\frac{H}{\kappa^2}\right)^{1/3}\,E\left(|\nabla
B_0(x)|\frac{H}{\kappa^2}\right)\,ds(x)\right)\leq
C\frac{\kappa^2}{H}\,\rho\,.$$ In the same vein, since $g(b)$ is
bounded and vanishes when $b\geq 1$, then,
$$\int_{D_\rho\setminus
D}g\left(\frac{H}\kappa |B_0(x)|\right)\,dx\leq
C\rho\frac\kappa H\,.$$
 As a consequence, we get,
$$C_0(\kappa,H;D_\rho\setminus D)\leq C\frac{\kappa^3}{H}\,\rho\,,$$
and
$$
\mathcal E_0(\psi,\Ab;D)\leq
C_0(\kappa,H;D)+\frac{\kappa^3}{H}\Big(C\rho+{\rm
err}_\rho(\kappa)\Big)+\frac{C}{\rho^2}\frac\kappa H\\
$$
Sending $\kappa$ to infinity, we deduce that,
$$\limsup_{\kappa\to \infty}\frac{H}{\kappa^3}\left\{\mathcal E_0(\psi,\Ab;D_\rho)-C_0(\kappa,H;D)\right\}\leq C\rho\,.
$$
Next, we send $\rho$ to $0_+$ and get,
\begin{equation}\label{eq:loc-ub}
\limsup_{\kappa\to \infty}\frac{H}{\kappa^3}\left\{\mathcal E_0(\psi,\Ab;D)-C_0(\kappa,H;D)\right\}\leq 0\,.
\end{equation}
Notice that the upper bound in \eqref{eq:loc-ub} is valid for any
open set $D\subset \Omega$ with smooth boundary. In particular, it
is true when $D$ is replaced by
$\overline{D}^c=\Omega\setminus {\overline D}$, i.e.
\begin{equation}\label{eq:loc-ub'}
\limsup_{\kappa\to \infty}\frac{H}{\kappa^3}\left\{\mathcal E_0(\psi,\Ab;\overline{D}^c)-C_0(\kappa,H;\overline{D}^c)\right\}\leq 0\,.
\end{equation}
\subsection{Lower bound.}
We continue to assume that $(\psi,\Ab)$ is  a minimizer of the
functional in \eqref{eq-3D-GLf}. We will give a lower bound of the
energy $\mathcal E_0(\psi,\Ab;D)$. We plug the lower bound in
Corollary~\ref{corl:lb} into the following simple decomposition of
the energy,
$$\mathcal E_0(\psi,\Ab;D)+\mathcal E_0(\psi,\Ab;\overline{D}^c)=\mathcal
E_0(\psi,\Ab;\Omega)\,.$$ In that way, we get,
$$\mathcal E_0(\psi,\Ab;D)\geq C_0(\kappa,H;\Omega)-\mathcal
E_0(\psi,\Ab;\overline{D}^c)+\frac{\kappa^3}{H}{\rm
err}(\kappa)\,.$$ Notice the following simple decomposition of
the term on the right hand side,
$$
\mathcal E_0(\psi,\Ab;D)\geq C_0(\kappa,H;D)+\frac{\kappa^3}{H}{\rm
err}(\kappa)\\
-\left\{\mathcal
E_0(\psi,\Ab;\overline{D}^c)-C_0(\kappa,H;\overline{D}^c)\right\}\,.$$
 Now we send $\kappa$ to $\infty$ and using
\eqref{eq:loc-ub'}, we get,
\begin{equation}\label{eq:loc-lb}
\liminf_{\kappa\to \infty}\frac{H}{\kappa^3}\left\{\mathcal E_0(\psi,\Ab;D)-C_0(\kappa,H;D)\right\}\geq0\,.
\end{equation}

\subsection{Conclusion for the  local energy}\label{sec:locen}~\\
 Combining \eqref{eq:loc-ub} and \eqref{eq:loc-lb}, we get, in the two regimes we are considering,  that the local energy in $D$ of a minimizer $(\psi,\Ab)$
satisfies,
\begin{equation}\label{con74}
\mathcal
E_0(\psi,\Ab;D)= C_0(\kappa,H;D)+\frac{\kappa^3}{H}\,o(\kappa)\,,
\end{equation}
where $C_0(\kappa,H;D)$ is introduced in \eqref{eq:gse-D}.

\section{Proof of Theorem~\ref{thm:HK-loc}}
The proof of (1) in Theorem~\ref{thm:HK-loc} is a simple combination
of the upper bound in Theorem~\ref{thm:up} and the lower bound in
Theorem~\ref{thm:lb} (used with $D=\Omega$).

The assertion (2) in Theorem~\ref{thm:HK-loc} is the conclusion of
Section~\ref{sec:locen}.

The rest of the section is devoted to the proof of  statement  (3)
in Theorem~\ref{thm:HK-loc}. This will be done in three steps.
Recall the definition of the quantity $C_0(\kappa,H;D)$ in
\eqref{eq:gse-D} and that we work under the assumption on  $H$ described in Regimes~I~and~II. It is sufficient to prove that the following formula is true in Regimes~I~and~II,
$$\int_D|\psi(x) |^4\,dx=-\frac2{\kappa^2}C_0(\kappa,H;D)+\frac\kappa
H\,o(1)\,,\quad(\kappa\to\infty)\,,$$ where $(\psi,\Ab)$ is a minimizer of the energy in \eqref{eq-3D-GLf}.

\subsection*{Step 1: The case $D=\Omega$}~\\
A minimizer $(\psi,\Ab)$ satisfies the Ginzburg-Landau equation in \eqref{eq:GL}. Recall the useful identity in \eqref{eq:GLen},
$$-\frac{\kappa^2}2\int_\Omega|\psi(x) |^4\,dx=\mathcal E_0(\psi,\Ab;\Omega)
\,.$$ Thanks to the formulas in Subsection~\ref{sec:locen} used with
$D=\Omega\,$, we observe that \eqref{con74} yields,
\begin{equation}\label{eq:op}
\int_\Omega|\psi(x)|^4\,dx=-\frac2{\kappa^2}C_0(\kappa,H;\Omega)+\frac{\kappa}{H}\,o(1)\,,
\end{equation}
where the formula is valid in Regimes~I~and~II.
\subsection*{Step 2: Upper bound}~\\
Let
$$\ell=\kappa^{-1/4}\sqrt{\frac\kappa H}\quad{\rm and}\quad D_\ell=\{x\in
D~:~{\rm dist}(x,\partial D)\geq \ell\}\,.$$ Consider a cut-off
function $\chi_\ell\in C_c^\infty(D)$ such that,
$$\|\chi_\ell\|_\infty\leq1\,,\quad \|\nabla\chi_\ell\|\leq
\frac{C}{\ell}\,,\quad\chi_\ell =1{~\rm in~}D_\ell\,.$$ Multiplying
both sides of the equation in \eqref{eq:GL} by
$\chi_\ell^2\overline\psi$ then integrating by parts and using the
estimate in Proposition~\ref{thm:ob} yield,
$$\int_D\left(|(\nabla-i\kappa
H\Ab)\chi_\ell\psi|^2-\kappa^2\chi_\ell^2|\psi|^2+\kappa^2\chi_\ell^2|\psi|^4\right)\,dx=\int_D|\nabla\chi_\ell|^2|\psi|^2\,dx=\mathcal
O\Big(\,\frac{C}{\ell^2}\,\frac\kappa
H\,\Big)=\frac{\kappa^3}{H}o(1)\,.$$ Since $1\geq \chi_\ell^2\geq
\chi_\ell^4$, this formula implies,
\begin{equation}\label{eq:op-loc}
-\frac{\kappa^2}{2}\int_D\chi_\ell^2|\psi|^4\,dx\geq \mathcal
E_0(\chi_\ell\psi,\Ab;D)-\frac{\kappa^3}{H}o(1)\,.\end{equation}
Using the bounds $\|\psi\|_\infty\leq 1$ and $\|\psi\|_2\leq
C\sqrt{\frac\kappa H}$\,, the fact that $\chi_\ell$ is supported in
$D$ and $\chi_\ell=1$ in $D_\ell$, we get,
\begin{align}\label{eq:op-loc*}
\int_{D}|\psi(x)|^4\,dx&=\int_D\chi_\ell^2(x)|\psi(x)|^4\,dx+\int_D(1-\chi_\ell^2(x))|\psi(x)|^4\,dx\nonumber\\
&=\int_D\chi_\ell^2(x)|\psi(x)|^4\,dx+\mathcal
O\Big(\,\sqrt{\ell}\,\sqrt{\frac\kappa H}\,\Big)\nonumber\\
&=\int_D\chi_\ell^2(x)|\psi(x)|^4\,dx+\frac\kappa
H\,o(1)\,.\end{align} Now, we infer from \eqref{eq:op-loc} and Theorem~\ref{thm:lb},
\begin{equation}\label{eq:op-up}
\int_{D}|\psi(x)|^4\,dx\leq
-\frac2{\kappa^2}C_0(\kappa,H;D)+\frac{\kappa}{H}\,o(1)\,.
\end{equation}
\subsection*{Step 3: Lower bound}~\\
Notice that \eqref{eq:op-up} is valid for any open domain $D\subset
\Omega$ with a smooth boundary, in particular, it is valid when $D$
is replaced by the complementary of $\overline{D}$ in $\Omega$:
$\overline{D}^c$. We have the simple decomposition,
\begin{align*}
\int_{D}|\psi(x)|^4\,dx&=\int_{\Omega}|\psi(x)|^4\,dx-\int_{\overline{D}^c}|\psi(x)|^4\,dx\\
&\geq
\int_{\Omega}|\psi(x)|^4\,dx-\frac2{\kappa^2}C_0(\kappa,H;\overline{D}^c)+\frac{\kappa}{H}\,o(1)\,.
\end{align*}
Using the asymptotics in \eqref{eq:op} obtained in  Step~1, we
deduce that,
$$\int_{D}|\psi(x)|^4\,dx\geq-\frac2{\kappa^2}C_0(\kappa,H;D)+\frac{\kappa}{H}\,o(1)\,.$$
Combining this lower bound and the upper bound in \eqref{eq:op-up},
we obtain the asymptotics announced in the third assertion of
Theorem~\ref{thm:HK-loc}.

\subsection*{Acknowledgements.} The authors would like to thank S.
Fournais and N. Raymond for useful discussions. B. Helffer is
partially supported by the ANR program NOSEVOL. A. Kachmar is
partially supported by a grant from Lebanese University.

\end{document}